\documentclass{article}

\usepackage[utf8]{inputenc}
\usepackage[a4paper]{geometry}
\usepackage{amsmath, amssymb}
\usepackage{enumitem}
\usepackage{microtype}
\usepackage{bbm}

\usepackage[backend=biber, style=alphabetic, maxbibnames=99, giveninits=true]{biblatex}
\renewbibmacro{in:}{}
\AtEveryBibitem{\clearfield{month}} 
\AtEveryBibitem{\clearfield{day}}
\AtEveryBibitem{\clearfield{isbn}}
\AtEveryBibitem{\clearfield{issn}}
\AtEveryBibitem{\clearlist{language}}
\renewbibmacro*{doi+eprint+url}{%
    \printfield{doi}%
    \newunit\newblock%
    \iftoggle{bbx:eprint}{%
        \usebibmacro{eprint}%
    }{}%
    \newunit\newblock%
    \iffieldundef{doi}{%
        \usebibmacro{url+urldate}}%
        {}%
    }
\usepackage[usenames,dvipsnames]{xcolor}
\definecolor{myred}{RGB}{255, 61, 65}

\usepackage{mathtools}
\usepackage[usenames,dvipsnames]{xcolor}
\definecolor{myred}{RGB}{255, 61, 65}
\usepackage{amsthm}
\usepackage{asymptote}
\usepackage[plainpages=false,pdfpagelabels,colorlinks=true,citecolor=blue,hypertexnames=false]{hyperref}
\usepackage[labelfont={sc}]{caption}

\usepackage[capitalize]{cleveref}
\usepackage{todonotes}

\usepackage[all]{xy}
\usepackage{tikz}
\usetikzlibrary{cd}
\def\ph{\vphantom{$\DiffMod {\Q(x)}$}}
\usepackage[capitalize]{cleveref}
\usepackage{colonequals}
\newcommand*{\longeq}{\ratio\Longleftrightarrow}
\usepackage{multicol}
\usepackage{pifont}

\usetikzlibrary{shapes.geometric, arrows, decorations.pathmorphing, arrows.meta}

\tikzstyle{startstop} = [rectangle, rounded corners, minimum width=3cm, minimum height=1cm,text centered, draw=black, fill=brown!30]
\tikzstyle{decision} = [diamond, aspect=1.5, inner xsep=0pt, text centered, draw=black, fill=lime!30, text width=1.8cm]
\tikzstyle{processyes} = [rectangle, minimum width=3cm, minimum height=1cm, text centered, draw=black, fill=green!30]
\tikzstyle{processno} = [rectangle, minimum width=3cm, minimum height=1cm, text centered, draw=black, fill=red!30]
\tikzstyle{arrow} = [thick,-latex,>=stealth]

\newcommand{\BeHe}[6]{
\begin{tikzpicture}[scale=0.8,cap=round,>=latex]
\pgfmathparse{0};
\edef\shiftvar{\pgfmathresult};

\foreach \n in {1,...,#1}{

    \pgfmathparse{gcd(\n,#1)};
    
    \ifthenelse{1=\pgfmathresult}{
    \draw[thick,xshift=3*\shiftvar cm] (0cm,0cm) circle(1cm);
    \filldraw[myred,xshift=3*\shiftvar cm] (\n*#2*360:1cm) circle(3pt);
    \filldraw[myred,xshift=3*\shiftvar cm] (\n*#3*360:1cm) circle(3pt);
    \filldraw[myred,xshift=3*\shiftvar cm] (\n*#4*360:1cm) circle(3pt);
    \filldraw[blue,xshift=3*\shiftvar cm] (\n*#5*360:1cm) circle(3pt);
    \filldraw[blue,xshift=3*\shiftvar cm] (\n*#6*360:1cm) circle(3pt);
    \filldraw[blue,xshift=3*\shiftvar cm] (\n*1*360:1cm) circle(3pt);
    \node[scale=0.8] at (3*\shiftvar cm, -1.5) { $\lambda = \n$ };
    \pgfmathparse{\shiftvar+1} \xdef\shiftvar{\pgfmathresult}
    }{};
}
\end{tikzpicture}
}

\definecolor{vertex}{HTML}{DDDD55}
\definecolor{vertex2}{HTML}{7DCEA0}
\definecolor{vertex3}{HTML}{85C1E9}

\usepackage{array}

\usepackage{todonotes}

\renewcommand{\i}{\mathbf i}

\renewcommand{\phi}{\varphi}
\renewcommand{\epsilon}{\varepsilon}
\renewcommand{\theta}{\vartheta}
\newcommand{\ua}{\boldsymbol{\alpha}}
\newcommand{\ub}{\boldsymbol{\beta}}
\newcommand{\ug}{\boldsymbol{\gamma}}

\newcommand{\um}{\boldsymbol{\mu}}
\newcommand{\sa}{\underline{a}}
\newcommand{\su}{\underline{u}}
\newcommand{\N}{\mathbb{N}}
\newcommand{\R}{\mathbb{R}}
\newcommand{\C}{\mathbb{C}}
\newcommand{\Z}{\mathbb{Z}}

\newcommand{\Hyp}{\mathcal{H}}
\newcommand{\Zp}{\mathbb{Z}_{(p)}}
\newcommand{\Q}{\mathbb{Q}}
\newcommand{\Qp}{\mathbb{Q}_p}
\newcommand{\F}{\mathbb{F}}

\newcommand{\Fp}{\mathbb{F}_p}
\newcommand{\Fpbar}{\bar{\mathbb{F}}_p}
\renewcommand{\O}{\mathcal{O}}
\newcommand{\Opp}{\O_{(\pp)}}
\newcommand{\Vp}{\mathcal{V}_p}
\newcommand{\D}{\mathfrak{D}}
\newcommand{\Dp}{\mathfrak{D}_p}
\newcommand{\Dpr}{\mathfrak{D}_{p,r}}

\newcommand{\pp}{\mathfrak p}
\newcommand{\qq}{\mathfrak q}
\newcommand{\kpt}{\GL_1(k_\pp)}
\newcommand{\kppt}{\GL_1(k_\pp^+)}

\newcommand{\Mod}{\mathrm{Mod}}
\DeclareMathOperator{\im}{im}
\DeclareMathOperator{\val}{val}
\newcommand{\ps}[1]{[\![#1]\!]}
\DeclareMathOperator{\GL}{GL}
\DeclareMathOperator{\Gal}{Gal}
\newcommand{\Rep}{\mathrm{Rep}}
\newcommand{\et}{\text{\rm ét}}
\newcommand{\sep}{\text{\rm sep}}
\newcommand{\alg}{\text{\rm alg}}
\newcommand{\red}{\text{\rm red}}
\newcommand{\FrobMod}[1]{\Mod^{\phi}_{#1}}
\newcommand{\FrobModp}[1]{\Mod^{p}_{#1}}
\newcommand{\FrobModq}[1]{\Mod^{q}_{#1}}
\newcommand{\FrobModet}[1]{\Mod^{\phi,\et}_{#1}}
\newcommand{\FrobModqet}[1]{\Mod^{q,\et}_{#1}}
\newcommand{\DiffMod}[1]{\Mod^{\nabla}_{#1}}
\newcommand{\Rp}{\mathcal{R}_p}
\newcommand{\Cp}{\mathcal{C}_p}
\newcommand{\ls}[1]{(\!(#1)\!)}
\newcommand{\dpz}[1]{\{ #1 \}_{\!{}_0}}
\newcommand{\dpzbig}[1]{\big\{ #1 \big\}_{\!{}_0}}
\newcommand{\dpc}[1]{\{ #1 \}_{\!{}_1}}

\newcommand{\Zdc}{(\Z/d\Z)^\times}
\def\pFq#1#2#3#4#5{{}_{#1}F_{#2}\left(#3, #4; #5\right)}
\newcommand{\pFqab}{\pFq{n}{n-1}{\ua}{\ub}{x}}
\def\pFqv#1#2{{}_{n}F_{n-1}\left(#1; #2\right)}

\DeclareMathOperator{\id}{id}

\DeclareMathOperator{\image}{image}

\numberwithin{equation}{subsection}

\theoremstyle{definition}
\newtheorem{defi}{Definition}[subsection]
\newtheorem{ex}[defi]{Example}
\newtheorem{quest}[defi]{Question}

\theoremstyle{plain}
\newtheorem{thmintro}{Theorem}[section]
\newtheorem{thm}[defi]{Theorem}
\newtheorem{lem}[defi]{Lemma}
\newtheorem{cor}[defi]{Corollary}
\newtheorem{prop}[defi]{Proposition}
\newtheorem{conj}[defi]{Conjecture}
\newtheorem{heuristic}[defi]{Heuristic observation}

\theoremstyle{remark}
\newtheorem{rem}[defi]{Remark}
\newtheorem{notation}[defi]{Notation}

\title{Galois groups of reductions modulo $p$ of $D$-finite series}
\author{Xavier Caruso, Florian Fürnsinn and Daniel Vargas-Montoya}
\date\today

\begin{document}
\maketitle

\setcounter{tocdepth}{2}
\begin{abstract}
The aim of this paper is to investigate the algebraicity behavior of reductions of $D$-finite power series modulo prime numbers. For many classes of $D$-finite functions, such as diagonals of multivariate algebraic series or hypergeometric functions, it is known that their reductions modulo prime numbers, when defined, are algebraic. We formulate a conjecture that uniformizes the Galois groups of these reductions across different prime numbers.

We then focus on hypergeometric functions, which serves as a test case for our conjecture. Refining the construction of an annihilating polynomial for the reduction of a hypergeometric function  modulo a prime number $p$, we extract information on the respective Galois groups and show that they behave nicely as $p$ varies.
\end{abstract}

\tableofcontents

\section{Introduction} 

\renewcommand{\theequation}{\thesection.\arabic{equation}}

The starting point of this paper is the observation that there exist many interesting series in $\Q\ps{x}$ which are transcendental over $\Q(x)$ 
but whose reductions modulo $p$ (when they exist) are algebraic over $\Fp(x)$ for a large set of primes~$p$.
A typical example is provided by ``diagonals''.
By definition, the \emph{diagonal} of a multivariate series
$$f(t_1, \ldots, t_k) = \sum_{\mathclap{n_1, \ldots, n_k \in \N}} a_{n_1, \ldots, n_k} t_1^{n_1} \cdots t_k^{n_k}$$
is the univariate series
$$\Delta f(x) = \sum_{n \in \N} a_{n,\ldots, n} x^n.$$
In characteristic zero, the diagonal construction usually does not preserve algebraicity.
However this astonishing property does hold in positive characteristic: the diagonal of an algebraic function over~$\Fp(t_1, \ldots, t_k)$ is algebraic over~$\Fp(x)$!
This result was originally observed and proved by Furstenberg~\cite{Fu67} when $f$ is a rational function;
it was then extended by Deligne~\cite{De84} to algebraic functions and reproved many times by different authors~\cite{Sa86,Sa87,DL87,Ha88,SW88} throughout the years.
Deligne's theorem shows that diagonals exhibit the phenomenon that interests us in this article:
the diagonal $\Delta f(x)$ of an algebraic function $f(t_1, \ldots, t_k) \in \Q\ps{t_1, \ldots, t_k}$ is usually transcendental
but its reductions modulo the primes, when they are defined, are algebraic (since they agree with the diagonals of $f \bmod p$).

The community then got interested in finding upper bounds on the degree of algebraicity of $\Delta f(x) \bmod p$.
When $f$ is a rational function, Furstenberg's original proof leads to a bound of the form $p^n$ where $n$ depends only on~$f$.
It turns out that Deligne's proof leads to a similar bound when $f(t_1, t_2)$ is a bivariate algebraic function.
The general case was solved by Adamczewski and Bell~\cite{AB13} who obtained again a general bound of the form $p^n$ where $n$ only depends on~$f$.
In~\cite{ABC23}, after providing sharper bounds on the exponent $n$, the authors observed that, not only $\Delta f(x) \bmod p$ is annihilated by a polynomial $Z_p(Y)$ of degree at most $p^n$,
but that there always exists such an annihilating polynomial of the form
\begin{equation}
\label{eq:Frobenius}
Z_p(Y) = c_0(x) Y + c_1(x) Y^p + \cdots + c_n(x) Y^{p^n} \qquad (c_i(x) \in \Fp[x]).
\end{equation}
They noticed moreover (see Remark~5.3 of \emph{loc.~cit.}) that this property implies that all the Galois conjugates of $\Delta f(x) \bmod p$ lie in a $\Fp(x)$-vector space of dimension $n$,
which further shows that the Galois group of $\Delta f(x) \bmod p$ canonically embeds, up to conjugacy, into $\GL_n(\Fp)$.
Going even further, they suggested that all those Galois groups could have a common origin living in characteristic zero.

In the present article, we explore this question in slightly different contexts.
To start with, we consider the class of $D$-finite functions.
By definition the series $f(x) \in \Q\ps{x}$ is called \emph{$D$-finite} if it is a solution of a linear differential equation of the form
\begin{equation}
\label{eq:Dfinite}
f^{(r)}(x) + a_{r-1}(x) f^{(r-1)}(x) + \cdots + a_1(x) f'(x) + a_0(x) f(x) = 0
\end{equation}
where the coefficients $a_i(x)$ are rational functions.
It is a standard fact that diagonals are $D$-finite but it turns out that the class of $D$-finite series is much larger;
for instance, it includes the function $f(x) = \exp(\arctan(x))$ whose reduction modulo~$p$ only makes sense for primes $p$ which are congruent to $1$ modulo $4$
(whereas diagonals can always be reduced modulo~$p$ for almost all primes~$p$).

In Section~\ref{sec:picture}, we formulate several hypotheses predicting the behavior of the Galois groups of $f(x) \bmod p$ when $f(x)$ is $D$-finite.
Roughly speaking, if $n$ denotes the minimal order of a differential equation satisfied by $f(x)$, 
what we expect is the existence of a number field $K$ and a finite family $G_1, \ldots, G_t$ of subgroups of $\GL_n(K)$ with the following property:
for any prime $\pp$ of $K$ for which $f(x) \bmod \pp$ is well defined and algebraic over $\Fp(x)$, there exists an integer $i \in \{1, \ldots, t\}$ such that
$$\Gal\big(f(x) \bmod p\mid\Fpbar(x)\big) \simeq \image\big(G_i\cap\GL_n(\mathcal O_{(\pp)})\rightarrow\GL_n(k_\pp)\big).$$
Here $\O_{(\pp)}$ denotes the localization of the ring of integers of $K$ at $\pp$, that is the ring consisting of fractions of the form $\frac a b$ with $b \not\in \pp$, 
and $k_\pp$ is the residue field. We refer to Conjecture~\ref{conj:weak} in the body of the text for a more precise statement.

We then elaborate further on this conjecture, trying to predict what could be the field $K$, the groups $G_i$, and the dependence of the index $i$ with respect to~$p$.
This leads us to formulate stronger forms (see Conjectures~\ref{conj:galoisequiv} and~\ref{conj:withL}) which, despite being technical, capture all what we have glimpsed at so far. Finally, in Subsection~\ref{sssec:tannakian}, we investigate a possible relationship between the groups $G_i$ and the differential Galois group of the differential equation~\eqref{eq:Dfinite}.

In Section~\ref{sec:hypergeom}, we provide evidence towards our conjectures by studying in more details the special case of hypergeometric functions.
We recall that, given two tuples $\ua = (\alpha_1, \ldots, \alpha_n)$ and $\ub = (\beta_1, \ldots, \beta_m)$ of rational numbers which are not nonpositive integers, the corresponding \emph{hypergeometric function} is
$$\pFq n m \ua \ub x \coloneqq \sum_{k=0}^\infty \frac{(\alpha_1)_k (\alpha_2)_k \cdots (\alpha_n)_k}{(\beta_1)_k (\beta_2)_k \cdots (\beta_m)_k} \cdot \frac{x^k}{k!}$$
where $(\alpha)_k \coloneqq \alpha \cdot (\alpha+1) \cdots (\alpha + k - 1)$ is the \emph{Pochhammer symbol}.
It is a standard fact that hypergeometric functions are $D$-finite (but not necessarily diagonals).
In Subsection~\ref{ssec:hypergeommodp}, refining a former result of Christol, we classify the primes $p$ for which the series $\pFq n {n-1} \ua \ub x$ reduces properly modulo~$p$.
In Subsection~\ref{ssec:algmodp}, we move to the main interest of this article and, following the strategy introduced in~\cite{Var24}, we exhibit an annihilating polynomial for the function $\pFq n {n-1} \ua \ub x \bmod p$ (when it is defined) and study the corresponding Galois group.
More precisely, we prove the following theorem.

\begin{thmintro}
\label{thintro:hypergeom}
Let $n$ be a positive integer. Let $\ua = (\alpha_1, \ldots, \alpha_n)$ and $\ub = (\beta_1, \ldots, \beta_{n-1})$ be tuples of rational numbers with $\alpha_i, \beta_j \not\in -\N$ for all $i,j$.
Let $d$ be the smallest common denominator of the $\alpha_i$ and $\beta_i$ and let $\Q(\zeta_d)$ be the $d$-th cyclotomic extension of $\Q$.
Let $K$ be the subextension of $\Q(\zeta_d)$ corresponding under Galois correspondence to the group
\begin{multline*}
D \coloneqq
\big\{\,\lambda\in (\Z/d\Z)^\times \,:\, \lambda{\cdot} \ua\equiv \ua \!\!\!\!\pmod \Z \text{ and } \lambda{\cdot} \ub \equiv\ub \!\!\!\!\pmod \Z\text{ as sets}\,\big\} \\
\subseteq (\Z/d\Z)^\times \simeq \Gal\big(\Q(\zeta_d)/\Q\big).
\end{multline*}
Let $p>2d{\cdot}\max\{|\alpha_i|+1, |\beta_j|+1\}$ be a prime number such that $\pFq n {n-1} \ua \ub x \in \Zp\ps{x}$.
Then, letting $k_\pp$ be the residue field of $K$ at a place $\pp$ above $p$ and writing $q \coloneqq \mathrm{Card}(k_\pp)$, the function $\pFq n {n-1} \ua \ub x \bmod p$ is annihilated by a nonzero polynomial of the form
$$Z_p(Y) = c_0(x) Y + c_1(x) Y^q + \cdots + c_n(x) Y^{q^n} \quad \text{with} \quad c_i(x) \in k_\pp(x).$$
In particular, the Galois group of $\pFq n {n-1} \ua \ub x \bmod p$ \emph{over $k_\pp(x)$} embeds into $\GL_n(k_\pp)$.
\end{thmintro}

We notice in particular the emergence of a number field $K$, which is perfectly in line with our formulation of Conjecture~\ref{conj:weak}.
Moreover, we underline that the integer $n$ appearing in $\GL_n(k_\pp)$, say the \emph{embedding dimension}, is truly the same as the integer $n$ we started with,
which is also the order of the minimal-order differential equation satisfied by $\pFq n {n-1} \ua \ub x$. Again, this is in perfect accordance with Conjecture~\ref{conj:weak}.
We mention nevertheless that, in this special case, the embedding dimension can even be lowered and looks more closely related to the dimension of the space of solutions of the hypergeometric equation in characteristic~$p$.

Finally, in Subsection~\ref{ssec:2F1}, we consider the special case of Gaussian hypergeometric functions $\pFq 2 1 \ua {(1)} x$ with $\ua = (\alpha_1, \alpha_2) \in \big(\Q \setminus (-\N)\big)^2$.
In this situation, two important facts occur. 
Firstly, keeping the notation and assumptions of Theorem~\ref{thintro:hypergeom}, the hypergeometric series $\pFq 2 1 \ua {(1)} x$ can be reduced modulo~$p$ for all primes $p$ not dividing $d$.
Secondly, the Galois group of $\pFq 2 1 \ua {(1)} x \bmod p$ over $k_\pp(x)$ always embeds into $\GL_1(k_\pp) = k_\pp^\times$; in particular, it is commutative.
The following theorem exhibits explicit subgroups of $\GL_1(K) = K^\times$, which are serious candidates for uniformizing the above Galois groups when $p$ varies.

\begin{thmintro}
\label{thintro:2F1}
We keep the notation and assumptions of Theorem~\ref{thintro:hypergeom} and assume moreover that $n = 2$, $\alpha_1, \alpha_2 \not\in \Z$ and $\beta_1 = 1$.
We set $K^+ \coloneqq K \cap \R$, we let $m$ be the denominator of an irreducible fraction representing $\alpha_1 + \alpha_2$, and we define the subgroup $G$ of $\GL_1(K)$ by
\[ G\coloneqq 
\begin{cases} 
  \GL_1(K) & \text{if $K = K^+$} \\
  \GL_1(K^+) \cdot \langle \xi - \bar\xi \rangle & \text{if $K \neq K^+$ and $m = 2$} \\
  \GL_1(K^+) \cdot \langle 1 + \zeta_m \rangle & \text{if $K \neq K^+$ and $m \neq 2$}
\end{cases}\]
where $\xi$ is an element of $K\setminus K^+$, $\bar\xi$ is its complex conjugate,
$\zeta_m = \exp(2\i\pi/m)$ is a primitive $m$-th root of unity and the notation $\langle\,\cdot\,\rangle$ stands for the generated subgroup.

Then, for any prime number $p$ which does not divide $d$ and any prime $\pp$ of $K$ above $p$, we have
\[
  \Gal\big(\pFq 2 1 \ua {(1)} x \bmod p \mid k_\pp(x)\big) \subseteq \image\big(G\cap\GL_1(\O_{(p)})\rightarrow\GL_1(k_\pp)\big),
\]
where $\O_{(\pp)}$ is the localization of the ring of integers of $K$ at $\pp$.

Moreover, the above inclusion is an equality when $d \in \{2, 3, 4, 6, 8, 12, 24\}$.
\end{thmintro}

We actually expect that the inclusion of Theorem~\ref{thintro:2F1} is (almost) always an equality but, unfortunately, we were not able to prove it in full generality.
If our expectation is correct, this would say that the group $G$ of Theorem~\ref{thintro:2F1} uniformizes the Galois groups of $\pFq 2 1 \ua {(1)} x \bmod p$.
For this reason, it appears as a very important invariant attached to the hypergeometric function $\pFq 2 1 \ua {(1)} x$.
Nonetheless, it remains quite mysterious to us: it looks like it is built from several pieces of different nature but, so far, we do not have a clear understanding of where these parts come from.

\paragraph{Notation.}

Throughout the article, we let $\Z$, $\Q$, $\R$ and $\C$ denote the set of integers, rational numbers, real numbers and complex numbers respectively.
We also write $\N$ for the set of nonnegative integers, and we denote accordingly the set of nonpositive integers by $-\N$.
Besides, we choose and fix once for all a square root of $-1$ in $\C$, that we denote by the bold letter $\i$.

If $a$ is an element in a ring $A$ and $I$ is an ideal of $A$, we use the notation $a \bmod I$ to denote the image of $a$ in the quotient of $A/I$.
We will often abuse notation and simply write $a \bmod d$ for $a \bmod dA$ when $d$ is an element of $A$.

If $g_1, \ldots, g_n$ are some elements in a group $G$, we write $\langle g_1, \ldots, g_n \rangle$ for the subgroup they generate in $G$.
In particular, if $a$ and $d$ are two coprime integers, the notation $\langle a \bmod d \rangle$ will refer to the \emph{multiplicative} subgroup of $(\Z/d\Z)^\times$ generated by $a$.

\paragraph{Acknowledgements.}
The authors thank Boris Adamczewski, Alin Bostan, Gilles Christol, and Wadim Zudilin for their interest in this work and helpful comments. They also thank Alin Bostan for carefully reading a manuscript version of this text.

The second-named author was funded by the Austrian Science Fund FWF, grant \href{https://doi.org/10.55776/P34765}{10.55776/P34765}, and a DOC Fellowship (27150) of the \href{https://www.oeaw.ac.at/en/}{Austrian Academy of Sciences} at the University of Vienna. Further he thanks \href{https://federation-margaux.math.cnrs.fr/}{Fédération MARGAUx} for supporting a one month research stay at the University of Bordeaux and the French–Austrian project EAGLES (ANR-22-CE91-0007 \& FWF grant \href{https://doi.org/10.55776/I6130}{10.55776/I6130}) for financial support.

The authors thank \href{https://oead.at/en/}{Austria’s Agency for Education and Internationalisation (OeAD)} and Campus France for providing funding for research stays via WTZ collaboration project/Amadeus project FR02/2024.  

\renewcommand{\theequation}{\thesubsection.\arabic{equation}}

\section{A general picture of what we expect}
\label{sec:picture}

In this section, we consider the general class of $D$-finite series $f(x)$ and 
aim at drawing a general picture of the behavior of the Galois groups of $f(x) \bmod p$ (when it is defined) when $p$ varies.
Let us first recall from the introduction that a series $f(x) \in \Q\ps{x}$ is \emph{$D$-finite} (or \emph{differentiably finite}) if it is solution of a linear differential equation of the form
\begin{equation}
\label{eq:Dfinite2}
f^{(r)}(x) + a_{r-1}(x) f^{(r-1)}(x) + \cdots + a_1(x) f'(x) + a_0(x) f(x) = 0
\end{equation}
with $a_i(x) \in \Q(x)$ for all $i$.
It is a standard fact that $D$-finiteness translates to a recurrence relation on the coefficients of $f(x)$.
It is also well known that the class of $D$-finite series is closed under sums and products; in other words, it forms a subring of $\Q\ps{x}$~\cite{Sta80}.

On the Galois side, we introduce the following notation.

\begin{defi}
Let $K$ be a field and let $K^\sep$ be a separable closure of $K$.
For an element $a \in K^\sep$, we set $\Gal(a \mid K) \coloneqq \Gal(L/K)$ where $L$ denotes the extension of $K$ generated by $a$ and its Galois conjugates.
\end{defi}

From now on, we fix a separable closure $\Fp(x)^\sep$ of $\Fp(x)$.
Since the extension $\Fp(\!(x)\!)/\Fp(x)$ is itself separable, any algebraic series $f(x) \in \Fp\ps{x}$ is automatically separable over $\Fp(x)$
and thus can be considered (noncanonically) as an element of $\Fp(x)^\sep$.
It then makes sense to talk about $\Gal\big(f(x)\mid\Fp(x)\big)$ and we check that this Galois group does not depend, up to isomorphism, on the choice of the representative of $f(x)$ in $\Fp(x)^\sep$.

In what follows, if $f(x)$ is a series with coefficients in $\Zp$ and $f(x) \bmod p$ is algebraic, we shall often slightly abuse notation and write $\Gal\big(f(x)\mid\Fp(x)\big)$ for $\Gal\big(f(x) \bmod p\mid\Fp(x)\big)$.

\subsection{Getting intuition by examples}
\label{ssec:examples}

In this introductory subsection, we present and discuss various examples, illustrating the most important phenomena we have observed.
For each example $f(x)$, we will determine the primes~$p$ for which $f(x) \bmod p$ is well-defined and algebraic and compute (or at least estimate) the corresponding Galois groups.

In the next subsections, we will rely on these results to elaborate our conjectures and justify their statements.

\subsubsection{$f(x) = (1-x)^{a/d}$}
\label{sssec:ex:alg}

We begin with a simple example, which is already algebraic over $\Q(x)$.
Before proceeding, we underline that the expansion
$$f(x) = \sum_{k=0}^\infty (-1)^k \cdot \frac a d \cdot \left(\frac a d - 1\right) \cdots \left(\frac a d - k + 1\right) \cdot \frac{x^k}{k!}$$
shows that $f(x)$ is indeed in $\Q\ps{x}$.
We also note that it is also $D$-finite\footnote{In fact, it is a standard result that algebraic functions are always $D$-finite.} since it is solution of the differential equation
$d\cdot(1-x)f'(x) - a f(x) = 0$.

As $f(x)$ is itself algebraic over $\Q(x)$, we start by computing its Galois group over $\Q(x)$. 
For this, assuming that $a$ and $d$ are coprime, we observe that the conjugates of $f(x)$ are the functions $\zeta_d^v f(x)$ ($0 \leq v < d$) where $\zeta_d$ is a primitive $d$-th root of unity, namely $\zeta_d = \exp(\frac{2\i\pi} d) \in \C$.
The corresponding splitting field is $\Q(x)\big(\zeta_d, f(x)\big)$, which can be viewed as a Kummer extension on top of a cyclotomic extension:

\begin{center}
\begin{tikzpicture}[yscale=1.4]
\node (Qx) at (0,0) { $\Q(x)$ };
\node (cyc) at (0,1) { $\Q(x)\big(\zeta_d\big)$ };
\node (kum) at (0,2) { $\Q(x)\big(\zeta_d, f(x)\big)$ };
\draw (Qx)--(cyc)
  node[midway, right] { cyclotomic extension };
\draw (cyc)--(kum)
  node[midway, right] { Kummer extension };
\end{tikzpicture}
\end{center}
The Galois group of the cyclotomic extension is $(\Z/d\Z)^\times$, while that of the Kummer extension is cyclic of order $d$, and hence isomorphic to $\Z/d\Z$.
We deduce from this that
$$\Gal\big(f(x) \mid \Q(x)\big) \simeq \Z/d\Z \rtimes (\Z/d\Z)^\times$$
where $(\Z/d\Z)^\times$ acts on $\Z/d\Z$ by multiplication. 
Concretely, the action of the Galois group on $\Q(x)\big(\zeta_d, f(x)\big)$ is given by $\zeta_d\mapsto \zeta_d^u$, $f(x)\mapsto \zeta_d^v f(x)$ for $u\in (\Z/d\Z)^\times$ and $v\in \Z/d\Z$.

The series $f(x)$ can be properly reduced modulo $p$ for any prime number $p$ which does not divide $d$.
Assuming this, we can compute $\Gal\big(f(x) \mid \Fp(x)\big)$ by applying the same strategy as in characteristic zero.
The only difference is that the Galois group of the cyclotomic part can now be a \emph{strict} subgroup of $(\Z/d\Z)^\times$.
More precisely, it is the same as the Galois group of $\Fp(\xi_d)/\Fp$ where $\xi_d$ is a primitive $d$-th root of unity in an algebraic closure of $\Fp$.
This Galois group being generated by the Frobenius $z \mapsto z^p$, we conclude that $\Gal\big(\Fp(x)(\xi_d)/\Fp(x)\big)$ is the subgroup of $(\Z/d\Z)^\times$ generated by $p$.
Therefore, we finally find
$$\Gal\big(f(x) \mid \Fp(x)\big) \simeq \Z/d\Z \rtimes \langle p \bmod d\rangle.$$
Altogether this shows that there are only finitely many possibilities for $\Gal\big(f(x) \mid \Fp(x)\big)$ when $p$ varies.
More precisely, this Galois group is entirely determined by the class of $p$ modulo $d$.

\subsubsection{$f(x) = \exp(\arctan x)$}
\label{sssec:ex:expatan}

In this example, the function $f(x)$ is transcendental over $\Q(x)$ but still it is the solution of a linear differential equation of order~$1$, given that
\begin{equation}
\label{eq:diffatan}
\frac{f'(x)}{f(x)} = \arctan'(x) = \frac 1 {x^2 + 1}.
\end{equation}
Using Dwork's criterion~\cite[p.409]{Ro00}, one proves that $f(x) \in \Zp\ps{x}$ if and only if $p \equiv 1 \pmod 4$ (see also \cite[Example~3.4]{BCR24}).
For those primes, it turns out that the differential equation~\eqref{eq:diffatan} can be solved over the $p$-adics.

Let $\Z_p$ be the ring of $p$-adic integers and let $i\in \Z_p$ be a square root of $-1$ (which exists because $p$ is congruent to $1$ modulo $4$).
Over $\Z_p(x)$, we have the following partial fraction decomposition:
$$\frac 1{x^2 + 1} = \frac 1 2 \cdot \left( \frac 1 {1 + ix} + \frac 1 {1-ix} \right).$$
Solving Equation~\eqref{eq:diffatan} and taking care of the initial condition $f(0) = 1$, we obtain the following closed formula for $f(x)$:
\begin{equation}
\label{eq:solatan}
f(x) = \left(1 + ix\right)^{-i/2} \cdot \left(1 - ix\right)^{i/2} \in \Z_p\ps{x}.
\end{equation}
Here the $p$-adic exponentiation is defined through its $p$-adic expansion; namely, for $a \in \Z_p$ and a formal expression $u$, we set
$$(1+u)^a \coloneqq \sum_{k=0}^\infty \frac{a\cdot (a-1) \cdots (a-k+1)}{k!} \cdot u^k.$$
We now write the $p$-adic expansions of the exponents $-i/2$ and $i/2$ as follows:
$$\textstyle 
-\frac i 2 = a_0 + p a_1 + p^2 a_2 + \cdots 
\quad \text{and} \quad
\frac i 2 = b_0 + p b_1 + p^2 b_2 + \cdots$$
with $a_k, b_k \in \{0, 1, \ldots, p{-}1\}$.
Reducing Equation~\eqref{eq:solatan} modulo $p$ then gives
$$f(x) \bmod p = \prod_{k=0}^\infty \big(1+ i x^{p^k}\big)^{a_k} \cdot \big(1- i x^{p^k}\big)^{b_k} \in \Fp\ps{x},$$
where $i \equiv 2 b_0\pmod p$. 

\begin{prop}
For all primes $p$ congruent to $1$ modulo $4$, the function $f(x) \bmod p$ is transcendental over $\Fp(x)$.
\end{prop}

\begin{proof}
For an integer $u \in \{0, \ldots, p{-}1\}$, we define $g_u(x) \coloneqq (1+ i x)^u \cdot (1- i x)^{p-1-u}$ and,
for a sequence $\su = (u_k)_{k \geq 0}$, we set
$$g_{\su}(x) \coloneqq \prod_{k=0}^\infty g_{u_k}\big(x^{p^k}\big) = \prod_{k=0}^\infty g_{u_k}(x)^{p^k}.$$
Noticing that $a_0 + b_0 = p$ and $a_k + b_k = p{-}1$ as soon as $k \geq 1$, 
we find $f(x) = (1-i x) \cdot g_{\sa}(x)$ in $\Fp\ps{x}$ where $\sa = (a_k)_{k \geq 0}$.
Hence it is enough to prove that $g_{\sa}(x)$ is transcendental.

For $r \in \{0, \ldots, p{-}1\}$, we introduce the section operator
$$\sigma_r : \Fp\ps{x} \to \Fp\ps{x}, \quad
\sum_{n=0}^\infty a_n x^n \mapsto \sum_{n=0}^\infty a_{pn+r} x^n.$$
The action of $\sigma_r$ on the series $g_{\su}(x)$ is easily described.
Indeed, remarking that $\sigma_r(g(x)^p h(x)) = g(x) \sigma_r(h(x))$ for all series $g(x)$ and $h(x)$, we obtain the relation
$$\sigma_r\big(g_{\su}(x)\big) = \sigma_r\big(g_{u_0}(x)\big) \cdot g_{S \su}(x)$$
where $S \su$ is the shifted sequence defined by $(S \su)_k = u_{k+1}$ for $k \in \N$.
We notice in addition that the prefactor $\sigma_r(g_{u_0}(x))$ is a scalar in $\Fp$ given that $g_{u_0}(x)$ is a polynomial of degree $p{-}1$.
Moreover, when $r = 0$, this scalar is $1$ and we just have $\sigma_0(g_{\su}(x)) = g_{S \su}(x)$.

We now use the following characterization of algebraicity, known as Christol's theorem~\cite{Ch79, CKMR80}:
a series $g(x) \in \Fp\ps{x}$ is algebraic over $\Fp(x)$ if and only if there exists a finite set which contains $g(x)$ and is stable by the $\sigma_r$.
In our case of interest, it is then enough to prove that the series $\sigma_0^k(g_{\sa}(x)) = g_{S^k \sa}(x)$ are pairwise distinct.

Let $\ell > k$ be two nonnegative integers.
We first notice that the sequences $S^k \sa$ and $S^\ell \sa$ cannot be equal.
Indeed, from their coincidence, we would deduce that $\sa$ is ultimately periodic (with period $\ell{-}k$), which is not the case because $i \not\in \Q$.
We now fix an integer $n$ such that restricting to the $n$ first terms is enough to distinguish between $S^k \sa$ and $S^\ell \sa$.
An easy computation shows that
\begin{align*}
g_{S^k \sa}(x) & \equiv \frac 1 {1 - i x} \cdot \left(\frac{1+i x}{1-i x}\right)^{a_k + p a_{k+1} + \cdots + p^{n-1} a_{k+n-1}} \pmod {x^{p^n}} \\
g_{S^\ell \sa}(x) & \equiv \frac 1 {1 - i x} \cdot \left(\frac{1+i x}{1-i x}\right)^{a_\ell + p a_{\ell+1} + \cdots + p^{n-1} a_{\ell+n-1}} \pmod {x^{p^n}}
\end{align*}
Thanks to our choice of $n$, the exponents appearing in the above expressions are distinct.
Observing moreover that the change of variables $y = \frac{1+i x}{1-i x}$ is invertible in $\Fp\ps{x}$ (with inverse $x = i {\cdot} \frac{y-1}{y+1}$),
we conclude that $g_{S^k \sa}(x)$ and $g_{S^\ell \sa}(x)$ are not congruent modulo $x^{p^n}$, and so that the are not equal.
\end{proof}

\begin{rem}
\label{rem:order1}
Examples~\ref{sssec:ex:alg} and~\ref{sssec:ex:expatan} are representative of all what can happen with differential equations of order~$1$.
Indeed, let us consider a series $f(x)$ which is solution of the equation $f'(x) = a(x) f(x)$ for some $a(x) \in \Q(x)$.
If $a(x)$ exhibits a multiple pole (possibly at infinity), then $f(x)$ is transcendental and cannot be reduced modulo $p$ for any prime $p$.
Otherwise we have the following dichotomy:
\begin{itemize}
\item (Case of Example~\ref{sssec:ex:alg}) 
  If the residues of $a(x)$ all lie in $\Q$, then $f(x)$ is algebraic over $\Q(x)$.
  In this case, $f(x)$ can be reduced modulo almost all primes~$p$ and all those reductions are algebraic with Galois groups closely related to the Galois group of $f(x)$ over $\Q(x)$.
\item (Case of Example~\ref{sssec:ex:expatan}) 
  If one residue of $a(x)$ is not in $\Q$, then $f(x)$ is transcendental over $\Q(x)$.
  In this case, $f(x)$ can be reduced modulo $p$ if and only if $a(x) \bmod p$ continues to have only simple poles in $\mathbb P^1(\bar \F_p)$ and all the corresponding residues are in $\Fp$.
  Moreover, when this occurs, $f(x) \bmod p$ is always transcendental.
\end{itemize}
\end{rem}

\begin{rem}
    Adamczewski and Delaygue conjectured (see \cite[Conjecture~1.1]{Var21}) that if a \emph{G-function}\footnote{Roughly speaking, a G-function is a $D$-finite series whose coefficients and their denominators grow at most like a geometric sequence.} $f(x)$ can be reduced modulo infinitely many primes $p$, then these reductions are algebraic for almost all of them. This example shows that the condition on $f(x)$ being a G-function cannot be dropped.
\end{rem}

\subsubsection{$\displaystyle f(x) = \sum_{n=0}^\infty \binom{2n}{n}^{\!2} \cdot x^n$}
\label{sssec:ex:binom2}

Before computing Galois groups, we briefly mention that the series $f(x)$ is indeed $D$-finite as the sequence of its coefficients obviously satisfies a recurrence relation with polynomial coefficients.
Concretely, $f(x)$ is a solution of the differential equation
\begin{equation}
\label{eq:diffbinom2}
x (16x - 1) f''(x) + (32x - 1) f'(x) + 4 f(x) = 0.
\end{equation}
Apart from this, the function $f(x)$ is also one of the simplest examples of a series exhibiting the $p$-Lucas property, which will play a very important role throughout this paper.
To explain what it means, we pick a nonnegative integer $n$ and write the Euclidean division of $n$ by $p$ as follows: $n = u + pv$.
We then have the congruence
$$(1+x)^{2n} \equiv (1+x)^{2u} \cdot (1 + x^p)^{2v} \pmod p,$$
which, after extracting the coefficient of $x^n$, leads to
$$\binom{2n}n \equiv \binom{2u}u \cdot \binom{2v} v \pmod p.$$
Raising this congruence to the square, we deduce that the coefficients $a_n$ of $f(x)$ satisfy the congruence
$a_n \equiv a_u a_v \pmod p$ as well; we say that the series $f(x)$ is \emph{$p$-Lucas}.

The $p$-Lucas property has strong consequences on the Galois groups we are interested in.
Indeed, defining the truncation
$$A_p(x) = \sum_{n=0}^{p-1} a_n x^n \in \Fp[x],$$
being $p$-Lucas translates to the algebraic relation 
$$f(x) \equiv A_p(x) \cdot f(x)^p \pmod p.$$
In particular, $f(x) \bmod p$ is algebraic and, even better, it appears as a $(p{-}1)$-th root of the polynomial $A_p(x)$.
Since $\Fp$ contains all the $(p{-}1)$-th roots of unity, Kummer's theory ensures that 
$\Gal\big(f(x) \mid \Fp(x)\big)$ embeds into $\um_{p-1}(\Fp) = \Fp^\times$ and that equality holds as soon as $A_p(x)$ is a not a power (with exponent at least $2$) of a rational function.

The following lemma will be used many times throughout the rest of the paper. It helps to determine the multiplicity of certain polynomials tied to solutions of differential equations. 

\begin{lem}\label{lem:vals} 
    Let $g(x)\in \Fp\ps{x}$ be a power series, which satisfies a linear differential equation of order $r<p$ with polynomials coefficients.
    Let $\bar \F_p$ denote a fixed algebraic closure of $\F_p$ and let $\xi\in \bar \F_p$ be an ordinary point of that differential equation.
    If $g(x) = A(x)h(x^p)$ for some polynomial $A(x)$ of degree less than $p$ and some power series $h(x)\in \F_p\ps{x}$, then $\val_\xi(A(x))<r$.

    In particular, if $g(x)$ is $p$-Lucas, then $\val_\xi(g(x)\bmod x^p)<r$.
\end{lem}

\begin{proof}
    From $g(x)=A(x)h(x^p)$, it is clear that $A(x)$ is a solution of the same differential equation as $g(x)$. The local exponents of this equation at $\xi$ are $0,1,\ldots, r-1$. Thus, $\val_\xi(A(x))\in \{0, 1, \ldots, r-1\}\bmod p$. Because $\deg(A(x))<p$, we conclude $\val_\xi(A(x))\in \{0, 1, \ldots, r-1\}$.
\end{proof}

Back to our example of the generating function of the squares of the central binomial coefficients, we can use Lemma~\ref{lem:vals} to obtain the following.

\begin{cor}
\label{cor:binom2}
If $p > 2$,
the polynomial $A_p(x)$ cannot be written as $B(x)^e$ with $B(x) \in \Fp(x)$ and $e > 1$.
\end{cor}

\begin{proof}
To start with, we observe that $A_p(x)$ is a polynomial of degree exactly $\frac{p-1} 2$, because, when $\frac p 2 \leq n < p$, the central binomial coefficient $\binom{2n} n$ is divisible by $p$.
Moreover, any element $\xi \in \bar \F_p \setminus \big\{0, \frac 1{16}\big\}$ is an ordinary point of the differential equation~\eqref{eq:diffbinom2}. By Lemma~\ref{lem:vals}, we find that $\xi$ cannot be a double root of $A_p(x)$.
Clearly $A_p(0) = 1$, so $0$ is also not a (double) root of $A_p(x)$.
To conclude, it then remains to exclude the case $A_p(x) = (1 - 16x)^{\frac{p-1} 2}$.
This can be done easily by looking at the coefficient of $x$: in $A_p(x)$, it is equal to $4$, while in the right-hand side, it is equal to $-8(p-1)$, which is $8$ modulo $p$.
\end{proof}

It follows from all what precedes that $\Gal\big(f(x) \mid \Fp(x)\big) \simeq \F_p^\times$ for $p > 2$. 
The isomorphism is actually also correct when $p = 2$ because $f(x) \bmod p = 1$ in this case.

\begin{rem} \label{rem:transcendental}
    Our computations of the Galois group also give a proof that $f(x)$ is transcendental over $\Q(x)$. Indeed, assume it were algebraic of given algebraicity degree $d$. Then also $f(x)\bmod p$ would have algebraicity degree at most $d$, hence the size of its Galois group would be bounded independently of $p$. However, we have just shown that this is not the case. The same argument will also apply to show the transcendence of the next examples we are going to consider.

    Another proof of the transcendence of $f(x)$ can be easily deduced by rewriting it as a hypergeometric function and applying the classification of algebraic hypergeometric functions which we will recall in Theorem~\ref{thm:BH}. 
\end{rem}

\subsubsection{$\displaystyle f(x) = \sum_{n=0}^\infty \binom{2n}{n}^{\!3} \cdot x^n$}
\label{sssec:ex:binom3}

This example looks similar to the previous one (the only difference is that the exponent~$2$ has been replaced by~$3$) 
but we will see that the behavior of the Galois groups exhibits a new interesting phenomenon.

The function $f(x)$ is still $D$-finite and is now a solution of a differential equation of order $3$:
$$x^2 (64x - 1) f'''(x) + x (288x - 3) f''(x) + (208x - 1) f'(x) + 8 f(x) = 0.$$
Of course, $f(x)$ is again $p$-Lucas, from which we get the algebraic relation
$$f(x) \equiv A_p(x) \cdot f(x)^p \pmod p,$$
where $A_p(x)$ is the truncation of $f(x)$ defined by
$$A_p(x) = \sum_{n=0}^{p-1} \binom{2n}{n}^{\!3} \cdot x^n \in \Fp[x].$$
Again, this implies that $\Gal\big(f(x) \mid \Fp(x)\big)$ embeds into $\F_p^\times$ but it might happen now that this inclusion is strict.

\begin{lem} \label{lem:Clausen}
Let $p$ be prime number with $p \equiv 1 \pmod 4$.
Then $A_p(x)$ is a square in $\Fp[x]$.
\end{lem}

\begin{proof}
We reinterpret the function $f(x)$ as an hypergeometric series as follows:
$$\textstyle f(x) = \pFq 3 2 {\big(\frac 1 2, \frac 1 2, \frac 1 2\big)} {(1,1)} {64 x}.$$
It then follows from Clausen's formula~\cite{Cl28} that $f(x) = g(x)^2$ with
$\textstyle g(x) = \pFq 2 1 {\big(\frac 1 4, \frac 1 4\big)} {(1)} {64 x}$.
We now let $B_p(x)$ denote the truncation at $x^p$ of $g(x) \bmod p$, so that
\begin{equation}
\label{eq:ApBp2}
A_p(x) \equiv B_p(x)^2 \pmod{x^p}.
\end{equation}
As in the proof of Corollary~\ref{cor:binom2}, we find that $A_p(x)$ has degree $\frac{p-1}2$.
Similarly, using the congruence $p \equiv 1 \pmod 4$, we prove that $\deg B_p(x) = \frac{p-1}4$.
Hence the congruence~\eqref{eq:ApBp2} has to be an equality in $\Fp[x]$, which proves the lemma.
\end{proof}

By Lemma~\ref{lem:vals}, we have that $A_p(x)$ is not an $e$-th power for any $e > 2$.
Moreover, when $p \equiv 3 \pmod 4$, it is not a square either because it has odd degree.
We conclude that
\begin{equation} \label{eq:ex:binom3} \Gal\big(f(x) \mid \Fp(x)\big) \simeq
\begin{cases}
(\F_p^\times)^\square & \text{if } p \equiv 1 \pmod 4 \\
\F_p^\times & \text{if } p \equiv 3 \pmod 4
\end{cases},
\end{equation}
where, when $G$ is a commutative group, the notation $G^\square$ refers to its subgroup of squares.

\subsubsection{$\displaystyle f(x) = \sum_{n=0}^\infty \sum_{k=0}^n \binom{n+k}{k}^{\!2} \binom{n}{k}^{\!2} \cdot x^n$ and related functions}
\label{sssec:ex:apery}

The series $f(x)$ is the so-called \emph{Apéry series} because the sequence of its coefficients appears as a key ingredient in Apéry's proof of the irrationality of $\zeta(3)$~\cite{Ap79,Fi04}.
It is solution of the differential equation
\begin{equation}
\label{eq:diffapery}
x^2 (x^2 - 34x + 1) f'''(x) + 3x (2x^2 - 51x + 1) f''(x) + (7x^2 - 112x + 1) f'(x) + (x-5) f(x) = 0
\end{equation}
and it turns out that the reductions of $f(x)$ modulo the primes behave quite similarly to what we described in the previous example.
First of all, it is known that $f(x)$ is $p$-Lucas~\cite{Ge82}. Therefore, once again, its reductions modulo $p$ satisfy the algebraic relation
$$f(x) \equiv A_p(x) \cdot f(x)^p \pmod p$$
where $A_p(x)$ is the truncation of $f(x)$:
$$A_p(x) = \sum_{n=0}^{p-1} \sum_{k=0}^n \binom{n+k}{k}^{\!2} \binom{n}{k}^{\!2} \cdot x^n \in \Fp[x].$$
As a consequence the Galois group $\Gal\big(f(x) \mid \Fp(x)\big)$ appears as a subgroup of $\F_p^\times$.
Moreover since the order of the differential equation~\eqref{eq:diffapery} is $3$, Lemma~\ref{lem:vals} shows that $A_p(x)$ cannot be an $e$-th power with $e > 2$ (except if it is a constant polynomial).
However it still can be a square and, as we will see in the sequel, this actually happens from time to time.

Computational experiments suggest that whether $A_p(x)$ is a square or not is governed by the congruence class of $p$ modulo $24$.
More precisely, we observed the following.

\begin{heuristic}
\label{heu:Ap}
There exists a polynomial $B_p(x) \in \Fp[x]$ such that
\begin{itemize}
\item $A_p(x) = B_p(x)^2$ if $p \equiv 1, 5, 7, 11 \pmod{24}$,
\item $A_p(x) = (x^2 - 34x + 1) \cdot B_p(x)^2$ if $p \equiv 13, 17, 19, 23 \pmod{24}$.
\end{itemize}
\end{heuristic}

\noindent
We leave the above heuristic as an open question for the time being and explore its consequences.
First of all, at the level of Galois groups, it implies
$$\Gal\big(f(x) \mid \Fp(x)\big) \simeq
\begin{cases}
(\F_p^\times)^\square & \text{if } p \equiv 1, 5, 7, 11 \pmod{24} \\
\F_p^\times & \text{if } p \equiv 13, 17, 19, 23 \pmod{24}
\end{cases}$$
given that the polynomial $x^2 - 34x + 1$ is not a square in $\Fp[x]$ as soon as $p \geq 5$.

Going further, Heuristic~\ref{heu:Ap} suggests to introduce the following functions:
$$g(x) = \sqrt{f(x)} \quad \text{and} \quad h(x) = \sqrt{\frac{f(x)}{x^2 - 34x + 1}}.$$
Those can be expanded as series in the variable $x$, and so they are elements of $\Q\ps{x}$.
Moreover, it turns out that they are both $D$-finite, being solutions of the differential equations
\begin{align*}
4x(x^2 - 34x + 1) g''(x) + 4(2x^2 - 51x + 1) g'(x) + (x - 10) g(x) & = 0, \\
4x(x^2 - 34x + 1) h''(x) + 4(4x^2 - 85x + 1) h'(x) + (9x - 78) h(x) & = 0.
\end{align*}
Of course, $g(x) \bmod p$ and $h(x) \bmod p$ are also algebraic and one can be interested in their Galois groups.
When $A_p(x) = B_p(x)^2$, we have
$$g(x) = \frac{g(x)^p}{\sqrt{f(x)^{p-1}}} \equiv B_p(x)\cdot g(x)^p \pmod p$$
and similarly
$$h(x) \equiv (x^2 - 34x +1)^{\makebox{${}^\frac{p{-}1}2$}} B_p(x) \cdot h(x)^p \pmod p.$$
In this case, we see that the series $g(x) \bmod p$ and $h(x) \bmod p$ satisfy algebraic relations of the same shape and,
using that $B_p(x)$ is a not an $e$-th power for any $e>1$ (because $A_p(x)$ has no ordinary points as roots of multiplicity $2e$), we conclude that
$$\Gal\big(g(x) \mid \Fp(x)\big) \simeq \Gal\big(h(x) \mid \Fp(x)\big) \simeq \Fp^\times.$$
On the contrary, when $A_p(x) = (x^2 - 34x +1) \cdot B_p(x)^2$, we get the relations
\begin{align}
g(x) & \equiv (x^2 - 34x +1)^{\makebox{${}^\frac{p{+}1}2$}} B_p(x) \cdot h(x)^p \pmod p \label{eq:algapery:start} \\
h(x) & \equiv B_p(x) \cdot g(x)^p \pmod p
\end{align}
which now intertwine the functions $g(x)$ and $h(x)$.
Combining these equations, we deduce the untangled algebraic relations
\begin{align}
g(x) & \equiv (x^2 - 34x +1)^{\makebox{${}^\frac{p{+}1}2$}} B_p(x)^{p+1} \cdot g(x)^{p^2} \pmod p \\
h(x) & \equiv (x^2 - 34x +1)^{\makebox{${}^\frac{p(p{+}1)}2$}} B_p(x)^{p+1} \cdot h(x)^{p^2} \pmod p. \label{eq:algapery:end}
\end{align}
We observe that the coefficients of $g(x)^{p^2}$ and of $h(x)^{p^2}$ on the right-hand sides of these equations both are $\frac{p+1}{2}$-th powers, but not higher ones. So, at the level of Galois groups, we obtain
$$\Gal\big(g(x) \mid \Fp(x)\big) \simeq \Gal\big(h(x) \mid \Fp(x)\big) \simeq \sqrt{\Fp^{\times}} \rtimes \Z/2\Z,$$
where by definition 
$$\sqrt{\Fp^{\times}} \coloneqq \Big\{ \, x \in \F_{p^2}^\times \,\,\text{such that}\,\, x^2 \in \Fp^\times \, \Big\}$$
and where $\Z/2\Z$ acts on it through the Frobenius.

\begin{rem}
In these examples, we observe that both of the Galois groups $\Gal\big(g(x)\mid\Fp(x)\big)$ and $\Gal\big(h(x) \mid \Fp(x)\big)$ might be commutative for infinitely many primes~$p$ and noncommutative for another infinity of primes~$p$ at the same time (at least if we believe in Heuristic~\ref{heu:Ap}).
\end{rem}

\begin{rem}
We thank Alin Bostan for pointing out to us that similar phenomena occur with several other series of interest.
Here are two examples:
\begin{itemize}
\item The generating series of Domb numbers~\cite{Do60}
$$f(x) = \sum_{n=0}^\infty \sum_{k=0}^n \binom{2k} k \binom{2n-2k}{n-k} \binom n k ^{\!2} \cdot x^n.$$
In this case, computations suggest that the Galois group $\Gal\big(f(x) \mid \Fp(x)\big)$ is $\Fp^\times$ when $p \equiv -1 \pmod 6$ and its subgroup of squares when $p \equiv 1 \pmod 6$.
\item The generating series of Almkvist--Zudilin numbers~\cite{Zag09, AZ06} 
$$f(x) = \sum_{n=0}^\infty \sum_{k=0}^n (-1)^{n-k} 3^{n-3k} \frac{(3k)!}{k!^3} \binom n {3k} \binom{n+k} k \cdot x^n.$$
In this case, our computations tend to imply that the Galois group $\Gal\big(f(x) \mid \Fp(x)\big)$ is $\Fp^\times$ when $p \equiv 5, 7 \pmod 8$ and its subgroup of squares when $p \equiv 1, 3 \pmod 8$.
\end{itemize}
These examples, together with the Apéry numbers, are appearing as sporadic examples of differential equations having solutions with integral coefficients, studied for example in \cite{Beu02,Zag09, AVZ11}. Especially \cite[Theorem~4.1]{AVZ11} is reminiscent of our argument involving Clausen's formula in the proof of Lemma~\ref{lem:Clausen}. It expresses the Apéry series (where the variable is substituted by a rational function) as the square of another, explicitly determined, power series times a polynomial, and could provide direction in the quest to prove our Observation~\ref{heu:Ap}.
\end{rem}

\subsubsection{$f(x) = \pFq 3 2 {\big(\frac 1 9, \frac 4 9, \frac 5 9\big)} {\big(\frac 1 3, 1\big)} x$}
\label{sssec:ex:3F2}

We conclude our tour by an example of a hypergeometric function, anticipating on the forthcoming results of Section~\ref{sec:hypergeom}.
Here, the series $f(x)$ is a solution of the differential equation
$$729x^2(x - 1) f'''(x) + 81 x (37 x - 21) f''(x) + 9 (200x-27) f'(x) + 20 f(x) = 0.$$
In Section \ref{sec:hypergeom} (see Example \ref{ex:3F2}), we will show that $f(x)$ can be reduced modulo any prime $p \neq 3$ and that the Galois group of $f(x) \bmod p$ satisfies
\begin{equation} \label{eq:GalGrpsEx}\Gal\big(f(x)\mid\Fp(x)\big) \subset
\begin{cases}
\GL_2(\Fp) & \text{if } p \equiv 1 \pmod 9 \\
\F_{p^6}^\times \rtimes \Z/6\Z & \text{if } p \equiv 2 \text{ or } 5 \pmod 9 \\
\GL_2(\F_{p^3}) \rtimes \Z/3\Z & \text{if } p \equiv 4 \text{ or } 7 \pmod 9 \\
\F_{p^2}^\times \rtimes \Z/2\Z & \text{if } p \equiv 8 \pmod 9 
\end{cases}
\end{equation}
where the element $1 \in \Z/d\Z$ ($d \in \{2, 3, 6\}$) always acts through the Frobenius on the left-hand factor of the semi-direct product.
We actually expect that all these inclusions are equalities but were not able to prove it in full generality.

This example is novel in the sense that, for the first time, it goes beyond cyclic groups (or simple semi-direct products of them) and involves groups of matrices of higher rank.
We see moreover that the latter does not show up in each case: for some classes of congruences, the expected Galois group remains resoluble while it suddenly becomes more complicated for other congruences classes.

\subsection{The conjectures: weak and strong forms}

The aim of this subsection is to formulate plausible conjectures describing the behavior of $\Gal\big(f(x)\mid\Fp(x)\big)$ when $p$ varies.

\subsubsection{Discarding too naive expectations}

Since $\Gal\big(f(x)\mid\Fp(x)\big)$ always appears as a subgroup of $\GL_n(\Fp)$ for some integer $n$ (depending possibly on $p$),
the first naive question we may wonder about is the following: does there exist a subgroup of $\GL_n(\Z)$ whose reduction modulo $p$ agrees with $\Gal\big(f(x)\mid\Fp(x)\big)$ for (almost) all primes~$p$?
Unfortunately, the answer is negative, even if we replace $\Z$ by $\Z[\frac 1 N]$ for a positive integer $N$.
Indeed, Subsection~\ref{sssec:ex:apery} exhibits a situation where $\Gal(g(x)\mid\Fp(x))$ is commutative for half of the primes and noncommutative for the other half;
the next lemma shows that this cannot occur for the reductions modulo $p$ of a single subgroup of $\GL_n\big(\Z[\frac 1 N]\big)$.

\begin{lem}
\label{lem:redcommutative}
Let $G$ be a subgroup of $\GL_n\big(\Z[\frac 1 N]\big)$ for some positive integers $n$ and $N$.
We assume that, for infinitely many primes~$p$, the image of $G$ in $\GL_n(\Fp)$ is commutative.
Then $G$ is commutative.
\end{lem}

\begin{proof}
We pick $A, B \in G$ and form the commutator $C = A B A^{-1} B^{-1}$.
For any prime $p$ not dividing~$N$, the matrix $C \bmod p$ is the commutator of $A \bmod p$ and $B \bmod p$.
Thanks to our assumption, these commutators are trivial for infinitely many such primes~$p$.
Hence, denoting by $I_n$ the identity matrix, we get $C \equiv I_n \pmod p$ for infinitely many primes~$p$.
Therefore $C = I_n$ and we conclude that $A$ and $B$ commute in $G$.
\end{proof}

\begin{rem}
One may argue that the issue we just raised with the function $g(x)$, defined as the square root of the Apéry series in Example~\ref{sssec:ex:apery}, can be easily fixed by considering Galois groups over $\F_{p^2}(x)$ instead of $\Fp(x)$:
doing this, the factors $\Z/2\Z$ disappear and all the Galois groups $\Gal(g(x)\mid\F_{p^2}(x))$ are commutative. This is indeed correct.
However, Example~\ref{sssec:ex:3F2} shows an even worse situation where half of the groups are resoluble, while the other half of them are not; in this case, the issue cannot be solved by changing the base field.
\end{rem}

If one unique subgroup of $\GL_n\big(\Z[\frac 1 N]\big)$ is not enough to uniformize, one may now wonder if a finite family of them could do the job.
Unfortunately, this is again not the case: a counterexample is given by Example~\ref{sssec:ex:binom2}, as shown by the next proposition.

\begin{prop}
Let $t, n$, and $N$ be positive integers and let $G_1, \ldots, G_t$ be subgroups of $\GL_n\!\big(\Z[\frac 1 N]\big)$.
Then, there exist infinitely many prime numbers~$p$, not dividing $N$, for which none of the groups $G_1 \bmod p, \ldots, G_t \bmod p$ are cyclic of order $p{-}1$.
\end{prop}

\begin{proof}
For $i \in \{1, \ldots, t\}$, we let $\mathcal P_i$ denote the set of primes~$p$, not dividing $N$, for which $G_i \mod p$ is cyclic of order $p{-}1$.
We also set $\mathcal P \coloneqq \mathcal P_1 \cup \cdots \cup \mathcal P_t$ and denote by $\mathcal P'$ the complement of $\mathcal P$ in the set of all primes.
We have to show that $\mathcal P'$ is infinite.

Without loss of generality, we may assume that each $\mathcal P_i$ is infinite.
Indeed, when $\mathcal P_i$ is finite, the lemma for the family $(G_1, \ldots, G_{i-1}, G_{i+1}, \ldots, G_t)$ implies the lemma for $(G_1, \ldots, G_t)$.

We fix an index $i$. We claim that all matrices in $G_i$ are semi-simple, \emph{i.e.}, diagonalizable over an algebraic closure $\bar \Q$.
Indeed, pick $M \in G_i$ and a change-of-basis matrix $A \in \GL_n(\bar\Q)$ such that $T \coloneqq A^{-1} M A$ is in Jordan form.
We notice that $A$ has all its entries in a certain number field~$K$. Besides, $A$ and $M$ can be reduced modulo~$p$ for almost all prime numbers~$p$.
We assume by contradiction that $T$ has a nontrivial Jordan block of the form
$$T_0 \coloneqq \left(\begin{matrix}
\lambda & 1 \\
& \ddots & \ddots \\
& & \lambda & 1 \\
& & & \lambda
\end{matrix}\right).$$
For almost all primes $p \in \mathcal P_i$, we have
$$T_0^{p-1} \equiv \left(\begin{matrix}
\lambda^{p-1} & -\lambda^{p-2} & \star & \star \\
& \ddots & \ddots & \star \\
& & \lambda^{p-1} & -\lambda^{p-2} \\
& & & \lambda^{p-1}
\end{matrix}\right) \pmod p.$$
From the fact that $G_i \bmod p$ is cyclic of order $p{-}1$, we deduce that $\lambda$ must vanish modulo $p$. Since this holds for infinitely many primes~$p$, we finally obtain $\lambda = 0$, which is a contradiction.

On the other hand, we deduce from Lemma~\ref{lem:redcommutative} that $G_i$ is commutative.
Consequently, all matrices in $G_i$ are simultaneously diagonalizable,
\emph{i.e.}, there exists a matrix $A \in \GL_n(K_i)$ (for a certain number field~$K_i$) such that $A^{-1} M A$ is diagonal for all $M \in G_i$.
Up to enlarging $N$, we may further assume that $A \in \GL_n\!\big(\O_{K_i}[\frac 1 N]\big)$ where $\O_{K_i}$ denotes the ring of integers of $K_i$.
Let $H_i$ be the subgroup of $\O_{K_i}[\frac 1 N]^\times$ generated by all the eigenvalues of all the matrices $M \in G_i$.
By Dirichlet's theorem, we know that $\O_{K_i}[\frac 1 N]^\times$ is of finite type. Hence $H_i$ is also and we can choose a finite set of generators $S_i$ of $H_i$.

We now uniformize this construction over~$i$: we let $K$ be the number field generated by all the $K_i$ and set $S \coloneqq S_1 \cup \cdots \cup S_t$.
By Chebotarev's density theorem applied to the number field $L = K(\sqrt s, s \in S)$, 
there exist infinitely many primes~$p$ such that $p$ splits in $K$ and $s \bmod \pp$ is a square in $\Fp$ for all $s \in S$ and all primes $\pp$ of $K$ above $p$.
For those primes, all groups $G_i$ reduce modulo~$p$ to a subgroup of $(\Fp^\times)^\square$ and hence cannot have order $p{-}1$.
\end{proof}

\subsubsection{First formulation of the conjecture}

The previous discussion leaves open the possibility that the Galois groups $\Gal\big(f(x)\mid\Fp(x)\big)$ are uniformized by a finite family $(G_1, \ldots, G_t)$ of subgroups of $\GL_n\!(\Q)$ in the following sense:
for each prime $p$, there exists an index $i$ such that
$$\Gal\big(f(x)\mid\Fp(x)\big)\simeq \image\big (G_i\cap \GL_n(\Zp) \to \GL_n(\Fp)\big),$$
where we recall that $\Zp$ denotes the localization of $\Z$ at $p$.
At least, this formulation clearly allows for catching $\Fp^\times = \GL_1(\Fp)$ for all~$p$, as we can simply take $G = \GL_1(\Q)$.

Although this expectation looks rather reasonable, many examples encountered in Subsection~\ref{ssec:examples} involve the multiplicative group of some finite extension of $\Fp$.
It then sounds natural to look for the groups $G_i$, not inside $\GL_n(\Q)$, but rather inside $\GL_n(K)$ where $K$ is a number field.

In a similar fashion, we observe that most of the Galois groups computed in Subsection~\ref{ssec:examples} are not exactly subgroups of some $\GL_n$ but rather
semi-direct products of groups of this type with cyclic groups, the latter corresponding to an unramified initial extension of the form $\F_q(x)/\Fp(x)$.
This phenomenon is actually a general fact about splitting fields of series in $\Fp\ps x$ and can be isolated.

\begin{defi}
The \emph{residual splitting field} of an algebraic series $f(x) \in \Fp\ps{x}$ is $L \cap \Fpbar,$ where $L \subseteq \Fp(x)^\sep$ is the splitting field of $f(x)$.
\end{defi}

\begin{lem}
\label{lem:splitting}
Let $f(x) \in \Fp\ps{x}$ be an algebraic series and let $\ell$ be its residual splitting field.
For any subfield $k$ of $\Fpbar$, we have an exact sequence
\begin{equation}
\label{eq:splitting}
1 \longrightarrow \Gal \big(f(x)\mid k(x)\big) \longrightarrow \Gal \big(f(x)\mid\Fp(x)\big) \longrightarrow \Gal\big(k \cap \ell/\Fp\big) \longrightarrow 1.
\end{equation}
\end{lem}

\begin{proof}
For a subfield $k$ of $\Fpbar$, we let $S_k \subseteq \Fp(x)^\sep$ be the splitting field of $f(x)$ over $k(x)$.
By definition, $\Gal(f(x)\mid k(x))$ is then the Galois group of $S_k$ over $k(x)$.
For convenience, we also set $S \coloneqq S_{\Fp}$.

We consider the field $\Fp(x)[f]$. It is a subfield of $\Fp(\!(x)\!)$ and so it is linearly disjoint from~$k$.
As a consequence, the minimal polynomial of $f(x)$ over $\Fp(x)$ remains irreducible over $k(x)$.
Therefore, all the Galois conjugates of $f(x)$ over $\Fp(x)$ are also Galois conjugates over $k(x)$.
It follows that $S_k$ is the compositum of $S$ and $k(x)$, \emph{i.e.}, $S_k = S \cdot k(x)$ as subfields of $\Fp(x)^\sep$.
Noticing moreover that $S \cap k(x) = (k \cap \ell)(x)$, the exact sequence~\eqref{eq:splitting} follows from classical Galois theory.
\end{proof}

It follows from Lemma~\ref{lem:splitting} that $\Gal (f(x)\mid k(x))$ is independent of $k$ as soon as $k$ contains the residual splitting field of $f(x)$.
This is the part of the Galois group we will be mostly interested in and which is depicted by our conjecture.

\begin{conj}
\label{conj:weak}
Let $f(x)\in \Q\ps{x}$ be a $D$-finite series and let $n$ be the minimal order of a differential operator that annihilates $f(x)$.\\
Then, there exist a number field $K$ and a family $(G_1, \ldots, G_t)$ of subgroups of $\GL_n(K)$ such that the following holds:
for almost all rational primes $p$ for which $f(x) \bmod p$ is well defined and algebraic over $\Fp(x),$ and for any prime $\pp$ of $K$ above $p$,
there exists an index $i \in \{1, \ldots, t\}$ such that
\begin{enumerate}[label=(\alph{enumi})]
\item \label{item:conj:rsf} the residual splitting field of $f(x) \bmod p$ is contained in $k_\pp$, and
\item \label{item:conj:iso} there exists an (abstract) isomorphism of groups
$$\Gal\big(f(x)\mid k_\pp(x)\big)\simeq \image\big (G_i\cap \GL_n(\mathcal{O}_{(\pp)}) \to \GL_n(k_\pp)\big),$$
\end{enumerate}
where, in what precedes, $\mathcal{O}_{(\pp)}$ denotes the localization of the ring of integers of $K$ at $\pp$ and $k_\pp$ is the residue field of $K$ at $\pp$.
\end{conj}

We briefly comment Conjecture~\ref{conj:weak} on several important points.

First, we remark that the assertion~\ref{item:conj:rsf} above indicates that the Galois group $\Gal(f(x)\mid k_\pp(x))$ has attained the limit, 
in the sense that it is equal to $\Gal(f(x)\mid k(x))$ for any extension $k$ of $k_\pp$ sitting inside $\Fpbar$.
We deduce from this observation that, if Conjecture~\ref{conj:weak} holds with the number field $K$, then it holds as well for any finite extension $K'$ of $K$ with the same family of groups $(G_1, \ldots, G_t)$.

Second, we underline that the $n$ in $\GL_n(K)$, \emph{i.e.}, the size of matrices, is expected to be controlled by the order of a differential operator $\mathcal L$ with $\mathcal L(f(x)) = 0$.
In fact, more is expected: for a given prime $p$, we actually believe that $\Gal\big(f(x)\mid k_\pp(x)\big)$ embeds into $\GL_{n_p}(k_\pp)$ where $n_p$ is the $\Fp(x^p)$-dimension of the space of solutions of $\mathcal L \bmod p$
(observe that $n_p \leq n$).
We will come back to these dimension bounds and give more evidences and arguments supporting them in Subsection~\ref{ssec:category}.

Finally, we mention that it is still unclear to us what could be the (minimal) number field $K$ in general, although we expect it to be somehow related to the differential Galois group (or maybe the monodromy group) of the minimal differential equation satisfied by $f(x)$. Again, we will give more evidences supporting this hope in Subsection~\ref{ssec:category}.

\paragraph{Back to the examples.}

We now come back to the examples of Subsection~\ref{ssec:examples} and show that Conjecture~\ref{conj:weak} perfectly fits with all of them.

In Example~\ref{sssec:ex:alg}, the part of the Galois group that the conjecture is concerned with is the factor $\Z/d\Z$, which is always present regardless of the choice of the prime~$p$.
Since this factor corresponds to the Galois group over $\Fp(x)(\zeta_d)$, the natural choice for the number field $K$ is the cyclotomic extension $K = \Q(\zeta_d)$.
For any prime $\pp$ of $K$, we have $\Gal(f(x)\mid k_\pp(x))\simeq \Z/d\Z$ and we can simply take $t = 1$ and $G_1 = \um_d(K) \subseteq \GL_1(K)$.

In Example~\ref{sssec:ex:expatan}, there is nothing to comment on, since $f(x) \bmod p$ is never algebraic.

In Example~\ref{sssec:ex:binom2}, we can simply take $t = 1$ and $G = \Q^\times = \GL_1(\Q)$.

In Example~\ref{sssec:ex:binom3}, one option is to take $t = 2,$ as well as $G_1 = \GL_1(\Q)^\square$, and $G_2 = \GL_1(\Q)$:
when $p \equiv 1 \pmod 4$, we choose the index $i = 1$ whereas, when $p \equiv 3 \pmod 4$, we choose $i = 2$ (and we choose either $i = 1$ or $i = 2$ when $p = 2$).
It is however a nice observation that the Galois groups can actually be uniformized by a unique subgroup $G \subseteq \GL_1(\Q)$, namely
$$G \coloneqq \GL_1(\Q)^\square \cdot \{\pm 1\} \subseteq \GL_1(\Q).$$
Indeed, one checks that $G$ reduces modulo $p$ to $\GL_1(\Fp)^\square \cdot \{\pm 1\}$ and, now, we remember that $-1$ is a square modulo $p$ if and only if $p = 2$ or $p \equiv 1 \pmod 4$.

The Apéry series $f(x)$ of Example~\ref{sssec:ex:apery} is quite similar to what precedes:
either, we can take $G_1 = \GL_1(\Q)^\square$ and $G_2 = \GL_1(\Q)$ or we can pick just one group which is
$$G \coloneqq \GL_1(\Q)^\square \cdot \langle -6 \rangle \subseteq \GL_1(\Q).$$
Indeed, it follows from the law of quadratic reciprocity that $-6$ is a square modulo a prime number $p > 3$ if and only if $p \equiv 1, 5, 7, 11 \pmod{24}$.

For the series $g(x)$ and $h(x)$, we observe that the relevant part of $\Gal(g(x)\mid\Fp(x)) \simeq \Gal(h(x)\mid\Fp(x))$ is, in some sense, the ``square root'' of $\Gal(f(x)\mid\Fp(x))$.
This observation leads us to consider the number field $K = \Q(\i\sqrt{6})$ together with the group
$$G \coloneqq \GL_1(\Q) \cdot \big\langle \i\sqrt{6}\big\rangle \subseteq \GL_1(K).$$
When $p$ is $1$, $5$, $7$ or $11$ modulo $24$, there are two places $\pp_1$ and $\pp_2$ above $p$ with residue field $\Fp$.
Moreover, in this case, the reduction of $\i\sqrt{6}$ in the residue field lies obviously in $\Fp$, and so
$$\image\big (G\cap \GL_1(\mathcal{O}_{(\pp_i)}) \to \GL_1(\Fp)\big) = \GL_1(\Fp) = \Fp^\times.$$
On the contrary, when $p \equiv 13, 17, 19, 23 \pmod {24}$, there is only one place $\pp$ above $p$ with residue field $\F_{p^2}$ and we find
$$\image\big (G\cap \GL_1(\mathcal{O}_{(\pp)}) \to \GL_1(\F_{p^2})\big) = \sqrt{\Fp^\times},$$
which is also the Galois group of $g(x)$ (resp. $h(x)$) over $\F_{p^2}(x)$.

In Example~\ref{sssec:ex:3F2} finally, getting inspired by what precedes, we are looking for a number field for which the splitting properties of a rational prime~$p$ depends on its
congruence class modulo~$9$. A natural candidate is the cyclotomic extension $K = \Q(\zeta_9)$. Regarding the groups, one checks that we can take
$$\begin{array}{r@{\hspace{0.5ex}}ll}
G_1 & \coloneqq \GL_1(K) &
  \text{(works for $p \equiv 2 \!\!\!\pmod 3$),} \smallskip \\
G_2 & \coloneqq \GL_2(K) &
  \text{(works for $p \equiv 1 \!\!\!\pmod 3$).}
\end{array}$$

\begin{rem}
    In the Conjecture~\ref{conj:weak} we consider prime numbers for which $f(x) \bmod p$ is well-defined and algebraic, without considering which prime numbers satisfy this condition. In a refined version, Conjecture~\ref{conj:withL} we ask for some structure of this set of prime numbers. Here, we mention one important case related to an old conjecture of Christol. A power series $f(x)\in \Q\ps x$ is said to be \emph{globally bounded}, if there is $a\in \Z$ such that $f(ax)-f(0)\in \Z\ps x$. In particular, in that case $f(x)$ can be reduced modulo almost all primes $p$, possibly with the exceptions of the primes dividing $a$ and the denominator of $f(0)$. Christol's conjecture \cite{Chr86} states that every D-finite power series that is globally bounded, is the diagonal of a multivariate rational function. By Fursteberg's Theorem~\cite{Fu67}, $f(x)\bmod p$ then is algebraic for almost all $p$. Christol's conjecture is still wide open, for example, for the globally bounded function $\pFq 3 2 {\left[\frac 1 9, \frac 4 9, \frac 5 9\right ]} {\left[ \frac 1 3, 1\right ]} x$, discussed in Subsection~\ref{sssec:ex:3F2}, it is not known, whether it is a diagonal \cite[p.1085, footnote]{BY22}.  
\end{rem}

\subsubsection{Two refinements}
\label{sssec:refinements}

Although Conjecture~\ref{conj:weak} might look already rather strong and difficult, there are several points on which it remains quite vague and unprecise.
The first one is that it does not say too much about the groups $G_i$, except that they exist and are contained in some $\GL_n$.
However, while this fragility is certainly the most significant one, discussing it will require some categorical language and machinery.
For this reason, we prefer postponing this to Subsection~\ref{ssec:category} and start by presenting, in this subsection, two other simpler improvements.

\paragraph{A Galois equivariant version.}

A first evident weakness of Conjecture~\ref{conj:weak} is that it does not concern the complete Galois group $\Gal(f(x)\mid\Fp(x))$ but only a part of it, even if it is quite large and arguably the most interesting one.
This issue can be solved by adding descent data in the formulation of the conjecture as follows.
We require that $K/\Q$ is Galois and that each group $G_i \subseteq \GL_n(K)$ is stable by the Galois action. For a prime $\pp$ of $K$, the group
\begin{equation}
\label{eq:Gip}
G_{i,\pp} \coloneqq \image\big (G_i\cap \GL_n(\mathcal{O}_{(\pp)}) \to \GL_n(k_\pp)\big)
\end{equation}
which appears in Conjecture~\ref{conj:weak} is then stable by the natural action of $\Gal(k_\pp/\Fp)$;
indeed this action corresponds to that of the decomposition group at $\pp$ which, by assumption, stabilizes $G_i$.
We can then form the semi-direct product $G_{i,\pp} \rtimes \Gal(k_\pp/\Fp)$ (where $\Gal(k_\pp/\Fp)$ acts on $G_{i,\pp}$ through its natural action on $\GL_n(k_\pp)$), which looks like a very good candidate to be the complete Galois group $\Gal(f(x)\mid\Fp(x))$.
This expectation is actually a bit too naive because the factor $\Gal(k_\pp/\Fp)$ may vary a lot with the choice of $K$. Instead, we propose the following formulation.

\begin{conj}
\label{conj:galoisequiv}
Let $f(x)\in \Q\ps{x}$ be a $D$-finite series and let $n$ be the minimal order of a differential operator that annihilates $f(x)$.\\
Then, there exist a Galois extension $K/\Q$ and a family $(G_1, \ldots, G_t)$ of subgroups of $\GL_n(K)$ stable by the action of $\Gal(K/\Q)$ such that the following holds:
for almost all rational primes $p$ for which $f(x) \bmod p$ is well defined and algebraic over $\Fp(x),$ and for any prime $\pp$ of $K$ above $p$,
there exists an index $i \in \{1, \ldots, t\}$ such that
\begin{enumerate}[label=(\alph{enumi})]
\item \label{item:conj2:rsf} the residual splitting field of $f(x) \bmod p$ is the field $\ell_\pp$
defined as the smallest subfield of $k_\pp$ for which $G_{i,\pp} \subseteq \GL_n(\ell_\pp)$.
\item \label{item:conj2:iso} there is an isomorphism of exact sequences:
$$\xymatrix @R=3ex {
1 \ar[r] & \Gal \big(f(x)\mid k_\pp(x)\big) \ar[r] \ar[d]^-{\sim} & \Gal \big(f(x)\mid\Fp(x)\big) \ar[r] \ar[d]^-{\sim} & \Gal\big(\ell_\pp/\Fp\big) \ar[r] \ar@{=}[d] & 1 \\
1 \ar[r] & G_{i,\pp} \ar[r] & G_{i,\pp} \rtimes \Gal\big(\ell_\pp/\Fp\big) \ar[r] & \Gal\big(\ell_\pp/\Fp\big) \ar[r] & 1}$$
where $G_{i,\pp}$ is defined by Equation~\eqref{eq:Gip} and the exact sequence on the top is that of Lemma~\ref{lem:splitting}.
\end{enumerate}
\end{conj}

\begin{rem}
The point~\ref{item:conj2:iso} above is equivalent to requiring, in addition of the assertion~\ref{item:conj:iso} of Conjecture~\ref{conj:weak}, that the exact sequence of Lemma~\ref{lem:splitting} splits and that the induced action of $\Gal(\ell_\pp/\Fp)$ on $G_{i,\pp}$ agrees with its natural action on $\GL_n(\ell_\pp)$.
\end{rem}

\begin{rem}
Another interesting feature of Conjecture~\ref{conj:galoisequiv} is that it clearly highlights that all the places $\pp$ above a fixed rational prime $p$ play the same role (since they are permuted by $\Gal(K/\Q)$).
\end{rem}

Finally, we leave it as an exercise to the reader to check that all the examples of Subsection~\ref{ssec:examples} satisfy Conjecture~\ref{conj:galoisequiv}.

\paragraph{Controlling the partition of primes.}

The second point on which we would like to add precisions to Conjecture~\ref{conj:weak} concerns the dependence of the index $i$, \emph{i.e.}, the subgroup $G_i$ we choose, with respect to~$p$.
On all the examples of Subsection~\ref{ssec:examples}, we observe that this dependence is very easy to settle: the choice of $i$ always only depends on the congruence class of $p$ modulo a certain fixed number.
We would not expect however such a simple pattern to be repeated in full generality.
Our doubts come already from the case of linear differential equation of order $1$, namely
$$f'(x) + a(x) \cdot f(x) = 0.$$
More precisely, we have explained in Remark~\ref{rem:order1} that for a solution $f(x)$ of the above equation, the property to be reducible modulo~$p$ is governed by the factorization properties of $a(x) \bmod p$ which, in turn, can be read off from the decomposition of $p$ in the splitting field $L$ of $a(x)$ over $\Q$.
This behavior is indeed controlled by congruence conditions on $p$ when $L$ is an abelian extension of $\Q$; however it might be much more complicated in general.

For this reason, we prefer formulating our refined conjecture as follows.

\begin{conj}
\label{conj:withL}
Let $f(x)\in \Q\ps{x}$ be a $D$-finite series.
Let $n$ be the minimal degree of a differential operator that annihilates $f(x)$.
Then, there exist
\begin{itemize}
\item a number field $K$,
\item a finite Galois extension $L/\Q$,
\item two subsets $\mathcal{C}_{\alg}\subseteq \mathcal{C}_{\red}$ of the set of conjugacy classes of $\Gal(L/\Q)$ and
\item for each $C\in \mathcal{C}_{\alg}$, a subgroup $G_C\subseteq \GL_n(K)$
\end{itemize}
such that, for almost all primes $p$, if $C$ denotes the conjugacy class containing the Frobenius morphism at $p$ in $\Gal(L/\Q)$, we have:
\begin{enumerate}[label=(\arabic{enumi})]
    \item $C\in \mathcal{C}_{\red}$ if and only if $f(x)\in \Zp\ps{x}$,
    \item $C\in \mathcal{C}_{\alg}$ if and only if $f(x)\in \Zp\ps{x}$ and $f(x)\bmod p$ is algebraic over $\Fp(x)$,
    \item if $C\in \mathcal{C}_{\alg}$, for any prime $\pp$ in $K$ above $p$, the residue field $k_\pp$ contains the residual splitting field of $f(x) \bmod p$ and
    there is an isomorphism
    \[\Gal\big(f(x)\mid k_\pp(x)\big)\simeq \image\big (G_C\cap \GL_n(\mathcal{O}_{(\pp)}) \to \GL_n(k_\pp)\big).\]
\end{enumerate}
\end{conj}

We underline that Conjecture~\ref{conj:withL} also includes a criterion for recognizing whether $f(x)$ can be reduced modulo~$p$ and, when this is the case, whether the reduction is algebraic or not.
This part was not covered in the previous formulations.

On a different note, it is obvious that if Conjecture~\ref{conj:withL} holds with a number field $L$ and it holds for any extension $L'$ of $L$ (with $L'/Q$ Galois) as well.
Since a similar property also holds for the number field $K$, one can safely assume that $K = L$ in Conjecture~\ref{conj:withL}.
However, it looks important to us to make the distinction between $K$ and~$L$ as they seem to play different roles.
For instance, in many examples considered in Subsection~\ref{ssec:examples}, taking simply $L = \Q$ (which corresponding to having just one group $G_i$) is possible while it is often necessary to work with a nontrivial number field $K$.

Finally, we notice that the two refinements provided by Conjectures~\ref{conj:galoisequiv} and~\ref{conj:withL} are somehow independent and can be combined.

\subsection{The language of differential modules and Frobenius modules}
\label{ssec:category}

In this subsection, following standard ideas which go back to Dwork and Christol (see in particular \cite{Chr86a}),
we reinterpret the main protagonists we have encountered so far using the more powerful language of differential modules and Frobenius modules and use it to gain height on Conjecture~\ref{conj:weak}.
This will help us to better understand what could be the mysterious groups~$G_i$ by connecting them to the differential Galois group of the underlying differential equation.

\subsubsection{Different types of modules}

\paragraph{Differential modules.}

We begin with linear differential equations.
It is a standard fact that the categorical shadow of these equations is the notion of differential modules, whose definition is recalled below.

\begin{defi}
Let $K$ be a field equipped with a derivation $\partial: K \to K$.
A \emph{differential module} (or a \emph{module with connection}) over $(K,\partial)$ is 
a finite dimensional $K$-vector space $M$, together with an additive map $\nabla$ satisfying
$$\forall \lambda \in K,\, \forall m \in M, \quad \nabla(\lambda m)=\partial(\lambda) m + \lambda \nabla(m).$$
We write $\DiffMod K$ for the category of differential modules over $(K,\partial)$.
\end{defi}

To any differential equation of the form
\begin{equation}
\label{eq:diffeq}
\partial^r(f) + a_{r-1} \partial^{(r-1)}(f) + \cdots + a_1 \partial(f) + a_0 f=0 \qquad (a_i \in K)
\end{equation}
we attach the differential module $M = K e_0 \oplus \cdots \oplus K e_{r-1}$ with connection $\nabla$ defined by
\begin{align*}
\nabla(e_i) & = e_{i+1} \qquad \text{if } 0 \leq i \leq r{-}2,\\
\nabla(e_{r-1}) & = -(a_0 e_0 + \cdots + a_{r-1} e_{r-1}).
\end{align*}
The map $f \mapsto \sum_{i=0}^{r-1} \partial^i(f) e_i$ induces a bijection between the space of solutions in $K$ of the differential equation~\eqref{eq:diffeq}
with the space $M^{\nabla = 0}$ of \emph{horizontal vectors}, that are, by definition, vectors $m$ such that $\nabla(m) = 0$.

This framework also allows for looking at solutions in extensions.
More precisely, let $L$ be an extension of $K$ over which the derivation $\partial$ extends.
Given a differential module $M$ over $K$, one can form the tensor product $L \otimes_K M$ and endow it with a connection defined by
$$\nabla(\lambda \otimes m) = \partial(\lambda) \otimes m + \lambda \otimes \nabla(m) \qquad (\lambda \in L, m \in M).$$
The solutions of $M$ over $L$ are then defined as the space of horizontal vectors of $L \otimes_K M$, that is $(L \otimes_K M)^{\nabla = 0}$.

\paragraph{Frobenius modules.}

For the general purpose of this article, we recall from the introduction (see Equation~\eqref{eq:Frobenius}) that we are mostly interested in algebraic relations of the form
\begin{equation}
\label{eq:Frobenius2}
Y^{p^n} + c_{n-1} Y^{p^{n-1}} + \cdots + c_1 Y^p + c_0 Y = 0
\end{equation}
where $Y$ is the unknown and the coefficients $c_i$ lie in some field of characteristic $p$ which, in our case of interest, is $\Fp(x)$ or an extension of the form $\ell(x)$ where $\ell/\F_p$ is a field extension.
Fortunately, this type of equations can be handled very similarly to what we have presented above for linear differential equations.

In the algebraic situation, the relevant objects are the so-called Frobenius modules that we define now.

\begin{defi}
\label{def:frobmodules}
Let $R$ be a ring equipped with an endomorphism $\phi: R \to R$.
A \emph{Frobenius module} over $(R, \phi)$ is a finitely generated projective $R$-module $M$ equipped with an additive map $\phi_M : M \to M$ such that
$$\forall \lambda \in K,\, \forall m \in M, \quad \phi_M(\lambda m) = \phi(\lambda) \cdot \phi_M(m).$$
We let $\FrobMod R$ denote the category of Frobenius modules over $(R, \phi)$. 
\end{defi}

\begin{notation}
\label{not:FrobModq}
When the ground ring $R$ has characteristic $p$ and $\phi : R \to R$ is the map $x \mapsto x^q$ (where $q$ is a power of $p$), we write $\FrobModq R$ for $\FrobMod R$.
\end{notation}

The relationship between Frobenius equations and Frobenius modules is realized as follows.
Starting from Equation~\eqref{eq:Frobenius2} with $c_i \in R$, we define the Frobenius module $M = R e_0 \oplus \cdots \oplus R e_{r-1}$ (where $R$ is equipped with its natural Frobenius) with
\begin{align*}
\phi_M(e_i) & = e_{i+1} \qquad \text{if } 0 \leq i \leq r{-}2,\\
\phi_M(e_{r-1}) & = -(c_0 e_0 + \cdots + c_{r-1} e_{r-1}).
\end{align*}
The solutions of Equation~\eqref{eq:Frobenius2} then correspond to vectors $m \in M$ such $\phi_M(m) = m$ \emph{via} the map $y \mapsto \sum_{i=0}^{r-1} y^{p^i} e_i$.

\paragraph{Galois representations.}

The last notion we would like to encapsulate in our new language is, of course, that of Galois groups.
For this, given a field $K$, we write $G_K$ for its absolute Galois group, \emph{i.e.}, $G_K \coloneqq \Gal(K^\sep/K)$.
Given an auxiliary field $E$, we introduce the category $\Rep_E(G_K)$ of all $E$-linear finite dimensional representations of $G_K$.

The connection with the Galois groups we used to consider earlier is as follows.
Assume that $K$ contains $E$ and let then $x$ be an element in the separable closure $K^\sep$ of $K$.
We consider its Galois conjugate $x_1 = x, x_2, \ldots, x_N \in K^\sep$ and the $E$-vector space $V$ they generate.
Clearly $V$ is stable by the action of $G_K$. Hence, it is a $E$-linear representation of $G_K$, \emph{i.e.}, an object in the category $\Rep_E(G_K)$.
Besides, if $\rho : G_K \to \GL(V)$ is the morphism defining the representation, one checks that $\Gal(x\mid K) \simeq \im\rho$.

\subsubsection{Katz' functor}\label{sssec:Katz}

Katz' functor connects Frobenius modules with Galois representations;
in some sense, it is the categorical realization of the operation that takes an algebraic equation of the form~\eqref{eq:Frobenius2} and produces the Galois group of one of its solutions.

Concretely, it is defined as follows. We pick a field $K$ of characteristic~$p$, together with a finite extension $\F_q$ of $\Fp$. We then define
\begin{equation} \label{eq:katzfunctor}
\begin{array}{rcl}
\mathbf V : \quad \FrobModq K & \longrightarrow & \Rep_{\F_q}(G_K) \smallskip \\
M & \longmapsto & \big(K^\sep \otimes_K M\big)^{\varphi = 1}
\end{array}
\end{equation}
where we let $\varphi$ act on $K^\sep \otimes_K M$ by $\lambda \otimes m \mapsto \lambda^q \otimes \phi_M(m)$ and 
where the notation ``$\varphi = 1$'' means that we consider the subset of elements that are fixed by $\varphi$.
When $M$ comes from an equation of the form~\eqref{eq:Frobenius2} (and thus $q = p$), it follows directly from the definition that
the representation $\mathbf V(M)$ is the $\Fp$-vector space of the solutions in $K^\sep$ equipped with the standard Galois action.

To study Katz' functor, it is convenient to introduce the following definition.

\begin{defi}\label{def:etale}
Let $R$ be a ring equipped with an endomorphism $\phi: R \to R$.
We say that a Frobenius module $M \in \FrobMod R$ is \emph{étale} when the image of $\phi_M$ spans $M$. \smallskip \\
We write $\FrobModet K$ for the subcategory of $\FrobMod K$ consisting of étale Frobenius modules.
Moreover, when $\phi$ is the map $x \mapsto x^q$, we write $\FrobModqet K$ for $\FrobModet K$ (see also Notation~\ref{not:FrobModq}).
\end{defi}

\begin{thm}[Katz] \label{thm:Katz}
Let $K$ be a field of characteristic $p$ and let $q$ be a power of $p$.
The functor $\mathbf V$ induces an equivalence of categories $\FrobModqet K \stackrel\sim\longrightarrow \Rep_{\F_q}(G_K)$.
Moreover, it preserves the dimension in the sense that, for all $M \in \FrobModqet K$, we have
$$\dim_K M = \dim_{\F_q} \mathbf V(M).$$
\end{thm}

\begin{proof}
See \cite[Proposition~4.1.1]{Ka73}.
\end{proof}

\begin{rem}
\label{rem:inversekatz}
There exists an explicit formula for the inverse functor $\mathbf M$ of $\mathbf V$:
if $V$ is a $\F_q$-linear representation of $G_K$, its corresponding étale Frobenius module is
$\mathbf M(V) \coloneqq (K^\sep \otimes_K V)^{G_K}$
where $G_K$ acts diagonally on the tensor product $K^\sep \otimes_K V$.
\end{rem}

Beyond Katz' theorem, it will be also important for our purpose to study the action of Katz' functor on nonétale Frobenius modules.
In this perspective, we set the following definition.

\begin{defi}
\label{def:etalepart}
Let $K$ be a field equipped with a ring endomorphism $\phi : K \to K$ and let $M$ be a Frobenius module over $(K, \phi)$.
The \emph{étale part} of $M$ is defined as
$$M^\et \coloneqq \bigcap_{k \geq 0} \langle \phi_M^k(M) \rangle_K$$
where the notation $\langle \, \cdot \, \rangle_K$ stands for the $K$-span.
\end{defi}

It is straightforward to check that $M^\et$ is stable by $\phi_M$, \emph{i.e.}, it is again a Frobenius module, and that it is étale.
In other words, the construction $M \mapsto M^\et$ defines a functor $\FrobMod K \to \FrobModet K$.

\begin{lem}
\label{lem:VMet}
We assume that $K$ has characteristic $p$ and that $q$ is a power of $p$.
Then, for all $M \in \FrobModq K$, we have $\mathbf V(M) = \mathbf V(M^\et)$.
\end{lem}

\begin{proof}
We form the quotient $M/M^\et$ which inherits an action of $\phi_M$ and thus is an object in the category $\FrobMod K$.
Besides, it follows readily from the definition that $\phi_M$ acts nilpotently on $M/M^\et$.
Therefore, it also acts nilpotently on $K^\sep \otimes_K M/M^\et$, from which follows that $\mathbf V(M/M^\et) = 0$.
We now consider the exact sequence
$$0 \longrightarrow M^\et \longrightarrow M \longrightarrow M/M^\et \longrightarrow 0$$
which remains exact after tensoring by $K^\sep$. Taking then the fixed points under $\phi$ (which is a left exact functor), we end up with
$$0 \longrightarrow \mathbf V(M^\et) \longrightarrow \mathbf V(M) \longrightarrow \mathbf V(M/M^\et) = 0.$$
The lemma follows.
\end{proof}

\begin{rem}
Lemma~\ref{lem:VMet} tells us that $M^\et$ is the étale Frobenius module which corresponds to the Galois representation $\mathbf V(M)$.
With the notation of Remark~\ref{rem:inversekatz}, we then have $M^\et = \mathbf M(\mathbf V(M))$.
\end{rem}

\subsubsection{Frobenius structures on differential modules} \label{sssec:FrobStruct}

The relationship between differential modules and Frobenius modules is far less concrete, but it can nevertheless be approached through the notion of Frobenius structures.
Basically, the idea is to equip a differential module $M$ with an extra structure which will eventually define a Frobenius module after reduction modulo~$p$.

In order to define it properly, we cannot continue working over $\Q(x)$ but need to move over the $p$-adics.
Precisely, we define the field $E_p$ as the completion of $\Q(x)$ for the $p$-adic Gauss norm (defined as the sup norm on the $p$-adic unit disc);
it is the so-called field of \emph{$p$-adic analytic elements}. It is equipped with the derivation $\frac d{dx}$ and a Frobenius map $\varphi : E_p \to E_p$ which comes from the change of variables $x \mapsto x^p$.

\begin{defi}
\label{def:frobmorphism}
Let $(M_1, \nabla_1)$ and $(M_2, \nabla_2)$ be two differential modules over $E_p$.
A \emph{Frobenius morphism} is a $\varphi$-semilinear map $\Phi : M_1 \to M_2$ such that
$\nabla_2 \circ \Phi = p x^{p-1} \cdot \Phi \circ \nabla_1$.
\end{defi}

It is instructive to understand what the effect of a Frobenius morphism at the level of solutions is.
The commutation relation of Definition~\ref{def:frobmorphism} ensures that a Frobenius morphism $\Phi : M_1 \to M_2$ induces a $\Qp$-linear mapping on horizontal vectors
$\Phi : M_1^{\nabla = 0} \longrightarrow M_2^{\nabla = 0}$.
If we equip $M_1$ and $M_2$ with distinguished bases and if $\mathrm{Mat}(\Phi)$ denotes the matrix of $\Phi$ in those bases, we deduce that
any solution $(f_1(x), \ldots, f_n(x))$ of $M_1$ produces a solution of $M_2$ \emph{via} the formula
$$\big(\varphi(f_1(x)), \ldots, \varphi(f_n(x))\big) \cdot \mathrm{Mat}(\Phi) =
  \big(f_1(x^p), \ldots, f_n(x^p)\big) \cdot \mathrm{Mat}(\Phi).$$
If we assume in addition that $\Phi : M_1 \to M_2$ is an isomorphism, then any solution of $M_2$ can be (uniquely) written in this way.
Frobenius morphisms then appear as a general machinery that permits to interpret the images under $\phi$ of a solution of a given differential equation as linear combinations of the solutions of another differential equation
and \emph{vice versa} when $\Phi$ is an isomorphism.
Phrased in this way, we understand that they could provide a powerful tool to find the algebraic relations we are looking for (remember that $f(x^p) = f(x)^p$ after reduction modulo~$p$).

In order to fully take advantage of this observation, it is convenient to use the concept of \emph{preimages under the Frobenius}.
To avoid unnecessary complications, we assume from now on that we are given a $D$-finite series $f(x) \in \Zp\ps{x}$ and that
$M$ is the differential module associated to the linear differential equation over $E_p$ of minimal order satisfied by $f(x)$.
A theorem of Christol (see~\cite[Theorem~10.4.2]{Ke10}) then implies that $M$ admits a preimage under the Frobenius,
which means that there exists a Frobenius \emph{isomorphism} $M^{(-1)} \to M$ where $M^{(-1)}$ is another differential module.
Even better, one can apply Christol's theorem repeatedly and get an infinite sequence of Frobenius isomorphisms
$$\cdots \longrightarrow
M^{(-3)} \stackrel\sim\longrightarrow
M^{(-2)} \stackrel\sim\longrightarrow
M^{(-1)} \stackrel\sim\longrightarrow M.$$

We now come to an alternative: either the sequence of $M^{(-i)}$ is periodic or it is not.
In the first case, \emph{i.e.}, when there exists a positive integer $h$ such that $M^{(-h)}$ is isomorphic to $M$ as a differential module, we say that $M$ admits a \emph{strong Frobenius structure} (of period $h$).
When this occurs, the composite
$$M \simeq M^{(-h)} \stackrel\sim\longrightarrow \cdots \stackrel\sim\longrightarrow M^{(-1)} \stackrel\sim\longrightarrow M.$$
defines a structure of Frobenius modules on $M$ for the twisting morphism $\varphi^h$.
This makes the connection we were looking for.

On the contrary, when $M$ does not admit a strong Frobenius structure, the above construction fails;
however, in this case, we suspect that $f(x) \bmod p$ is not algebraic.

\paragraph{Some examples.}

It is quite interesting to revisit some examples of Subsection~\ref{ssec:examples} in light of what precedes.
Let us start with Example~\ref{sssec:ex:expatan}.
Here $f(x) = \exp(\arctan(x))$ and we recall that, whenever $p \equiv 1 \pmod 4$, we had obtained the following alternative formula for $f(x)$ as a $p$-adic series:
$$f(x) = \left(1 + i x\right)^{-i/2} \cdot \left(1 - i x\right)^{i/2} \in \Z_p\ps{x}$$
where $i \in \Z_p$ is a square root of $-1$.
We recall that we also decomposed the exponents $\pm i/2$ in base $p$ as follows:
$$\textstyle 
-\frac i 2 = a_0 + p a_1 + p^2 a_2 + \cdots 
\quad \text{and} \quad
\frac i 2 = b_0 + p b_1 + p^2 b_2 + \cdots$$
It turns out that this decomposition actually reflects the successive preimages by Frobenius of the underlying differential module.
Indeed, setting $u_1 \coloneqq \frac {-i/2 - a_0} p = a_1 + p a_2 + \cdots$,
we can rewrite
$$\left(1 + i x\right)^{-i/2} = \left(1 + i x\right)^{a_0} \left(\frac{(1 + i x)^p}{1 + i x^p}\right)^{u_1} \left(1 + i x^p\right)^{u_1}.$$
Now we observe that the two first factors of the right-hand side lie in $E_p$:
it is obvious for the first one, which is a polynomial, and 
it holds true for the second one as well because $\frac{(1 + i x)^p}{1 + i x^p}$ is a rational function, which is congruent to $1$ modulo $p$.
Proceeding similarly for the factor $\left(1 - i x\right)^{i/2}$, we end up with a factorization of the form $f(x) = g_0(x) f_1(x^p)$ where $g_0(x) \in E_p$ and
$$f_1(x) = \left(1 + i x\right)^{u_1} \cdot \left(1 - i x\right)^{v_1}$$
with $v_1 \coloneqq \frac {i/2 - b_0} p = b_1 + p b_2 + \cdots$.
Moreover, a simple computation shows that the function $f_1(x)$ is a solution of the linear differential equation
$$y'(x) = \frac{i (u_1{-}v_1) + (u_1{+}v_1) x} {1 + x^2} \cdot y(x).$$
It is the preimage under the Frobenius of our initial differential equation for $f(x)$.
Repeating this argument, we find that the successive preimages under the Frobenius are the differential equations
$$y'(x) = \frac{i (u_k{-}v_k) + (u_k{+}v_k) x} {1 + x^2} \cdot y(x)$$
with $u_k = a_k + p a_{k+1} + p^2 a_{k+2} + \cdots$ and $v_k = b_k + p b_{k+1} + p^2 b_{k+2} + \cdots$.
The nonperiodicity of the sequences $(a_k)_{k \geq 0}$ and $(b_k)_{k \geq 0}$, or equivalently that of the sequences $(u_k)_{k \geq 0}$ and $(v_k)_{k \geq 0}$, now implies that the sequence of preimages under the Frobenius is not periodic either.
This conclusion is perfectly in line with the fact that $f(x) \bmod p$ is not algebraic.

\medskip

Another interesting example to consider is Example~\ref{sssec:ex:apery} with the Apéry series.
Here, we recall that we had defined two functions $g(x)$ and $h(x)$ whose reductions modulo $p$ satisfy interlaced algebraic relations (see Equations~\eqref{eq:algapery:start}--\eqref{eq:algapery:end}).
Those equations suggest that the differential equations that $g(x)$ and $h(x)$ satisfy are fixed by the Frobenius when $p \equiv 1, 5, 7, 11 \pmod{24}$ and exchanged by the Frobenius when $p \equiv 13, 17, 19, 23 \pmod{24}$.

Similar patterns will often occur in Section~\ref{sec:hypergeom}, when we will study hypergeometric series.

\subsubsection{The long road from $\DiffMod{\Q(x)}$ to $\Rep_{\F_q(x)}(G_{\F_q(x)})$}

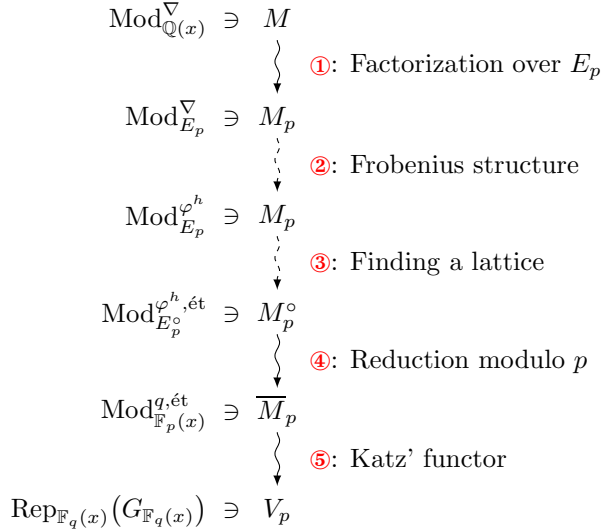
\begin{figure}
\hfill%
\begin{tikzpicture}[decoration={snake, amplitude=0.2ex, post length=3pt}, yscale=-1.3]
\clip (-5,0.7) rectangle (5,6.3);
\node at (0,1) { \ph $M$ };
  \node[left] at (-0.8,1) { \ph $\DiffMod {\Q(x)}$ };
  \node[xscale=-1] at (-0.6,1) { \ph $\in$ };
\node at (0,2) { \ph $M_p$ };
  \node[left] at (-0.8,2) { \ph $\DiffMod {E_p}$ };
  \node[xscale=-1] at (-0.6,2) { \ph $\in$ };
\node at (0,3) { \ph $M_p$ };
  \node[left] at (-0.8,3) { \ph $\Mod^{\phi^h}_{E_p}$ };
  \node[xscale=-1] at (-0.6,3) { \ph $\in$ };
\node at (0,4) { \ph $M^\circ_p$ };
  \node[left] at (-0.8,4) { \ph $\Mod^{\phi^h, \et}_{E^\circ_p}$ };
  \node[xscale=-1] at (-0.6,4) { \ph $\in$ };
\node at (0,5) { \ph $\overline{M}_p$ };
  \node[left] at (-0.8,5) { \ph $\FrobModqet {\Fp(x)}$ };
  \node[xscale=-1] at (-0.6,5) { \ph $\in$ };
\node at (0,6) { \ph $V_p$ };
  \node[left] at (-0.8,6) { \ph $\Rep_{\F_q(x)}\big(G_{\F_q(x)}\big)$ };
  \node[xscale=-1] at (-0.6,6) { \ph $\in$ };
\draw[-latex, decorate] (0,1.2)--(0,1.75)
  node[right,xshift=2ex,midway] { \ref{item:step1}: Factorization over $E_p$ };
\draw[-latex, decorate, dash pattern={on 2pt off 2pt}] (0,2.2)--(0,2.75)
  node[right,xshift=2ex,midway] { \ref{item:step2}: Frobenius structure };
\draw[-latex, decorate, dash pattern={on 2pt off 2pt}] (0,3.2)--(0,3.75)
  node[right,xshift=2ex,midway] { \ref{item:step3}: Finding a lattice };
\draw[-latex, decorate] (0,4.2)--(0,4.75)
  node[right,xshift=2ex,midway] { \ref{item:step4}: Reduction modulo $p$ };
\draw[-latex, decorate] (0,5.2)--(0,5.75)
  node[right,xshift=2ex,midway] { \ref{item:step5}: Katz' functor };
\end{tikzpicture}\hfill\null
\caption{From differential equations in characteristic $0$ to Galois groups in characteristic~$p$}
\label{fig:steps}
\end{figure}

All the previous constructions provide a path for going from modules with connections over $\Q(x)$ (which encode the differential equation satisfied by $f(x)$) to Galois representations of the absolute Galois group $G_{\F_q(x)}$ of $\F_q(x)$ (which encodes the Galois group of $f(x) \bmod p$).
It consists in the following steps.
\begin{enumerate}[label=\raisebox{-0.2ex}{\ding{\theenumi}}]
\addtocounter{enumi}{171}
\item (Factorization over $E_p$) \label{item:step1}
We ``$p$-minimize'' the differential equation we have for $f(x)$, meaning that we replace it by the differential equation \emph{over $E_p$} of minimal order which is satisfied by $f(x)$;
we emphasize that this operation may result in a drop of order even if our initial differential equation was already minimal over $\Q(x)$.
At the level of modules, this consists in replacing $M$ by a submodule $M_p \subseteq E_p \otimes_{\Q(x)} M$.
\item (Frobenius structure) \label{item:step2}
Following the discussion of Subsection~\ref{sssec:FrobStruct}, we equip $M_p$ with a Frobenius structure $\Phi_h : M_p \to M_p$, obtaining this way a Frobenius module over $E_p$ for the twisting morphism $\phi^h$ (where $h$ is the period);
we underline that this operation may fail if the sequence of the Frobenius preimages of $M_p$ is not periodic.
\item (Finding a lattice) \label{item:step3}
In order to reduce modulo $p$ afterwards, we need to find an integral structure inside $M_p$.
Let $E_p^\circ$ denote the valuation ring of $E_p$, that is the subring of $E_p$ consisting of elements of nonnegative valuation.
We select a finitely generated $E_p^\circ$-submodule $M_p^\circ$ of $M_p$ having the following properties: 
it generates $M_p$ over $E_p$, it is stable by $\Phi_h$ and it defines an étale Frobenius module over $E_p^\circ$.\\
Such a module may actually not exist; if it does not, we allow ourselves to change the Frobenius structure, \emph{i.e.}, to replace $\Phi_h$ by $\Phi_h \circ \iota$ where $\iota$ is a $E_p$-linear automorphism of $M_p$ commuting with the connection.
If even after this relaxation, we are still not able to find a suitable module $M_p^\circ$, we agree that this step has failed.
\item (Reduction modulo $p$) \label{item:step4}
We reduce $M_p^\circ$ modulo~$p$, \emph{i.e.}, we set $\overline{M}_p \coloneqq M_p^\circ/p M_p^\circ$.
It is an étale Frobenius module over $\Fp(x)$ for the twisting morphism $f \mapsto f^q$ with $q = p^h$.
\item (Katz' functor) \label{item:step5}
We finally apply Katz' functor to get a representation of $G_{\F_q(x)} \coloneqq \Gal(\F_q(x)^\sep/\F_q(x))$.
Some precaution should be taken here because Katz' theorem requires that the base field of the Frobenius module contains $\F_q$, which is not the case in our setting.
We then first extend scalar to $\F_q(x)$, \emph{i.e.}, we define
$$V_p \coloneqq \mathbf V\big(\F_q(x) \otimes_{\Fp(x)} \overline{M}_p\big) 
 = \big(\F_q(x)^\sep \otimes_{\Fp(x)} \overline{M}_p\big)^{\varphi = 1}.$$
\end{enumerate}
All this process is summarized in Figure~\ref{fig:steps}.
Dashed arrows correspond to steps that may fail. However, if a failure is encountered at some point, our feeling is that $f(x) \bmod p$ cannot be algebraic.
On the contrary, when all operations succeed, it is known that $f(x) \bmod p$ is algebraic~\cite[Theorem~2.6]{Var21}. 
What we expect actually is that $f(x) \mod p$ is somehow ``contained'', together with its conjugates, in the Galois representation $V_p$ which is produced at the end of the process.
More precisely, if we let $\rho_p : G_{\F_q(x)} \to \GL(V_p)$ be the corresponding group morphism, we expect the Galois group $\Gal(f(x)\mid\F_q(x))$ to be isomorphic to a quotient of the image of $\rho_p$.

\begin{rem}
It worths noticing that Theorem~2.6 of \cite{Var21} only requires the existence of a strong Frobenius structure, and not that of a lattice. This suggests that Step~\ref{item:step3} in our construction could never fail.
\end{rem}

\subsubsection{Connecting Galois groups: insights towards our conjectures} \label{sssec:tannakian}

The Tannakian language provides efficient tools to define Galois groups at each step of our constructions and, to some extent, to follow them throughout the whole process.
We refer to \cite{DM82} for an exposition of the Tannakian theory. 
Here, we just recall that the Tannakian formalism attaches a ``Galois group'' to any object in a Tannakian category. Its definition goes through the notion of constructions.
Given $T$ in a Tannakian category $\mathcal T$, a \emph{construction} from $T$ is by definition a subquotient of an object of the form
$$C \coloneqq T^{\otimes{n_1}} \oplus T^{\otimes{n_2}} \oplus \cdots \oplus T^{\otimes{n_t}} \qquad (n_i \in \Z)$$
where, when $n$ is negative, $T^{\otimes n}$ is understood as the dual of $T^{\otimes(-n)}$.
The Galois group $\Gal_{\mathcal T}(T)$ is then the group of linear automorphisms $\sigma : T \to T$ satisfying the following condition: 
for any $C$ as above and any subobject $S \subseteq C$, the natural extension of $\sigma$ to $C$ stabilizes $S$.
One proves that $\Gal_{\mathcal T}(T)$ is an \emph{algebraic} subgroup of $\GL(T)$.

We now analyze the effect on Galois groups of each step of Figure~\ref{fig:steps} in light of what precedes.
Step~\ref{item:step1} has already quite a nontrivial impact.
To describe it properly, it is convenient to split it into two parts: firstly, we extend scalars to $E_p$ and secondly, we minimize the differential equation.
Regarding scalar extension, it follows from the definition that we have an inclusion
$$\Gal_{\DiffMod{E_p}}(E_p \otimes_{\Q(x)} M) \subseteq \Gal_{\DiffMod{\Q(x)}}(M).$$
Besides, if $E_p \otimes_{\Q(x)} M$ admits a submodule $M'$, any element of the Galois group of $E_p \otimes_{\Q(x)} M$ necessarily stabilizes $M'$.
Applying this remark with $M' = M_p$, we obtain a group morphism
$$\varpi: \Gal_{\DiffMod{E_p}}(E_p \otimes_{\Q(x)} M) \longrightarrow \GL(M_p)$$
and it follows from the Tannakian theory that the Galois group of $M_p$ is the image of $\varpi$.
As a summary, one can retain that $\Gal_{\DiffMod{E_p}}(M_p)$ appears as a subquotient of $\Gal_{\DiffMod{\Q(x)}}(M)$.

Although the steps~\ref{item:step2} and~\ref{item:step3} seem to be the most delicate ones, their effect on the Galois groups is quite harmless.
Indeed, a theorem of Matzat~\cite{Ma09a,Ma09b} shows that
$$\Gal_{\DiffMod{E_p}}(M_p) = \Gal_{\FrobModqet{E_p^\circ}}(M_p^\circ)$$
as soon as $M_p^\circ$ is defined.
More precisely, the datum of the lattice $M_p^\circ$ defines a prolongation to $\text{Spec }\Z_p$ of the algebraic group $\Gal_{\DiffMod{E_p}}(M_p)$, which was \emph{a priori} only defined over $\text{Spec }\Q_p$.

Step~\ref{item:step4} is a scalar extension and so, it has a similar effect on Galois groups as in Step~\ref{item:step1}:
the Galois group of $\overline{M}_p$ is a subgroup of the reduction modulo $p$ of the Galois group of $M_p^\circ$.
Concerning Step~\ref{item:step5}, the same phenomenon occurs when extending scalars to $\F_q(x)$, and finally, applying Katz' functor is transparent at the level of Galois groups because it is an equivalence of categories.

All in all, the conclusion is that the Galois group of $f(x) \bmod p$ appears as a subquotient of a certain scalar extension of the reduction modulo~$p$ of the differential Galois group of $M$.
This remark suggests that the groups $G_i$ of Conjecture~\ref{conj:weak} could show up as subquotients of the $K$-valued points of $\Gal_{\DiffMod{\Q(x)}}(M)$.
The fact that we are loosing dimensions when passing from $M$ to $M_p$ also explains that $G_i$ sometimes (naturally) embeds in $\GL_m(K)$ with $m < \dim M$.

\paragraph{Uniformity in $p$.}

So far, we have fixed a prime number~$p$ from the beginning and carried out all our constructions and reasonings with this particular~$p$.
In fact, the content of Conjecture~\ref{conj:weak} (and its refinements) is somehow that all of this can be achieved uniformly with respect to~$p$, with precise control encoded by certain number fields!

Recycling the same set of ideas, one could also expect that some numerical invariants that we have encountered along the way, \emph{e.g.}, the dimension of $M_p$, behave quite nicely when $p$ varies.
This leads to the following two independent questions, which sound quite interesting to us, and that we ask to the community.

\begin{quest}
\label{quest:orderEp}
Let $f(x) \in \Q\ps x$ be a $D$-finite series and,
for all prime number~$p$, let $r_p$ be the minimal order of a differential equation over $E_p$ satisfied by $f(x)$.
How does $r_p$ vary with respect to~$p$?
\end{quest}

\begin{quest}
\label{quest:pcurv}
Let $M$ be a differential module over $\Q(x)$ and,
for almost all prime number~$p$, let $d_p$ be the dimension over $\Fp(x^p)$ of the space of the solutions of $M$ over $\Fp(x)$.
By Cartier's Lemma (see~\cite[Theorem~3.19]{BCR24}), $d_p$ is also the dimension of the kernel of the $p$-curvature of $M$.
How does $d_p$ vary with respect to~$p$?
\end{quest}

If the exponents at $0$ of the minimal differential equation of a $D$-finite series $f(x) \in \Q\ps{x}$ are all $0$, the function $f(x)$ is said to have \emph{maximal unipotent monodromy (MUM)} and one proves that $r_p = d_p = 1$ for almost all $p$. This answers Questions~\ref{quest:orderEp} and~\ref{quest:pcurv} in this case.
In Section~\ref{sec:hypergeom} (see Theorem~\ref{thm:block}), we will study the case where $M$ is the differential module corresponding to a hypergeometric operator $\mathcal H$ and show, in this situation, that $d_p$ depends only on the congruence class of $p$ modulo a common denominator of the parameters of $\mathcal H$.

\section{Hypergeometric series and their reductions modulo~$p$}
\label{sec:hypergeom}

The hypergeometric differential operator is defined as
\begin{equation} \label{eq:hyp}
\Hyp(\ua, \ub)=-x\prod_{i=1}^n \big(\theta+\alpha_i\big)+\theta\prod_{j=1}^{m}(\theta+\beta_j-1) \qquad \Big(\theta\coloneqq x\cdot \frac{d}{d x}\Big)
\end{equation}
where $n$ and $m$ are positive integers, $\ua=(\alpha_1,\ldots, \alpha_n)\in (\Q\setminus (-\N))^n$ and $\ub=(\beta_1,\ldots, \beta_{m})\in (\Q\setminus (-\N))^{m}$. The solutions of the equation $\Hyp(\ua, \ub)y=0$ over the rational numbers can be expressed in terms of hypergeometric series 
\begin{equation}
\pFq{n}{m}{\ua}{\ub}{x} \coloneqq \sum_{k=0}^\infty h(\ua, \ub; k)x^k \in \Q\ps{x} ,
\end{equation}
with
\[h(\ua, \ub; k)\coloneqq \frac{(\alpha_1)_k\cdots (\alpha_n)_k}{(\beta_1)_k\cdots (\beta_{m})_k}\cdot \frac{1}{k!}\in \Q,\]
where $(\gamma)_k\coloneqq \gamma\cdot (\gamma+1)\cdots (\gamma+k-1)$ denotes the rising factorial.

In the following we will restrict to the case $m=n-1$. Moreover, if we consider a set of parameters $\ua, \ub\in \Q^{n+(n-1)}$, we set $\beta_n=1$ and write $\ub$ both for the $(n{-}1)$-tuple of rational parameters and for the $n$-tuple also including $\beta_n=1$.  For given parameters $\ua, \ub$ we denote by $d(\ua,\ub)$ the least common multiple of their denominators in reduced form. If the parameters are clear, we will omit them from the notation and will simply write $d$.

Let us recall some facts about the hypergeometric differential equation $\Hyp(\ua, \ub)y=0$. It is \emph{Fuchsian}, meaning that all of its \emph {singularities} are \emph {regular}. The singular points are $0, 1, \infty$ and the sets of \emph{local exponents} at those singular points are
$$\begin{array}{rl}
\text{at 0:} & 
  \big\{1-\beta_1,\ldots, 1-\beta_{n-1}, 0\big\}, \smallskip \\
\text{at 1:} & \textstyle
  \big\{0, 1, 2, \ldots, n-2, -1+(\sum_{j=1}^{n}\beta_j-\sum_{i=1}^n \alpha_i)\big\}, \smallskip \\
\text{at $\infty$:} &
  \big\{\alpha_1,\ldots, \alpha_n\big\}.
\end{array}$$
Finally, assuming that $\alpha_i-\beta_j\not \in \Z$ for all $i,j \in \{1,\ldots, n\}$, a $\Q$-basis of solutions of $\Hyp(\ua, \ub)y=0$ at the singularity $0$ is given by 
\begin{equation} \label{eq:basishypergeom}
z^{1-\beta_j} \pFq{n}{n-1}{1 + \ua - \beta_j}{1 + \ub - \beta_j}{x},
\end{equation}
where $1+\beta_j-\beta_j$ is omitted in the bottom parameter.

Hypergeometric functions and equations serve as test cases for many conjectures and investigations concerning $D$-finite power series. For example, Christol conjectured that any globally bounded $D$-finite power series is a diagonal of an algebraic multivariate series, and investigated this for hypergeometric series \cite{Chr86}. Adamczewski and Delaygue, as attributed by Vargas-Montoya \cite[Conjecture~1.1]{Var21}, conjectured that the reduction modulo $p$ of any \textit{G-function}, which can be reduced modulo an infinite number of primes, is algebraic over $\F_p(x)$ for almost all such primes $p$. For hypergeometric series ${}_nF_{n-1}$, this statement follows from Christol's work \cite{Chr86a} showing that solutions of equations with strong Frobenius structure are $p$-automatic, and thus they are algebraic, following earlier work by Christol \cite{Ch79, CKMR80}. This result was made explicit by Vargas-Montoya \cite{Var21, Var24}, bounding the algebraicity degree and giving an explicit shape of an annihilating polynomial in the form of Equation~\eqref{eq:Frobenius2}. As test cases to Conjecture~\ref{conj:weak} and its refinements we will study hypergeometric series for the remaining part of this article.

\subsection{Reductions of hypergeometric series modulo $p$}
\label{ssec:hypergeommodp}

In this section, we investigate how hypergeometric functions and equations behave when passing to positive characteristic.
Our first result classifies those prime numbers $p$, modulo which a given hypergeometric function can be reduced, in terms of the repartition of the values $\exp(2\i\pi \lambda \alpha_i)$ and $\exp(2\i\pi \lambda \beta_j)$ (for $\lambda$ varying in $\Z$) on the complex unit circle.
This refines a theorem of Christol \cite[Proposition~1]{Chr86}, classifying all globally bounded hypergeometric functions, that is, those where the reduction is possible modulo almost all primes. We refer to \cite[Section~5]{DRR17} for a more recent treatment in English.

We then study the solution space of a hypergeometric differential equation modulo $p$ and describe an explicit basis of it.
We connect in particular its dimension to numerical invariants attached to the parameters $\ua$ and $\ub$ defining the underlying hypergeometric operator and show that it only depends on the congruence class of $p$ modulo $d$, a common denominator for $\ua, \ub$.
Our results provide in particular a satisfactory answer to Question~\ref{quest:pcurv}, raised at the end of Section~\ref{sec:picture} in the case of hypergeometric functions.

\subsubsection{Interlacing criteria}

We denote the $p$-adic valuation of a rational number by $v_p(\cdot)$. Further, we write $\dpz{\cdot}$ (resp. $\dpc{\cdot}$) for the decimal part function (resp. the decimal part where we assign to integers the value $1$ instead of $0$).
Christol defines a total ordering $\preceq$ on $\Q$ as follows:
\begin{equation*}
a\preceq b \longeq \dpc{a}<\dpc{b} \text{ or } (\dpc{a}=\dpc{b} \text{ and } a\geq b).
\end{equation*}
Let $f(x)=\pFq{n}{n-1}{\ua}{\ub}{x}$ be a hypergeometric function with rational parameters $\ua, \ub\in \big(\Q\setminus(-\N)\big)^{2n-1}$ and set $\beta_n=1$. We denote by $d$ the common denominator of all parameters. Then Christol defines a function $\mathcal{M} : \Q\times \Zdc \to \Z$ by
\begin{equation} \label{eq:defM}
    \mathcal{M}(x, \lambda)\coloneqq \big|\{1\leq i \leq n: \lambda \alpha_i\preceq x \}\big| - \big|\{1 \leq j \leq n: \lambda \beta_j\preceq x\}\big|.
\end{equation}
We note that the function $\mathcal{M}(\cdot, \lambda)$ is piecewise constant and continuous from the left, and that its jump points are among the values $\dpz{\alpha_i\lambda}+\Z$ and $\dpz{\beta_j\lambda}+\Z$.

We say that $f(x)$ satisfies \textit{Christol's interlacing condition} for $\lambda\in \Zdc$ if $\mathcal{M}(x, \lambda)\geq 0$ for all $x\in \Q$.
To determine whether $\mathcal{M}(\cdot, \lambda)$ is always positive, it suffices to check that $\mathcal{M}(\lambda \beta_j, \lambda)\geq 0$ for all $j$, or, to order the values of $\lambda \alpha_i$ and $\lambda \beta_j$ according to $\preceq$ simultaneously and check if among the first $\ell$ entries for $1\leq \ell \leq 2n$ of the joint ordered list there are always at least as many entries of the form $\lambda \alpha_i$ as of the form $\lambda \beta_j$.
If the parameters of a hypergeometric function $f(x)$ are pairwise incongruent modulo $\Z$, Christol's interlacing condition for $\lambda$ can be interpreted graphically on the unit circle as follows. We draw the two sets $\{\exp(2\i\pi \lambda \alpha_i):1\leq i\leq n\}$ in red and $\{\exp(2\i\pi \lambda \beta_j):1\leq j\leq n\}$ in blue on the unit circle. Then $f(x)$ satisfies Christol's interlacing condition if and only if, going around the unit circle starting after $1$, one always encounters at least as many red points as blue ones. 

\begin{figure}
\centering
\begin{tikzpicture}
  \draw[very thin, ->] (-0.2,0) -- (6.3,0) node[right] {$x$};
  \draw[very thin, ->] (0,-0.2) -- (0,2.3) node[above] {$\mathcal{M}(x, 1)$};
  
  \draw[-, very thick] (0, 0) --  (0.333, 0) {};
  \draw[-, very thick] (0.333, 1) -- (1, 1) {};
  \draw[-, very thick] (1, 0) -- (1.333, 0) {};
  \draw[-, very thick] (1.333, 1) -- (2, 1) {};
  
  \draw[-, very thick] (2.0, 0) -- (2.333, 0) {};
  \draw[-, very thick] (2.333, 1) -- (3, 1) {};
  \draw[-, very thick] (3, 0) -- (3.333, 0) {};
  \draw[-, very thick] (3.333, 1) -- (4, 1) {};
  
  \draw[-, very thick] (4.0, 0) -- (4.333, 0) {};
  \draw[-, very thick] (4.333, 1) -- (5, 1) {};
  \draw[-, very thick] (5, 0) -- (5.333, 0) {};
  \draw[-, very thick] (5.333, 1) -- (6, 1) {};
  
  \foreach \x in {0.333, 2.666, 1, 1.333, 2, 2.333, 3, 3.333, 4, 4.333, 5, 5.333, 6}
    \draw[thin, dotted] (\x, 2.3) -- (\x, -0.3);
     
  \foreach \y in {0, 1, 2}
    \draw[thin, dotted] (-0.3, \y) -- (6.3, \y);
  
  \filldraw (0,0) circle (2pt);
  \filldraw (0.333,1) circle (2pt);
  \filldraw (2.666,2) circle (2pt);
  \draw[fill = white] (2.666,1) circle (2pt);
  \filldraw (1,0) circle (2pt);
  \filldraw (1.333, 1) circle (2pt);

  \filldraw (2,0) circle (2pt);
  \filldraw (2.333,0) circle (2pt);
  \filldraw (3,1) circle (2pt);
  \filldraw (3.333, 0) circle (2pt);
  
  \filldraw (4,1) circle (2pt);
  \filldraw (4.333,0) circle (2pt);
  \filldraw (5,1) circle (2pt);
  \filldraw (5.333, 0) circle (2pt);
  \filldraw (6, 1) circle (2pt);
  
  \foreach \x in {1, 2, 2.333, 3.333, 4.333, 5.333}
    \draw[fill = white] (\x,1) circle (2pt);
    
  \foreach \x in {0.333, 1.333, 3, 4, 5, 6 }
    \draw[fill = white] (\x,0) circle (2pt);
  
  \node at (0, -0.5) {\footnotesize $0$};
  \node at (0.333, -0.5) {\footnotesize $\frac{1}{6}$};
  \node at (2.666, -0.5) {\footnotesize $\frac{4}{3}$};
  \node at (1, -0.5) {\footnotesize $\frac{1}{2}$};
  \node at (1.333, -0.5) {\footnotesize $\frac{2}{3}$};
  \node at (2, -0.5) {\footnotesize $1$};
  \node at (2.333, -0.5) {\footnotesize $\frac{7}{6}$};
  \node at (3, -0.5) {\footnotesize $\frac{3}{2}$};
  \node at (3.333, -0.5) {\footnotesize $\frac{5}{3}$};
  \node at (4, -0.5) {\footnotesize $2$};
  \node at (4.333, -0.5) {\footnotesize $\frac{13}{6}$};
  \node at (5, -0.5) {\footnotesize $\frac{5}{2}$};
  \node at (5.333, -0.5) {\footnotesize $\frac{8}{3}$};
  \node at (6, -0.5) {\footnotesize $3$};
  
  \node at (-0.5, 0) {\footnotesize $0$};
  \node at (-0.5, 1) {\footnotesize $1$};
  \node at (-0.5, 2) {\footnotesize $2$}; 

\end{tikzpicture}
\caption{The function $\mathcal{M}(\cdot, 1)$ for the function $\pFq{3}{2}{\big(\frac 1 6, \frac 2 3, \frac 4 3\big)}{\big(\frac 1 3, \frac 1 2\big)}{x}$ in the interval $[0, 3]$} \label{fig:ChristolM} 
\end{figure}

\begin{figure}
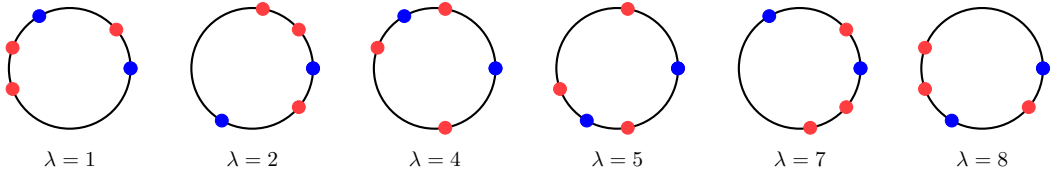

\centering
\BeHe{9}{1/9}{4/9}{5/9}{1/3}{1}
\caption{Christol's interlacing condition depicted on the unit circle for the function $\pFq{3}{2}{\big(\frac 1 9, \frac 4 9, \frac 5 9\big)}{\big(\frac 1 3, 1\big)}{x}$}
\label{fig:interlacing}
\end{figure}

Christol's classification of globally bounded hypergeometric functions reads as follows \cite[Proposition~1]{Chr86}.
\begin{thm}[Christol, 1986]\label{thm:Christol}
The hypergeometric function $f(x)=\pFq{n}{n-1}{\ua}{\ub}{x}$ with $\ua, \ub \in (\Q\setminus(-\N))^{2n-1}
$ is globally bounded if and only is $f(x)$ satisfies Christol's interlacing criterion for all $\lambda\in \Zdc$, where $d$ denotes the common denominator of the parameters $\ua, \ub$.
\end{thm}

\begin{ex}
We consider $f(x)=\pFq{3}{2}{\big(\frac 1 6, \frac 2 3, \frac 4 3\big)}{\big(\frac 1 3, \frac 1 2\big)}{x}$.
Then the functions $\mathcal{M}(\cdot, 1)$ is depicted in Figure~\ref{fig:ChristolM}.
This function is always nonnegative, and easily one also checks that also $\mathcal{M}(\cdot, 5)$ takes no negative values.
Similarly the function $g(x) = \pFq{3}{2}{\big(\frac 1 9, \frac 4 9, \frac 5 9\big)}{\big(\frac 1 3, 1\big)}{x}$
satisfies Christol's interlacing condition for all values of $\lambda \in (\Z/9\Z)^\times$ as shown in Figure~\ref{fig:interlacing}.
Christol's criterion implies that $f$ and $g$ are globally bounded.
\end{ex}

In Subsection~\ref{sssec:proofred} below, we will prove the following extension of Christol's theorem.

\begin{thm} \label{thm:red}
Let $f(x) = \pFq{n}{n-1}{\ua}{\ub}{x}=\sum_{k=0}^\infty h(\ua, \ub; k)x^k$ be a hypergeometric function with rational parameters $\ua, \ub \in (\Q\setminus (-\N))^{2n-1}$. Let $d\in \N$ be the smallest denominator of all parameters $\alpha_i, \beta_j$ and let $p$ be a prime number, coprime to $d$. 
Assume that $p>2d{\cdot}\max\{|\alpha_i|, |\beta_j|, 1\}+1$, and consider the subgroup $\langle p \bmod d \rangle \subseteq \Zdc$. Then $f(x)$ can be reduced modulo $p$ if and only if Christol's interlacing condition is fulfilled for all $\lambda \in \langle p \bmod d \rangle$. For all other $p$, $v_p(h(\ua, \ub; k))$ is not bounded from below.
\end{thm}

\begin{rem}\label{rem:lambda1}
We note in particular, that if Christol's interlacing condition is not fulfilled for $\lambda=1$, then $f(x)$ cannot be reduced modulo any prime number $p$, as $1\in \langle p \bmod d \rangle$ for all $p$.
\end{rem}

\begin{ex}
\label{ex:notreducible}
Let
\[
f(x)={\textstyle \pFq{2}{1}{\big(\frac 1 2 , \frac 2 3\big)}{\big(\frac 1 3\big)}{x}}
\]
At the first glance one might be tempted to think that $f(x)$ can be reduced modulo prime numbers congruent to $5$ modulo $6$, as Christol's interlacing condition in fulfilled for $\lambda = 5$  (see the left-hand part of Figure~\ref{fig:BeHeEx1}).
However, by Remark~\ref{rem:lambda1}, this is not the case.
\end{ex}

\begin{ex}
Consider the hypergeometric function
\[g(x)={\textstyle \pFq{2}{1}{\big(\frac 1 4, \frac 1 2\big)}{\big(\frac {15} 4\big)}{x}}= 1 + \frac{1}{30} x + \frac{1}{152} x^2 + \frac{15}{6992}x^3+\ldots \]
For large enough primes $p$ that are congruent to $3$ modulo $4$ the function $g(x)$ cannot be reduced modulo $p$, while reduction is possible for almost all prime numbers congruent to $1$ modulo $4$, as the right-hand part of Figure~\ref{fig:BeHeEx1} shows.
However, taking $p=5$, we see that already the coefficient of $x$ in $g(x)$ is $1/30$, which cannot be reduced modulo $5$. This does not contradict Theorem~\ref{thm:red}, as $5<2\cdot 4\cdot  15/4 + 1 = 31$. 
\end{ex}

\begin{figure}
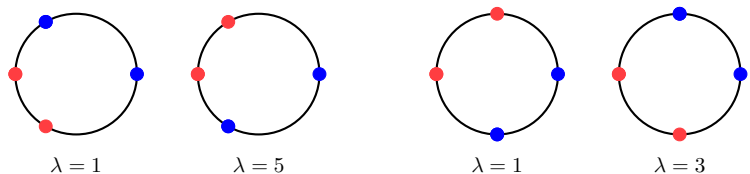

\centering
\BeHe{6}{1/2}{1/2}{2/3}{1/3}{1/3} \hspace{1cm} \BeHe{4}{1/4}{2/4}{2/4}{3/4}{3/4}
\caption{Christol's interlacing condition illustrated on the unit circle for
the function $f(x)=\pFq{2}{1}{\big(\frac 1 2 , \frac 2 3\big)}{\big(\frac 1 3\big)}{x}$ (on the left) and
the function $g(x)=\pFq{2}{1}{\big(\frac 1 4, \frac 1 2\big)}{\big(\frac {15} 4\big)}{x}$ (on the right)} \label{fig:BeHeEx1}
\end{figure}

We also mention here a similar criterion for the algebraicity of hypergeometric functions. It was formulated by Christol \cite[Proposition~3; p.15 Corollary]{Chr86} in the same paper as the criterion for global boundedness, although his proof relies on the validity of the $p$-curvature conjecture for hypergeometric differential equations, which follows from Katz' work \cite{Kat72}, as elaborated in \cite{Kat90}. Shortly after, it was independently stated by Beukers and Heckman \cite[Theorem~4.8]{BH89} with a more or less elementary proof, at the same time giving an elementary proof of the Grothendieck $p$-curvature conjecture for hypergeometric equations. 

\begin{thm}[Christol, 1986; Beukers--Heckman, 1989; Katz, 1990]
\label{thm:BH}
Let $f(x)=\pFq{n}{n-1}{\ua}{\ub}{x}$ with $\ua, \ub \in (\Q\setminus(-\N))^{2n-1}$ be a hypergeometric function and assume that $\alpha_i-\beta_j, \alpha_i\not \in \Z$ for all $i,j$.  Denote by $d$ the common denominator of the parameters $\ua, \ub$  and set $\beta_n=1$. Then $f(x)$ is algebraic over $\Q(x)$ if and only if for all $j=1,\ldots, n$ and all $\lambda\in \Zdc$,  we have
\[\mathcal{M}(\lambda \beta_j, \lambda)= 0.\]
\end{thm}  

In the interpretation on the unit circle, this means that the sets $\{\exp(2\i\pi \lambda \alpha_i):1\leq i\leq s\}$ and $\{\exp(2\i\pi \lambda \beta_j):1\leq j\leq s\}$ are interlaced on the unit circle.
We also underline that the criterion of Theorem~\ref{thm:BH} only works for hypergeometric functions with rational parameters without integer differences. A complete classification of algebraic hypergeometric functions with arbitrary complex parameters is given in \cite{FY24}. 

\begin{ex}\label{ex:gbalg}
The hypergeometric function $\pFq{3}{2}{\big(\frac 1 9, \frac 4 9, \frac 5 9\big)}{\big(\frac 1 3, 1\big)}{x}$ is not algebraic (although it is globally bounded), as it can be seen from the graphics in Figure~\ref{fig:interlacing}.
\end{ex}

By standard results from differential Galois theory, it is equivalent that a hypergeometric function $\pFqab$ is algebraic and that the monodromy group of its minimal differential operator, which is given by $\Hyp(\ua, \ub)$ under the assumption $\alpha_i-\beta_j\not \in \Z$, is finite.
The following result, appearing in Levelt's dissertation \cite{Lev61} (and there attributed to N. G. de Bruijn), see also \cite[Theorem~3.5]{BH89}, gives an explicit realization of the monodromy group of a hypergeometric differential equation as a group of matrices.
\begin{thm}
\label{thm:monodromy}
    We assume that there are no integer differences between the sets $\ua, \ub\in (\Q\setminus (-\N))^{n}$ with $\beta_n=1$ of parameters of a hypergeometric operator $\Hyp(\ua, \ub)$. Let $A_0,\ldots, A_{n-1}$ and $B_0,\ldots, B_{n-1}$ be defined by the identities
    \begin{align*}
        A(x)\coloneqq \prod_{i=1}^n (x-\exp(2\i\pi \alpha_i))&=x^n + A_{n-1}x^{n-1}+\cdots +A_0\\
        B(x)\coloneqq \prod_{j=1}^n (x-\exp(2\i\pi \beta_j))&=x^n + B_{n-1}x^{n-1}+\cdots +B_0.
    \end{align*}
    We define the matrices 
    \[A=\left(\begin{matrix}
        0 & 0 & \cdots & 0 & -A_0\\
        1 & 0 & \cdots & 0 & -A_1\\
        0 & 1 & & 0 & -A_2\\
        \vdots & & \ddots &  & \vdots\\
        0 & 0 & & 1 & -A_{n-1}
    \end{matrix} \right),\qquad
    B=\left(\begin{matrix}
        0 & 0 & \cdots & 0 & -B_0\\
        1 & 0 & \cdots & 0 & -B_1\\
        0 & 1 & & 0 & -B_2\\
        \vdots & & \ddots &  & \vdots\\
        0 & 0 & & 1 & -B_{n-1}
    \end{matrix} \right).\]
Then $A$ is the local monodromy matrix of $\Hyp(\ua, \ub)$ at infinity, $B^{-1}$ is the local monodromy matrix at $0$ and together they generate the global monodromy group of $\Hyp(\ua, \ub)$.
\end{thm}

\subsubsection{Proof of Theorem~\ref{thm:red}}
\label{sssec:proofred}

We start by recalling some technical results of Christol from his proof of Theorem~\ref{thm:Christol}.
We write $f(x)=\pFqab=\sum_k h(\ua, \ub; k) x^k$ and shorthand $h(\ua, \ub; k)\eqqcolon h_k$. We set $\beta_n=1$. Let $q$ be coprime to $d$ and let $\Delta \in \Z$ be such that $\Delta q\equiv 1\pmod d$. Then we define 
\begin{equation}\label{eq:defV}
    V(x,q)\coloneqq \left|\left\{1\leq i \leq n: \dpc{\Delta \alpha_i}-\frac{\alpha_i}{q}<x\right\}\right|-\left|\left\{1\leq j \leq n: \dpc{\Delta \beta_j}-\frac{\beta_j}{q}<x\right\}\right|.
\end{equation}
Christol \cite[p.6., Equation~(6)]{Chr86} proves the following on the $p$-adic valuation of the coefficients of a hypergeometric function.
\begin{lem}[Christol]\label{lem:peval}
We have $v_p(h_k)={\displaystyle \sum_{r=1}^\infty ~} V\left(\dpzbig{\frac{k}{p^r}}, p^r\right)$.
\end{lem}

For each $\alpha\in \Z_p$ and $q=p^r$ for some positive integer $r$, we uniquely write $\alpha=qQ(\alpha, q)-R(\alpha,q)$ with $Q(\alpha, q)\in \Z_p$ and $0\leq R(\alpha,q)<q$. 
Christol also proves the following about the quantities defining $V$ \cite[Lemma~4]{Chr86}:
\begin{lem}[Christol] \label{lem:resdec}
Let $\gamma\in \frac{1}{d}\Z\setminus(-\N)$, let $p$ be a prime number not dividing $d$. Let $r$ be a positive integer such that $p^r>d{\cdot}|\gamma|$ and let $\Delta$ be such that $\Delta p\equiv 1\pmod d$. Then we have
\[\frac{R(\gamma, p^r)}{p^r}=\dpc{\gamma \Delta^r}-\frac{\gamma}{p^r}.\]
\end{lem}

The following two Lemmata are implicitly part of Christol's proof of the criterion for global boundedness of hypergeometric functions.
\begin{lem}[Christol] \label{lem:smallx}
Let $0<x<\frac{1}{2d}$ and let $p$ be a prime number not dividing $d$. We choose $\Delta$ such that $\Delta p \equiv 1 \pmod d$ and let $r$ be a positive integer such that $p^r>2d{\cdot}\max\{|\alpha_i|, |\beta_j|\}$. Then
\begin{enumerate}[label=(\roman{enumi})]
\item we have $V\left(x, p^r\right)=0$, and
\item for all $j\in \{1,\ldots, n\}$, we have 
\[V\left(\frac{R(\beta_j, p^r)}{p^r}+x, p^r\right)=\mathcal{M}(\beta_j\Delta^r, \Delta^r).\]
\end{enumerate}
\end{lem}
\begin{proof}
For~(i), we note that 
\[\dpc{\gamma\Delta^r}-\frac{\gamma}{p^r}>\frac{1}{d}-\frac{1}{2d}>\frac{1}{2d}\]
for all $\gamma\in \{\alpha_1, \ldots, \alpha_n, \beta_1, \ldots, \beta_n\}$, as $\dpc{\gamma\Delta^r} \in \frac{1}{d}\N_{>0}$. Therefore $V(x, p^r)=0$.

For~(ii), we write $\xi_j = \dpc{\beta_j\Delta^r}-\frac{\beta_j}{p^r}$ and notice, as Christol, that 
\[\lim_{\mathclap{x \to \xi_j^+}}V(x, p^r)=\mathcal{M}(\Delta^r \beta_j, \Delta^r).\]
Indeed, for all $\gamma, \delta \in \{\alpha_1, \ldots, \alpha_n, \beta_1, \ldots, \beta_n\}$,
the condition $\dpc{\gamma\Delta^r}-\frac{\gamma}{p^r} \leq \dpc{\delta\Delta^r} - \frac{\delta}{p^r}$ is equivalent to $\gamma \Delta^r \preceq \delta \Delta^r$
because of the bound on $p^r$.
Moreover, by Lemma~\ref{lem:resdec}, we know that 
\[\frac{R(\beta_j, p^r)}{p^r}=\dpc{\beta_j \Delta^r}-\frac{\beta_j}{p^r}.\]
Finally, it follows as in the proof of~(i) that the function $V(\cdot, p^r)$ is constant on the interval $\left( \frac{R(\beta_j, p^r)}{p^r}, \frac{R(\beta_j, p^r)}{p^r}+\frac{1}{2d}\right)$.
\end{proof}
\begin{lem}[Christol] \label{lem:Vpos}
We assume $p^r>2d{\cdot}\max(|\alpha_i|, |\beta_j|)$ is coprime with $d$, we choose $\Delta$ such that $p\Delta\equiv 1\pmod d$ and assume that $f(x)$ satisfies Christol's interlacing condition for $\Delta^r$. Then $V(x, p^r)\geq 0$ for all $x\in \Q$.
\end{lem}
\begin{proof}
The function $V(\cdot, p^r)$ is piecewise constant and its jumps of negative height happen at the values $\dpc{\beta_j\Delta^r}- \beta_j/p^r$. By the assumption and Lemma~\ref{lem:smallx}.(ii), the function $V(\cdot, p^r)$ is still positive after these jump points.
\end{proof}

We are now ready to complete the proof of Theorem~\ref{thm:red}.
We assume that Christol's interlacing criterion is fulfilled for all $\lambda \in \langle p \bmod d \rangle$. Then, by Lemma~\ref{lem:peval}, we have
\[v_p(h_k)=\sum_{r=1}^\infty V\left({\textstyle \dpzbig{\frac{k}{p^r}}}, p^r\right)\]
for all $k$.
From Lemma~\ref{lem:Vpos}, we see that each of the summands is nonnegative, and thus $v_p(h_k)\geq 0$ for all~$k$. 
Conversely, let us assume that there is $\lambda\in \langle p \bmod d \rangle$ such that Christol's interlacing condition is not fulfilled. We denote by $m$ the minimal positive integer such that $\lambda\equiv \Delta^m \pmod d$ and set $\ell\coloneqq \operatorname{ord}_d(p)=|\langle p \bmod d \rangle|$. We choose $j$ in such a way that 
\[\mathcal{M}(\Delta^m \beta_j, \Delta^m)=\big|\{1\leq i \leq n: \Delta^m \alpha_i\preceq \Delta^m \beta_j \}\big|-\big|\{1 \leq i \leq n: \Delta^m \beta_i\preceq \Delta^m \beta_j\}\big|<0.\]
For any $a\in \N$, we will construct $k$ such that $v_p(h_k)\leq -a$.
To do this, we first set $k_{m+(a-1)\ell} \coloneqq R(\beta_j, p^{m+(a-1)\ell})$.
Successively for $r=m+(a-1)\ell-1, m+(a-1)\ell-2, \ldots, 1$, we then consider the unique integer $d_r\in \{0, 1,\ldots, p^{r-1}\}$ such that 
\[k_{r}+d_{r-1}\equiv \begin{cases}0 \pmod{p^{r-1}}& \text{if } r \not \equiv m\!\!\pmod \ell \\ R(\beta_j, p^{r-1}) \pmod{p^{r-1}}& \text{otherwise,} \end{cases}\]
and set $k_{r-1} \coloneqq k_r+d_{r-1}$. We finally take $k \coloneqq k_1$.
For each $r$, we then have $d_{r-1}+\cdots+d_0\leq p^{r-1}+\cdots + p+1=\frac{p^{r}-1}{p-1}< \frac{p^r}{2d}$, which ensures that $k$ satisfies
\[k\in  \begin{cases}\left[R(\beta_j,p^r)+1, R(\beta_j,p^r)+\frac{p^{r}}{2d}\right) \pmod{p^r} & \text{if } r <\ell \cdot a \text{ and } r \equiv m\!\! \pmod \ell  \\[10pt] \left[1,\frac{p^{r}}{2d}\right) \pmod{p^r}& \text{otherwise.}  \end{cases}\]
Thus, we get $\dpzbig{\frac{k}{p^r}} <\frac{1}{2d}$ for all $r\not \equiv m\pmod \ell$ and all $r>a \ell$.
For $r=m, m+\ell, \ldots, m+(a-1)\ell$  we have $\dpzbig{\frac{k-R(\beta_j, p^r)}{p^r}} \leq \frac{1}{2d}$, and therefore, by Lemma~\ref{lem:smallx}, 
\[V\left({\textstyle \dpzbig{\frac{k}{p^r}}}, p^r \right)= \begin{cases} \mathcal{M}(\Delta^r \beta_j, \Delta^r)  & r=m, m+\ell, \ldots, m+(a-1)\ell\\ 0 & \text{otherwise.}\end{cases}\]
Finally, using Lemma~\ref{lem:peval}, we compute
\begin{align*}
v_p(h_k) =\sum_{r=1}^\infty V\left({\textstyle \dpzbig{\frac{k}{p^r}}}, p^r\right) =a \mathcal{M}(\Delta^m \beta_j, \Delta^m)\leq -a.
\end{align*}

\begin{rem} 
For small prime numbers $p$ it is enough to check the $p$-adic evaluation of finitely many coefficients to conclude that the hypergeometric function might be reduced properly modulo~$p$. Indeed, let $r_0$ be such that $p^{r_0}>2d{\cdot}\max\{|\alpha_i|, |\beta_j|, 1\}+1$. Then, for all $n$, we have
\[
v_p(h_k) =\sum_{r=1}^\infty V\left({\textstyle \dpzbig{\frac{k}{p^r}}}, p^r\right)
=\sum_{r=1}^{r_0-1} V\left({\textstyle \dpzbig{\frac{k}{p^r}}}, p^r\right) + \sum_{r=r_0}^{\infty} V\left({\textstyle \dpzbig{\frac{k}{p^r}}}, p^r\right)
\]
where the right summand is always positive, if Christol's interlacing condition is fulfilled for all $\lambda \in \langle p \bmod d \rangle$, and the left summand is periodic in $k$ with period $p^{r_0-1}$.
\end{rem}

\begin{ex}
Continuing Example~\ref{ex:notreducible}, that is
\[
f(x)={\textstyle \pFq{2}{1}{\big(\frac 1 2 , \frac 2 3\big)}{\big(\frac 1 3\big)}{x}}=\sum_{n\geq 0}h_n x^n,
\]
we can construct explicitly an integer $n$ such that the denominator of $h_n$ is divisible by $17^2$ using the strategy of the proof of Theorem~\ref{thm:red}.
We have $a=2, \ell=2$ and $m=2$ and use $\beta_j=1/3$. With this, we compute $k_4=27840, k_3 = 29478, k_2=29574$ and $k_1=29580$. And, indeed, checking with a computer algebra system, we convince ourselves that the $17$-adic evaluation of $h_{29580}$ is~$-2$.
\end{ex}

\subsubsection{Solutions of the hypergeometric differential equation modulo $p$}
\label{sec:solp}

We now discuss the question of determining the dimension of the space of solutions of a hypergeometric differential equation modulo $p$ in $\F_p\ps{x}$ when $p$ does not divide $d$.
It turns out that this dimension is also encoded by the relative position of the exponentials of $\alpha_i$ and $\beta_j$ on the unit complex circle;
the slight difference between what precedes is that we will not work with the reductions of $\alpha_i$ and $\beta_j$ modulo $\Z$, but modulo $p$.
More precisely, for each index $i \in \{1, \ldots, n\}$, we consider the complex number
\[\alpha_{i,p} \coloneqq \zeta_{2p} \cdot \zeta_p^{-\alpha_i \bmod p} = \exp\left(2\i\pi \cdot \left(\frac 1 {2p} - \frac{\alpha_i \bmod p} p\right)\right)\]
and color it in green. We underline that the previous definition is not ambiguous because the fraction $(\alpha_i \bmod p)/p$ is well-defined modulo $\Z$.
In concrete terms, $\zeta_p^{-\alpha_i \bmod p} = \zeta_p^{-\alpha_{i,p}}$ where $\alpha_{i,p}$ is an \emph{integer} which is congruent to $\alpha_i$ modulo $p$.
Similarly, we color in yellow the points 
\[\beta_{j,p} \coloneqq \zeta_p \cdot \zeta_p^{-\beta_j \bmod p} = \exp\left(2\i\pi \cdot \left(\frac 1 p - \frac{\beta_j \bmod p} p\right)\right)\]
for $j$ varying in $\{1, \ldots, n\}$ as well. Since $\beta_n = 1$, the point $1 \in \C$ is always colored in yellow.

\begin{defi}
The \emph{$p$-interlacing number} of $(\ua, \ub)$ is the number of times we observe a green point followed by a yellow one (or, equivalently, a yellow one followed by a green) when running clockwise through the unit circle.
\end{defi}

The $p$-interlacing number will allow us to answer Question~\ref{quest:pcurv} for hypergeometric equations. 
The proof of the following result can be easily deduced from Katz' results on the $p$-curvature of a hypergeometric operator $\Hyp=\Hyp(\ua, \ub)$~\cite[Proof of Sublemma 5.5.2.1, p.~174~f.]{Kat90}.
Below, we reformulate it in a more down-to-earth fashion, avoiding the use of the $p$-curvature.

\begin{thm}  \label{thm:block}
Let $\mathcal{T}_p(\ua, \ub)$ denote the set of exponents $t$ corresponding to a yellow point for which there is a green point before we reach the next yellow point when running clockwise on the unit complex circle.
For $t\in \mathcal{T}_p(\ua, \ub)$, we denote by $k_t$ the exponent corresponding to the next green point. We set
\begin{equation}
\label{eq:ft}
    f_t(x)\coloneqq \sum_{k=t}^{k_t-1}h(\ua, \ub; k)\cdot p^{-v_p(h(\ua, \ub; t))}x^k.
\end{equation}
The family $(f_t(x))_{t \in \mathcal{T}_p(\ua, \ub)}$ is a $\F_p(x^p)$-basis of the space of solutions of $\Hyp(\ua,\ub)$.

In particular, the dimension of this vector space is given by the $p$-interlacing number of $(\ua, \ub)$.
\end{thm}

\begin{proof}[Proof (after Katz)]
We view $\F_p(x)$ as a $p$-dimensional vector space over $\F_p(x^p)$. The operator $\Hyp(\ua, \ub)$ is a linear endomorphism of this vector space. It is determined by $\Hyp(\ua, \ub)(x^k)= -A(k)x^{k+1}+B(k)x^{k}$, where $A(k)=\prod_{i=1}^n (k+\alpha_i)$ and $B(k)=\prod_{j=1}^n (k+\beta_j-1)$.
We chose $m\in \F_p$, such that $m\equiv 1-\alpha_i\pmod p$ for some $i\in \{1,\ldots, n\}$. The matrix of this linear map in the basis $(x^m, x^{m+1}, \ldots, x^{m+p-1})$ is of the form
\[\left(\begin{matrix}
    B(m) & 0 & 0 & \cdots & 0 & -A(m+p-1)\\
    -A(m) & B(m+1) & 0 & \cdots & 0 & 0\\
    0 & -A(m+1) & B(m+2) & & 0 & 0\\
    \vdots &  & \ddots & \ddots &  & \vdots\\
    0 & 0 & 0 &\ddots & B(m+p-2) & 0\\
    0 & 0 & 0 & & -A(m+p-2) & B(m+p-1)
\end{matrix}\right), \]
where the top right entry is zero, as $A(m+p-1)=0$ by definition. The matrix decomposes into a block diagonal form, each zero of $A$ separating two blocks. Each of the blocks is a lower diagonal matrix
with consecutive values of $B$ on the main diagonal, and consecutive values of $A$, all nonzero, on the diagonal below. The co-rank of this matrix is at most one, as the bottom left minor of order one less is non-zero. Moreover as the blocks are lower triangular, the co-rank is zero is and only if all entries on the main diagonal are different from zero.

Consider such a block of co-rank one and let $b\coloneqq 1-\beta_j \pmod p$ be the last zero of $B$ in this block, and let $a\coloneqq -\alpha_{i'}\pmod p$ the following zero of $A$. It is then easy to convince oneself that 
\[v_b\coloneqq \left(\begin{matrix}
    0 &
    \cdots &
    0 &
    1 &
    \frac{A(b)}{B(b+1)} &
    \frac{A(b)A(b+1)}{B(b+1)B(b+2)} &
    \cdots &
    \frac{A(b)\cdots A(a-1)}{B(b+1)\cdots B(a)}
\end{matrix}\right)^\top\]
is a well-defined vector in the kernel of this block. It corresponds to the function $h(\ua, \ub; b)^{-1}f_b(x)$. The proposition follows from these observations.
\end{proof}

\begin{rem}\label{rem:hypergeommodpxp}
    If we assume that $f(x)=\pFqab$ can be reduced modulo $p$, then the truncation at $x^p$ of $f(x) \bmod p$ is a linear combination of the polynomials $f_t(x)$.
    Those $t$, for which the polynomials $f_t(x)$ have nonzero coefficients in this combination, are precisely the elements of the set 
    \begin{equation} \label{eq:Sp}
        \mathcal S_p(\ua, \ub)\coloneqq\Big \{r\in \{0, \ldots, p-1\}:r\equiv 1{-}\beta_j \!\!\!\!\! \pmod p \text{ for some $j$ and } v_p\big( h(\ua, \ub; r)\big)=0\Big \}.
    \end{equation}
\end{rem}

\begin{rem}
    In \cite{Dwo90}, Dwork constructed a differential field extension of $\F_p[x]$ defined by considering countably many variables $z_i$ with $z_i'=\frac{z_{i-1}'}{z_{i-1}}$, resembling the derivation rule of iterated logarithms in characteristic $0$. Dwork's construction provides a full basis of $n$ solutions of differential equations with nilpotent $p$-curvature, such as the hypergeometric differential equation. Indeed, the Chudnovsky's showed in \cite{CC85} that minimal differential operators of $G$-functions, such as hypergeometric functions ${}_nF_{n-1}$ with rational parameters and no integer differences between them \cite{Gal81}, are globally nilpotent, \emph{i.e.}, they have $p$-curvature $0$ for almost all prime numbers. In \cite{FH23}, Dwork's ideas are extended to arbitrary regular singular differential equations with coefficients in $\F_p\ps{x}$, and an algorithm for explicitly computing solutions in the extension is depicted. Applying this to the hypergeometric differential equation $\Hyp(\ua, \ub)y=0$, we find a basis of $n$ linearly independent polynomial solutions of the form 
    $f_t(x)=\sum_{k=t}^{k_t-1}g_{k,t}(z_1, z_2,\ldots)x^k \in \F_p[x, z_1, z_2,\ldots]$ with $k_t\leq 2p$ and $g_{k,t}(z_1, z_2,\ldots)\in \F_p[z_1, z_2, \ldots]$ for all $t=1-\beta_j\bmod p.$
\end{rem}

We next discuss how the number of solutions of the reduction of a fixed hypergeometric differential equation $\Hyp(\ua, \ub) y=0$ with rational coefficients modulo $p$ varies with the prime number $p$. Let $d$ denote again the common denominator of the parameters $\ua, \ub\in (\Q\setminus (-\N))^{2n-1}$ of $\Hyp(\ua, \ub)$. 

\begin{prop} \label{prop:solmodd}
Let $p>2d{\cdot}\max\{|\alpha_i|, |\beta_j|, 1\}$ and choose $\Delta$ such that $p\Delta\equiv 1 \pmod d$. Then the $p$-interlacing number of $(\ua,\ub)$ is equal to the number of local minima in the cyclic sequence $(\mathcal{M}(\gamma_k \Delta, \Delta))_{k=1}^{2n}$, where $\{\gamma_k\}=\{\Delta \alpha_i\}\cup \{\Delta \beta_j\}$ and $\gamma_1\prec \gamma_2\prec \ldots \prec\gamma_{2n}$.
In particular, the number of solutions of the hypergeometric differential equation in $\F_p\ps{x}$ for large characteristics $p$ depends only on the congruence class of $p$ modulo $d$.
\end{prop}
\begin{proof}
We know from Lemma~\ref{lem:resdec} that for $p>2d{\cdot}\max\{|\alpha_i|, |\beta_j|, 1\}$ we have
\[\frac{R(\alpha_i, p)}{p}=\dpc{\alpha_i\Delta}-\frac{\alpha_i}{p}.\]
Moreover we have seen in the proof of Lemma~\ref{lem:smallx} that
\begin{equation}
\label{eq:equivdpc}
\textstyle
\dpc{\gamma\Delta}-\frac{\gamma}{p}<\dpc{\delta\Delta}-\frac{\delta}{p} \quad \Longleftrightarrow \quad \gamma\Delta\preceq \delta\Delta
\qquad \big(\gamma,\delta \in \frac 1 d \Z\big).
\end{equation}
This shows that the (negative) residues $R(\gamma_k, p)$ of the parameters $\gamma_k$ modulo $p$ are ordered in the same way as the decimal parts $\dpc{\gamma_k\Delta}$. 
Finally, a block of consecutive elements colored in yellow (resp. green), corresponds to consecutive zeroes of $A(t)$ (resp. $B(t)$) in $\F_p$, which then correspond to a block of increasing, respectively decreasing, sequence elements of the sequence $(\mathcal{M}(\gamma_k \Delta, \Delta))_{k=1}^{2n}$, which in turn is equal to the number of local minima attained in this sequence.
\end{proof}

\begin{rem}
Assuming that $\alpha_i-\beta_j\not \in \Z$, Proposition~\ref{prop:solmodd} can be visualized geometrically on the unit circle.
Indeed, under the assumptions of the proposition, the equivalence~\eqref{eq:equivdpc} shows that 
the points $\alpha_{i,p},\, \beta_{j,p} \,\, (1 \leq i, j \leq n)$
are ordered in the same way on the unit complex circle than the points $\exp(2\i\pi\Delta\alpha_i),\, \exp(2\i\pi\Delta\beta_j) \,\, (1 \leq i, j \leq n)$.
\end{rem}

\begin{ex}
The hypergeometric function $f(x)=\pFq{3}{2}{\big(\frac 1 8, \frac 3 8, \frac 1 2\big)}{\big(\frac 1 4, \frac 5 8\big)}{x}$ is not globally bounded. Christol's interlacing condition is fulfilled for $\lambda=1, 3$, but not for $\lambda= 5 ,7$. Therefore, according to Theorem~\ref{thm:red}, $f(x)$ can be reduced modulo almost all prime numbers congruent to $1$ and $3$ modulo $8$, but almost all prime numbers congruent to $5$ or $7$ modulo $8$ appear up to arbitrarily high powers in the denominators of the coefficients of $f(x)$. Moreover, the $\F_p\ps{x^p}$-dimension of the solutions of the hypergeometric differential equation 
\[\textstyle \Hyp\big(\big(\frac 1 8, \frac 3 8, \frac 1 2\big),\big(\frac 1 4, \frac 5 8\big)\big)y=0\]
in $\F_p\ps{x}$ is for sufficiently large primes $p$, according to Theorem~\ref{thm:block} and Proposition~\ref{prop:solmodd}, two if $p\equiv1, 7\pmod 8$ and one otherwise. All this information can be read off from Figure~\ref{fig:dimsolEx}.

\begin{figure}
\centering
\BeHe{8}{1/8}{3/8}{4/8}{2/8}{5/8}
\caption{The sets $\{\exp(2\i\pi \lambda \alpha_i)\}$ and $\{\exp(2\i\pi \lambda \beta_j)\}$ for $\pFq{3}{2}{\big(\frac 1 8, \frac 3 8, \frac 1 2\big)}{\big(\frac 1 4, \frac 5 8\big)}{x}$
} \label{fig:dimsolEx}
\end{figure}
\end{ex}

\subsection{Algebraicity modulo $p$ and Galois groups}
\label{ssec:algmodp}

In \cite{Var24}, polynomials annihilating reductions of hypergeometric functions modulo $p$ were constructed. 
In this subsection, we recall the methods of \emph{loc. cit.} and develop them further in order to get annihilating polynomials of smaller sizes.
This will be important for the computations of the Galois groups later on.

\subsubsection{The Dwork map}

A main tool in the construction is the Dwork map $\Dp : \Zp\to \Zp,$ defined by Euclidean division: 
as before, we write $\gamma\in \Zp$ uniquely as $\gamma=pQ(\gamma, p) - R(\gamma, p)$ with $Q(\gamma, p)\in \Zp$ and $0\leq R(\gamma,p)<p$.
We set $\Dp(\gamma)\coloneqq  Q(\gamma, p)$. In other words, $\Dp$ is uniquely defined by the relation \[p \cdot \Dp(x) - x \in \{0, 1, \ldots, p{-}1\}.\]
We collect a few facts about $\Dp$.

\begin{lem}
\label{lem:dworkmap}
\begin{enumerate}[label=(\arabic{enumi})]
    \item  For $\gamma \in \Zp$, let
    $$-\gamma = \gamma^{(0)} + p \gamma^{(1)} + p^2 \gamma^{(2)} + \cdots$$
    be the $p$-adic expansion of $-\gamma$. Then, for any integer $k\geq1$, $$-\Dp^k(\gamma)=\gamma^{(k)}+p\gamma^{(k+1)}+p^2\gamma^{(k+2)}+\cdots.$$

    \item The Dwork operator $\Dp$ acts on $\Zp/\Z$; more precisely, for any fixed $d\in \N$ not divisible by $p$, it acts on $(\frac{1}{d} \Z)/\Z$ by multiplication with $p^{-1}$.

    \item The Dwork operator $\Dp$ maps $(0,1]\cap \frac{1}{d}\Z$ to itself.

    \item For $\gamma \in (0,1]\cap \Z^{\times}_{(p)}$, the smallest positive integer $\ell$ such that $\Dp^{\ell}(\gamma)=\gamma$ is equal to the smallest positive integer $\ell$ such that $p^\ell\gamma\equiv \gamma\pmod \Z$, that is, if $\gamma = \frac c d$ with $\gcd(c,d)=1$, $\ell$ is given as the order of $p$ modulo $d$.
\end{enumerate}
\end{lem}

\begin{proof}
The first assertion follows directly from the definition when $k = 1$. For general $k$, it follows by induction.
For~(2), we write $\gamma=\frac c d$ with $\gcd(c,d) = 1$.
We then have $\frac c d=p\Dp(\gamma) - R(\gamma, p)$ with $R(\gamma, p)\in \{0,1,\ldots, p-1\}$, so clearly $\Dp(\gamma)\in \frac{1}{d} \Z$, say $\Dp(\gamma)=\frac e d$. Then $c\equiv p e \pmod d$, so $e$ only depends on $c\bmod d$ and $\Dp$ indeed acts by multiplication with $p^{-1}$. 
For~(3), keeping the previous notation, we note that $e>p$ implies $c=pe - R(\gamma, p)d$ and then $c>d$, contradicting $0<c\leq d$. Likewise, if $e\leq 0$, then $c\leq 0$, also a contradiction.
The last point is obvious from what precedes.
\end{proof}

    For $\gamma \in \Zp$, let $\ell_p(\gamma)$ denote the smallest positive integer $\ell$ such that $p^\ell \gamma \equiv \gamma \pmod \Z$ as in Lemma~\ref{lem:dworkmap}.(4).
    More generally, given an arbitrary sets of parameters $\ua, \ub \in \Zp^{n+(n-1)}$, we define $\ell_p(\ua, \ub)$ as the smallest positive integer such that $\Dp^{\ell_p(\ua, \ub)}(\ua)\equiv \ua\pmod \Z$ and $\Dp(\ub)^{\ell_p(\ua, \ub)}\equiv \ub\pmod \Z$, where these equations are to be read as equalities of sets. In particular, $\ell_p(\ua, \ub)$ always divides the order of $p$ in $(\Z / d\Z)^\times$, where $d$ denotes the common denominator of the parameters. However, in general there is no equality, as demonstrated by the following example. 
    \begin{ex}
        We have $\ell_{13}((\frac 1 7, \frac 6 7), (\frac 1 2))=1$, as $\D_{13} (\frac 1 7)= \frac 6 7$, $\D_{13} (\frac 6 7)= \frac 1 7$ and $\D_{13} (\frac 1 2)= \frac 1 2$. At the same time we have $\operatorname{ord}_{14}(13)=2$.
    \end{ex}

    When the parameters $\ua, \ub$ are clear from the context, we will suppress them in the notation and write simply $\ell_p$ for $\ell_p(\ua, \ub)$.

\subsubsection{The hypergeometric relation graph}

We are going to construct a graph, whose vertices correspond to parameters of hypergeometric functions, and whose edges will encode algebraic relations between those functions.
We denote by 
\[\mathcal V_p \coloneq \big(\Zp \setminus(-\N)\big)^{n + (n-1)}\]
the set of vertices corresponding to all possible parameters.
Moreover, we let
\[\Vp(\ua, \ub)\coloneqq \big\{(\ua', \ub')\in \Vp\,\,: \,\, \alpha_i\equiv \alpha_i'\!\!\!\! \pmod \Z, \,\, \beta_j\equiv \beta_j'\!\!\!\!\pmod \Z\big\}.\]
These sets form a partition of $\Vp$.

    We define the \emph{hypergeometric relation graph} $\mathcal{G}_p$ as the following directed graph with labeled edges. The vertices are given by $\Vp$. For the edges we carry out the following construction. For $\ug\in \Zp^m$ and let $r\in \N$, we put
    \[\Dpr(\gamma) \coloneqq \begin{cases} \Dp(\gamma)+1 & \text{if $(\gamma)_r\in p\Zp$ }\\  \Dp(\gamma) & \text{otherwise.}\end{cases}\]
For every $s\in\mathcal{S}_{p}(\ua,\ub)$, as defined in Equation~\eqref{eq:Sp}, we connect $(\ua, \ub)\in \Vp$ with $\D_{p,s}(\ua, \ub)\in \Vp$ with a directed edge $e : (\ua,\ub) \to \D_{p,s}(\ua,\ub)$ labeled by the polynomial $Q(e)\coloneqq f_s(x)$ as defined by Equation~\eqref{eq:ft}.

We define $\mathcal{G}_p(\ua, \ub)$ as the sub-graph, of $\mathcal{G}_p$ containing $(\ua, \ub)$, and all vertices reachable by an oriented path, starting at $(\ua, \ub)$.

    \begin{lem} \label{lem:Dpsmall}
    \begin{enumerate}[label=(\arabic{enumi})]
        \item Let $\gamma=\frac c d\in \Zp\setminus (-\N)$, where $p$ is a prime number with $p>2d{\cdot}|\gamma|$. Then $\Dp(\gamma)\in (0,1]$ 
        \item Let $r \in \N$ and $M\geq \frac{2p-1}{p-1}.$ If $|\gamma|\leq M$, then $|\Dpr(\gamma)|\leq M$.
    \end{enumerate}
        
    \end{lem}
    \begin{proof}
   
    For (1), we first assume by contradiction that $\Dp(\gamma)>1$. We have $\Dp(\gamma)\in \frac 1 d \Z$, and so
    \[p\Dp(\gamma)-\gamma\geq (p+p/d)-\gamma\geq p, \] contradicting the definition of $\Dp$. Similarly, if $\Dp(\gamma)<0$, we obtain 
    \[p\Dp(\gamma)-\gamma\leq (-p/d)-\gamma <0,\] again contradicting the definition of $\Dp$.
    Finally, if $\Dp(\gamma)=0$, we must have $-\gamma\in \{0,1,\ldots, p-1\}$, a contradiction to our assumption.

    For the second part, we have $p\Dp(\gamma)\in \gamma+\{0,1,\ldots, p-1\}$. So $-M\leq p\Dp(\gamma)\leq \gamma + p-1\leq M+p-1$, and therefore $-M\leq -M/p\leq \D_{p,r}(\gamma)\leq \Dp(\gamma) \leq \frac{M+p-1}{p}+1\leq M$.
\end{proof}
    
We conclude that, for each choice of $(\ua, \ub)$, the graph $\mathcal{G}_p(\ua, \ub)$ is a finite graph. This follows from the fact that for large enough $M$ the finite set $\frac{1}{d}[-M,M]$ is mapped into itself by $\Dpr$, as shown by the Lemmata~\ref{lem:dworkmap} and \ref{lem:Dpsmall}.

Let us assume that $(\ua, \ub) \in (0,1]^{n+(n-1)}$. In $\mathcal{G}_p(\ua, \ub)$, we say that a vertex $(\ua', \ub')$ is at level $k\in \{0,1, \ldots, \ell_p-1\}$ if it is in $\mathcal V(\Dp^k(\ua), \Dp^k(\ub))$. We define the \emph{width} of level $k$ of $\mathcal{G}_p(\ua, \ub)$ as the number of vertices on level $k$, and the \emph{width} of the graph as the minimum of the widths of each level.
For arbitrary $(\ua, \ub)$, we say that the width of $\mathcal{G}_p(\ua, \ub)$ is given by the width of $\mathcal{G}_p(\bar \ua, \bar \ub)$ with $ \bar \ua=\dpc{\ua}$ and $\bar \ub=\dpc{\ub}$.
This will make sense after the Proposition~\ref{prop:graph}, in which we show that (at least for large enough $p$), the graph $\mathcal{G}_p(\bar \ua, \bar \ub)$ is an induced subgraph of $\mathcal{G}_p(\ua, \ub)$.

\begin{ex}
\label{ex:graph1}
We consider the case where $\ub = (1, \ldots, 1)$.
Then, for all integer $k$, we have $\Dp^k(\ub) = (1, \ldots, 1)$ as well and that the set $\mathcal{S}_{p}(\Dp^k(\ua),(1))$ is the singleton $\{0\}$.
It follows by induction that each vertex in $\mathcal G_p(\ua,\ub)$ is of the form $(\Dp^k(\ua),(1))$ and has only one outgoing edge to $(\Dp^{k+1}(\ua),(1))$
Hence the graph $\mathcal G_p(\ua,\ub)$ has a shape of a ``$\rho$,'' as drawn in Figure~\ref{fig:graph1}.
\begin{figure}
\centering
\begin{tikzpicture}
\node (A) at (0.3,1) {};
\node (B) at (0.6,2) {};
\node (C) at (0.9,3) {};
\node (D) at (1.3,4) {};
\node (E) at (1.75,4.9) {};
\node (F) at (2.4,5.6) {};
\node (G) at (3.3,5.6) {};
\node (H) at (4.1,5) {};
\node (I) at (4.3,4.1) {};
\node (J) at (4,3.2) {};
\node (K) at (3,2.9) {};
\node (L) at (2.1,3.3) {};
\begin{scope}[-latex]
\draw (A)--(B);
\draw (B)--(C);
\draw (C)--(D);
\draw (D)--(E);
\draw (E)--(F);
\draw (F)--(G);
\draw (G)--(H);
\draw (H)--(I);
\draw (I)--(J);
\draw (J)--(K);
\draw (K)--(L);
\draw (L)--(D);
\end{scope}
\begin{scope}
\draw[fill=vertex] (A) circle (1mm);
\draw[fill=vertex] (B) circle (1mm);
\draw[fill=vertex] (C) circle (1mm);
\draw[fill=vertex] (D) circle (1mm);
\draw[fill=vertex] (E) circle (1mm);
\draw[fill=vertex] (F) circle (1mm);
\draw[fill=vertex] (G) circle (1mm);
\draw[fill=vertex] (H) circle (1mm);
\draw[fill=vertex] (I) circle (1mm);
\draw[fill=vertex] (J) circle (1mm);
\draw[fill=vertex] (K) circle (1mm);
\draw[fill=vertex] (L) circle (1mm);
\end{scope}
\begin{scope}
\node[left,scale=0.6] at (A) { Level $0$~~ };
\node[left,scale=0.6] at (B) { Level $1$~~ };
\node[left,scale=0.6] at (C) { Level $2$~~ };
\node[left,scale=0.6] at (D) { Level $3$~~ };
\node[left,scale=0.6] at (E) { Level $4$~~ };
\node[above left,scale=0.6] at (F) { Level $5$~ };
\node[above,scale=0.6] at (3.3,5.7) { Level $6$ };
\node[right,scale=0.6] at (H) { ~~Level $7$ };
\node[right,scale=0.6] at (I) { ~~Level $8$ };
\node[right,scale=0.6] at (J) { ~~Level $0$ };
\node[below,scale=0.6] at (3,2.8) { Level $1$ };
\node[left,scale=0.6] at (L) { Level $2$~~ };
\end{scope}
\end{tikzpicture}
\caption{The shape of the graph $\mathcal G_p(\ua,\ub)$ when $\ub = (1, \ldots, 1)$}
\label{fig:graph1}
\end{figure}
The value $\ell_p(\ua,\ub)$ is the length of the ``circle part'' (it is $9$ on the drawing of Figure~\ref{fig:graph1}) and the level increases by $1$ modulo $\ell_p(\ua,\ub)$ each time we move from a vertex to the next one.

In addition, the vertices in the circle part are characterized by the fact that the corresponding $\Dp^k(\ua)$ has all its coordinates in $(0,1]$.
In particular, when $(\ua,\ub)$ itself is in $(0,1]^{n+(n-1)}$, the initial part before the circle is not present and the graph thus reduces to a circle.
For arbitrary $(\ua,\ub)$, Proposition~\ref{prop:graph} below ensures that, when $p$ is large enough, the initial part is constituted of at most one point.
\end{ex}

\begin{prop} \label{prop:graph}
Let $\ua, \ub\in (\Q\setminus (-\N))^{n+(n-1)}$ be two sets of parameters of a hypergeometric function with common denominator $d$ and let $p$ be a prime with $p>2d{\cdot}\max\{|\alpha_i|+1, |\beta_j|+1\}$, we choose $\Delta$ such that $p\Delta\equiv 1 \pmod d$ and set $\bar \ua=\dpc{\ua}$ and $\bar \ub=\dpc{\ub}$.
\begin{enumerate}[label=(\arabic{enumi})]
    \item Assuming $\ua = \bar \ua$ and $\ub = \bar \ub$, there is an edge in $\mathcal{G}_p(\ua, \ub)$ from any vertex at level $k$ to any vertex at level $k+1$ for all $k$ taken modulo $\ell_p$.
    \item  If $(\ua, \ub) \in \mathcal{G}_p(\bar \ua, \bar \ub)$ then $\mathcal{G}_p(\bar \ua, \bar \ub)=\mathcal{G}_p(\ua, \ub)$. Otherwise, $\mathcal{G}_p(\bar \ua, \bar \ub)$ is the induced subgraph of  $\mathcal{G}_p(\ua, \ub)$ obtained by removing the vertex $(\ua, \ub)$.
    \item The shape of $\mathcal{G}_p(\ua, \ub)$ only depends on the congruence class of $p$ modulo $d$.
    \item Assuming $\ua = \bar \ua$ and $\ub = \bar \ub$, the width at level $k$ is equal to the number of $\beta_j$ such that $\mathcal{M}(\Delta^{k-1} \beta_j, \Delta^{k-1})=0$, and the width of $\mathcal G_p(\ua, \ub)$ is at most equal to the number of solutions of $\Hyp(\ua, \ub)y=0$ that can be reduced modulo $p$.
\end{enumerate}
\end{prop}

\begin{figure}
\centering

\label{graph}
\raisebox{-2cm}{
\begin{tikzpicture}[xscale=2]

\begin{scope}
\node[right,scale=0.8] at (2.5,0) { Level $0$ };
\node[right,scale=0.8] at (2.5,-1) { Level $1$ };
\node[right,scale=0.8] at (2.5,-2) { Level $2$ };
\node[right,scale=0.8] at (2.5,-3) { Level $3$ };
\end{scope}

\begin{scope}[thin,black!70]
\draw (-1.5,0.25)--(-1,-1);
\draw (-1.5,0.25)--(0,-1);
\draw (-1.5,0.25)--(1,-1);
\draw (-1.5,0.25)--(2,-1);

\draw (0,0)--(-1,-1);
\draw (0,0)--(0,-1);
\draw (0,0)--(1,-1);
\draw (0,0)--(2,-1);

\draw (1,0)--(-1,-1);
\draw (1,0)--(0,-1);
\draw (1,0)--(1,-1);
\draw (1,0)--(2,-1);

\draw (-1,-1)--(-0.5,-2);
\draw (-1,-1)--(0.5,-2);
\draw (-1,-1)--(1.5,-2);

\draw (0,-1)--(-0.5,-2);
\draw (0,-1)--(0.5,-2);
\draw (0,-1)--(1.5,-2);

\draw (1,-1)--(-0.5,-2);
\draw (1,-1)--(0.5,-2);
\draw (1,-1)--(1.5,-2);

\draw (2,-1)--(-0.5,-2);
\draw (2,-1)--(0.5,-2);
\draw (2,-1)--(1.5,-2);

\draw (-0.5,-2)--(0.5,-3);
\draw (0.5,-2)--(0.5,-3);
\draw (1.5,-2)--(0.5,-3);
\draw[-latex] (0.5,-3) to[in=-90, out=-120] (-1.25,-1.5);
\draw (-1.25,-1.52) to[in=120, out=90] (0,0);
\draw[-latex] (0.5,-3) to[in=-90, out=-60] (2.25,-1.5);
\draw (2.25,-1.52) to[in=60, out=90] (1,0);
\end{scope}

\begin{scope}
\draw[fill=vertex2] (-1.5,0.25) ellipse (0.7mm and 1.4mm);
\draw[fill=vertex3] (0,0) ellipse (0.7mm and 1.4mm);
\draw[fill=vertex] (1,0) ellipse (0.7mm and 1.4mm);
\draw[fill=vertex] (-1,-1) ellipse (0.7mm and 1.4mm);
\draw[fill=vertex] (0,-1) ellipse (0.7mm and 1.4mm);
\draw[fill=vertex] (1,-1) ellipse (0.7mm and 1.4mm);
\draw[fill=vertex] (2,-1) ellipse (0.7mm and 1.4mm);
\draw[fill=vertex] (-0.5,-2) ellipse (0.7mm and 1.4mm);
\draw[fill=vertex] (0.5,-2) ellipse (0.7mm and 1.4mm);
\draw[fill=vertex] (1.5,-2) ellipse (0.7mm and 1.4mm);
\draw[fill=vertex] (0.5,-3) ellipse (0.7mm and 1.4mm);
\end{scope}
\end{tikzpicture}}
\caption{The graph $\mathcal{G}_p(\ua, \ub)$ for $\ua = (-6/25, 9/25, 39/25, 54/25, 74/25), \ub=(-9/5, 3/5, 9/5, 3)$ with the vertex $(\ua, \ub)$ drawn in green, and the vertex $(\bar \ua, \bar \ub)$ drawn in blue for $p\equiv 13 \pmod{25}$ and $p>200$. All edges are oriented downwards except marked otherwise.} \label{fig:graph}
\end{figure}
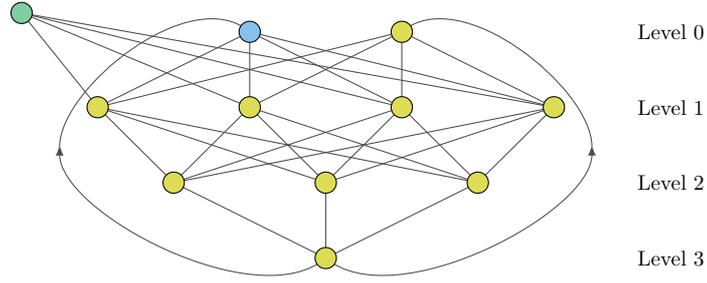
Figure~\ref{fig:graph} shows the typical shape of the induced subgraph $\mathcal G(\ua, \ub)$ of the hypergeometric relation graph, corresponding to a pair of parameters $(\ua, \ub)$ for a prime number $p$, large enough with respect to the common denominator $d$ of the parameters.
\begin{proof}
    We will show that the vertices on level $k$ are given by 
    \begin{equation} \label{eq:vertex}
        \big(\Dpr(\Dp^{k-1}(\ua)), \Dpr(\Dp^{k-1}(\ub))\big)
    \end{equation} for $r\in \mathcal{S}_{p}(\Dp^{k-1}(\ua),\Dp^{k-1}(\ub))$ for $k=0,\ldots, \ell_p-1$. First, those vertices are all distinct and the order of the tuple $(\ua,\ub)$ under $\Dp$ is $\ell_p$. Also, it is clear that these vertices are all to be found at level $k$. Moreover, we will show that any edge starting at one vertex of the form~\eqref{eq:vertex} at level $k$ ends at a vertex at level $k+1$ of the same form. Thus the induced subgraph of $\mathcal{G}_p(\ua, \ub)$ consisting of these vertices (and, possibly, $(\ua, \ub)$) contains all edges, and thus agrees with $\mathcal{G}_p(\ua, \ub)$.
    
    Without loss of generality, let us consider the vertex $(\ua, \ub)$ and another vertex $(\ua', \ub')$ of the form \eqref{eq:vertex} at the same level $0$. Then each entry of $\ua', \ub'$ differs from $\ua, \ub$ by either $0$ or $1$.
    We consider one local exponent $r\in \mathcal{S}_p(\ua, \ub)$ and take $r'\in\mathcal{S}_p(\ua', \ub')$ with $r-r'\in \{0,1\}$. If $r=0$, we pick $r'=0$ as well. Otherwise, we note that the exponent $r'$ exists, as $r$ is of the form $1-\beta_j\bmod p$, and we can pick $\beta_j'\in  \ub'$ differing from $\beta_j$ by at most $1$ to obtain $r'=1-\beta_j'$. It is left to show that the condition on the valuation of $h(\ua, \ub; r)$ is fulfilled if and only if it is fulfilled for $h(\ua', \ub'; r')$.
    According to Lemma~\ref{lem:peval}, we have 
    \[v_p\big(h(\ua', \ub'; r')\big)=\sum_{i=0}^\infty V'\left({\textstyle \dpzbig{\frac{r'}{p^i}}}, p^i\right), \]
    where $V'$ is the function defined in Equation~\eqref{eq:defV} for the parameters $\ua', \ub'$. By Lemma~\ref{lem:smallx}.(i) we now have for $i>1$ that 
    \[V'\left({\textstyle \dpzbig{\frac{r'}{p^i}}}, p^i\right) = 0,\]
    since $r'/p^i<1/p<1/2d$. For $i=1$, we observe that $\frac{r'}p=\frac{R(\beta_j', p)}p + \frac 1 p $ with $\frac 1 p < \frac 1{2d}$. Thus, by Lemma~\ref{lem:smallx}.(ii), we obtain 
    \[V'\left({\textstyle \dpzbig{\frac{r'}{p^i}}}, p\right) = \mathcal{M}'(\beta_j'\Delta, \Delta),\]
    where $\mathcal{M}'$ is defined as in Equation~\eqref{eq:defM} for the parameters $\ua', \ub'$. Similarly, we obtain
    \[V\left({\textstyle \dpzbig{\frac{r'}{p^i}}}, p\right) = \mathcal{M}(\beta_j\Delta, \Delta).\]
    As the parameters $\ua, \ub$ and $\ua', \ub'$ differ by integers and there are no integer differences between the parameters $\ua$ and $\ub$, or between $\ua'$ and $\ub'$, the values $\mathcal{M}(\beta_j\Delta, \Delta)$ and $\mathcal{M'}(\beta_j'\Delta, \Delta)$ agree, and thus the $p$-adic valuations of $h(\ua', \ub'; r')$ and $h(\ua, \ub; r)$ are the same as well.

    We finally show that $\D_{p, r'}(\ua', \ub') = \D_{p, r}(\ua, \ub)$. First we note that $\Dp(\gamma)=\Dp(\gamma+1)$ for all $\gamma \in \Zp$ if $p$ does not divide $\gamma$. By assumption $p>2d$, and $\alpha_i, \beta_j \in [0,2]$, so $\Dp(\alpha_i)=\Dp(\alpha_i')$ and $\Dp(\beta_j)=\Dp(\beta_j')$ for all $i,j$. Last one checks that $(\alpha_i)_r\in p \Zp$ if and only if  $(\alpha_i')_{r'}\in p \Zp$ and $(\beta_j)_r\in p \Zp$ if and only if  $(\beta_j')_{r'}\in p \Zp$ for all $j$. 
    
    Indeed, $(\alpha_i)_r\in p \Zp$ if and only if $\alpha_i\bmod p \in \{0, p-r+1, \ldots, p-1\}.$
    We distinguish between the four possible combinations of the cases $r=r'$ or $r=r'+1$ and $\alpha_i=\alpha_i'$ or $\alpha_i=\alpha_i'-1$. The only problems arise if $\alpha_i + 1 \equiv p-r + 1\pmod p $ or $\alpha_i \equiv 0 \pmod p$. In the first case, writing $\alpha_i=a/d$, $\beta_j=b/d$ with $a,b\in [0, d-1]$, this equation is equivalent to $a-b\equiv -d \pmod p$, a contradiction. Similarly, the second case can be discarded. 

    All what precedes concludes the proof of part~(1).
    Further, we notice that $\Dp(\ua)=\Dp(\bar \ua)$ and $\Dp(\ub)=\Dp(\bar \ub)$ (see Lemma~\ref{lem:Dpsmall}.(1)) and so
    the cardinalities of $\mathcal{S}_p(\ua, \ub)$ and $\mathcal{S}_p(\bar \ua, \bar \ub)$ agree, as the function $\mathcal{M}$ only depends on $\ua \bmod \Z$ and $\ub \bmod \Z$. Also,   it is clear that $(\bar \ua, \bar \ub)\in \mathcal{G}(\ua, \ub)$, by following the edges corresponding to $r=0$ in $\mathcal S_p(\Dp^k(\ua), \Dp^k(\ub))$ for all $k$.  In combination with what was said before, this shows~(2).
    
    By part~(2) it suffices to prove claim~(3) for $\ua, \ub$ between $0$ and $1$, which we will assume without loss of generality.
    We claim that whether $r=1-\beta_j\bmod p$ is in $\mathcal{S}_p(\ua, \ub)$ only depends on the congruence class of $p\bmod d$.
    Indeed, choose $\Delta$ with $\Delta p\equiv 1 \pmod d$. We have $v_p(h(\ua, \ub; r))=\mathcal{M}(\beta_{j(r)}\Delta ,\Delta)$ by a combination of Lemmata~\ref{lem:peval} and \ref{lem:smallx}. Here the right-hand side only depends on $p\bmod d$. Also $\Dp(\ua)$, $\Dp(\ub)$ and the value of $\ell_p$ only depend on the congruence class of $p\bmod d$. Thus, the number of vertices on each level and the number of levels is fixed for each congruence class, according to part (1). Also, by (1), all edges between vertices of two consecutive levels, exist in $\mathcal{G}_p(\ua, \ub)$. Thus the claim follows. 
    
    We finally prove~(4). We have seen in the proof of part~(1) that the width at level $k$ is equal to the cardinality of $\mathcal{S}_p(\Dp^{k-1}(\ua), \Dp^{k-1}(\ub))$.
    From the definition of $\mathcal M$ (see Equation~\eqref{eq:defM}), we find that this quantity equals the number of $\beta_j$ such that $\mathcal{M}(\Delta^{k} \beta_j, \Delta^{k})=0$. This proves the first statement.

    By Theorem~\ref{thm:red} a hypergeometric function can be reduced modulo $p$, if and only if it satisfies Christol's interlacing condition for all $\lambda \in \langle p \bmod d\rangle=\langle \Delta\bmod d\rangle$. According to Equation~\eqref{eq:basishypergeom}, a basis of solutions at $0$ of $\Hyp(\ua, \ub)y=0$ is given monomial multiples of $F_{(j)}\coloneqq {}_{n}F_{n-1}(\ua_{(j)},\ub_{(j)};x)$, where $\ua_{(j)}=\ua+1-\beta_j$ and $\ub_{(j)}=\ub+1-\beta_j$. Let $\mathcal{M}_{(j)}$ denote the function $\mathcal{M}$ for the sets of parameters $\ua_{(j)}, \ub_{(j)}$. We have $\mathcal{M}_{(j)}(\cdot,  \lambda)=\mathcal{M}(\cdot, \lambda)- \mathcal{M}(\lambda \beta_j, \lambda)$. The solution $F_{(j)}$ can be reduced modulo $p$ if  $\mathcal{M}_{(j)}(\cdot,  \lambda)$ is always non-negative for all $\lambda \in \langle p \bmod d \rangle$. As $\mathcal{M}(\cdot,  \lambda)$ is always non-negative for all $\lambda$, but assumes the value $0$, this can only be the case if $\mathcal{M}(\lambda \beta_j, \lambda)=0$. So the number of solutions that can be reduced modulo $p$ is less than the number of $b_j$ such that $\mathcal{M}(\lambda \beta_j, \lambda^{k})=0$ for all $\lambda \in \langle p \bmod p\rangle$, and the claim follows.
\end{proof}

\subsubsection{Finding algebraic relations}
    We now want to find relations between several reductions of hypergeometric functions modulo a fixed prime number $p$ and combine them to obtain annihilating polynomials for the reduction of a single hypergeometric function. We recall that we have determined in Theorem~\ref{thm:red} the set of prime numbers for which a hypergeometric function can be reduced.
    
    In the language of the hypergeometric relation graph $\mathcal{G}_p(\ua, \ub)$ introduced above, a modified version of Lemma~8.1 of \cite{Var24}, which we will need, reads as follows.

    \begin{prop} \label{prop:daniel} 
        Let $v = (\ua,\ub)\in\Vp$. If $p$ does not divide any of the numerators of the elements of $\Dp(\ua) \cup \Dp(\ub)$ and if $\pFqv v x$ can be reduced modulo $p$, then
        \[\pFqv v x \equiv \sum_{\mathclap{v\stackrel{e}{\longrightarrow}v'}}Q(e) \cdot \pFqv {v'} x ^p \pmod p.\]
    \end{prop}
\begin{proof}
According to Theorem~\ref{thm:block}, the polynomials 
$$f_t(x)=\sum_{k=t}^{k_t-1}h(v;k)p^{-v_p(h(v;t))}x^k$$
with $t\in \mathcal{T}_p(\ua, \ub)$ form a basis of solutions of $\Hyp(v)$ over $\mathbb{F}_p\ls{x^p}$. 
Since $\pFqv v x$ is a solution that can be reduced modulo $p$, there exist $g_t\in\mathbb{F}_p\ls{x}$ such that 
\[\pFqv v x \equiv \sum_{\mathclap{t\in \mathcal{T}_p(\ua, \ub)}} f_t(x)\cdot g_t(x^p) \pmod p.\]
We recall that, for $r \in \{0, \ldots, p{-}1\}$, we have introduce the section operator
\[\sigma_r : \Fp\ps{x} \to \Fp\ps{x}, \quad
\sum_{n=0}^\infty a_n x^n \mapsto \sum_{n=0}^\infty a_{pn+r} x^n.\]
Applying $\sigma_t$ to the previous equality, we obtain
\[\sigma_t\big(\pFqv v x\big) \equiv \tilde h(v;t) \cdot g_t(x) \pmod p,\]
where, by definition, $\tilde h(v;t) = h(v;t)\: p^{-v_p(h(v;t))} \in \Zp^\times$.
Besides, it follows from \cite[Lemma~8.7.(5)]{Var24} that, for all $t\in\mathcal{T}_p(v)$,
\[\sigma_t\big(\pFqv v x\big) \equiv h(v;t) \cdot \pFqv {\D_{p,t}(v)} x \pmod p.\]
Comparing these congruences, we find that $g_t(x) = 0$ whenever $h(v;t) \equiv 0 \pmod p$, \emph{i.e.}, $t\in \mathcal T_p(v)\setminus \mathcal S_p(v)$ and $g_t(x) = \pFqv {\D_{p,t}(v)} x$ otherwise. For $t\in \mathcal S_p(v),$ we have $Q(v\to \D_{p,t}(v))=f_t$ and the proposition follows.
\end{proof}

\begin{rem}
Proposition~\ref{prop:daniel} is connected to the notion of Frobenius preimages we discussed in Subsection~\ref{sssec:FrobStruct}.
Precisely, a theorem of Kedlaya~\cite[Theorem~4.1.2]{Ke22} (which is in fact attributed to Dwork in \emph{loc. cit.}) indicates that the hypergeometric equation $\Hyp(\Dp(\ua),\Dp(\ub))$ is the preimage by Frobenius of $\Hyp(\ua,\ub)$.
The relation of Proposition~\ref{prop:daniel} can be seen as an incarnation modulo~$p$ of this property.
\end{rem}    

Proposition~\ref{prop:daniel} provides a large set of algebraic relations between hypergeometric series in positive characteristic.
We will now combine them in order to find an annihilating polynomial for $\pFqab$.

    \begin{thm} \label{thm:hypalg}
         Let $(\ua, \ub)\in \Vp$. If $p > 2 d\cdot \max\{|\alpha_i|+1, |\beta_j|+1\}$
         and $\pFq{n}{n-1}{\ua}{\ub}{x}$ can be reduced modulo $p$ then there exists a relation of the form 
        \[A_0(x)\pFq{n}{n-1}{\ua}{\ub}{x}+\cdots +A_w(x)\pFq{n}{n-1}{\ua}{\ub}{x}^{q^w}\equiv 0 \pmod p,\]
        where $q=p^{\ell_p(\ua, \ub)}$, $w$ is the width of the graph $\mathcal{G}_p(\ua, \ub)$
        and the $A_i(x)$ are in $\F_q[x]$ and do not all vanish.
    \end{thm}

\begin{proof}
We use the theory of Frobenius modules (see Definition~\ref{def:frobmodules}).
We define
$$M \coloneqq \bigoplus_{\mathclap{v \in \mathcal G_p(\ua,\ub)}} \, \Fp(x) \cdot m_v,$$
where the notation means that the sum runs over the set of \emph{vertices} $v$ of $\mathcal G_p(\ua,\ub)$.
We then turn $M$ into an object of the category $\FrobModp {\Fp(x)}$ by setting
$$\phi_{M}(m_v) = \sum_{v' \stackrel e \longrightarrow v} Q(e) \cdot m_{v'}.$$
To it, is attached the $\Fp$-linear representation $\mathbf V(M)$ of $\Gal(\Fp(x)^\sep/\Fp(x))$ defined by Formula~\eqref{eq:katzfunctor}.
A straightforward computation shows that $\mathbf V(M)$ consists of families $(f_v)_{v \in \mathcal G_p(\ua,\ub)}$ of elements of $\Fp(x)^\sep$ satisfying the equations
$$f_v = \sum_{\mathclap{v \stackrel{e}{\longrightarrow}v'}}Q(e) \cdot f_{v'}^p.$$
In particular, Proposition~\ref{prop:daniel} shows that it contains the family $\big({}_nF_{n-1}(v;x) \bmod p\big)_{v \in \mathcal G_p(\ua,\ub)}$.

On the other hand, $M$ admits a canonical decomposition corresponding to the level partition of the graph $\mathcal G_p(\ua,\ub)$, namely we have
$$M = M_0 \oplus M_1 \oplus \cdots \oplus M_{\ell_p(\ua,\ub) - 1},$$
where the summand $M_k$ corresponds to the vertices of level~$k$. 
The fact that an edge in $\mathcal G_p(\ua,\ub)$ connects a vertex of level $k$ to a vertex of level $k{+}1$ implies that $\phi_{M}$ maps $M_{k+1}$ to $M_k$.
It follows that each $M_k$ is stable by the composite $\phi_{M}^{\ell_p(\ua,\ub)}$, which thus defines a structure of $q$-Frobenius module on this space for $q = p^{\ell_p(\ua,\ub)}$.

If $\mathcal G_p(\ua,\ub)_0$ denote the set of vertices of $\mathcal G_p(\ua,\ub)$ of level~$0$,
$\mathbf V(M_0)$ is a finite dimensional $\F_q$-representation of $\Gal(\F_q(x)^\sep/\F_q(x))$ and we have an inclusion
$$\mathbf V(M_0) \subseteq \big(\Q(x)^\sep\big)^{\mathcal G_p(\ua,\ub)_0}$$
commuting with the Galois action. Moreover $\mathbf V(M_0)$ contains a vector whose $(\ua,\ub)$-coordinate is $\pFq{n}{n-1}{\ua}{\ub}{x} \bmod p$.
Let then
$$\varpi : \big(\Q(x)^\sep\big)^{\mathcal G_p(\ua,\ub)_0} \to \Q(x)^\sep$$
denote the projection on the $(\ua,\ub)$-coordinate and $V$ be the image of $\mathbf V\big(M_0\big)$ under $\varpi$.
Clearly $V$ is finite dimensional over $\F_q$, it is stable by the action of $\Gal(\F_q(x)^\sep/\F_q(x))$ and it contains $\pFq{n}{n-1}{\ua}{\ub}{x} \bmod p$.
These properties permits to build an annihilating polynomial for $\pFq{n}{n-1}{\ua}{\ub}{x} \bmod p$, namely
$$Z(Y) \coloneqq \prod_{y \in V} \, (Y - y).$$
The stability by the Galois action ensures that $Z(Y)$ has coefficients in $\F_q(x)$. Moreover, 
the fact that $V$ is a $\F_q$-linear vector space implies that $Z(Y)$ is a $q$-linearized polynomial, \emph{i.e.}, it has the shape
$$Z(Y) = c_0 Y + c_1 Y^q + \cdots + c_s Y^{q^s} \qquad (c_i \in \F_q(x)).$$
By comparing degrees, we find that $s$ is the dimension of $V$ over $\F_q$.

It then only remains to show that $s \leq w$.
For this, we recall that we have introduced in Definition~\ref{def:etalepart} the \emph{étale part} of a Frobenius module.
Applying Theorem~\ref{thm:Katz} and Lemma~\ref{lem:VMet}, we get
$$s = \dim_{\F_q} V \leq \dim_{\F_q} \mathbf V(M_0) = \dim_{\F_q} \mathbf V\big(M_0^\et\big) = \dim_{\Q(x)} M_0^\et.$$
On the other hand, we observe that $\phi_{M}$ induces a map $M_{k+1}^\et \to M_k^\et$.
We need to be careful that it is not linear but we can actually make it linear by twisting the domain, \emph{i.e.}, we consider
$$\begin{array}{rcl}
\psi_k : 
\F_q(x) \otimes_{\phi, \F_q(x)} M_{k+1}^\et & \longrightarrow & M_k^\et \smallskip \\
\lambda \otimes m & \mapsto & \lambda \phi_{M}(m)
\end{array}$$
which is now $\F_q(x)$-linear.
Furthermore, the composite of the $\psi_k$ (correctly twisted) is the linearization of $\phi_{M}^{\ell_p(\ua,\ub)}$ acting on $M_0^\et$ and hence is an isomorphism by étaleness.
We conclude that each $\psi_k$ is itself an isomorphism, which finally leaves up with the inequality
$$\textstyle \dim_{\F_q(x)} M_0^\et = \min_k \dim_{\F_q(x)} M_k^\et.$$
Let $w_k$ be the width of level $k$ of $\mathcal G_p(\ua,\ub)$.
Under our assumptions on $p$, it follows from Proposition~\ref{prop:graph} that
$$\dim_{\F_q(x)} {M^\et_k} \leq \dim_{\F_q(x)} {M_k} = w_k$$
whenever $k > 0$. This estimation also holds for $k = 0$ if the vertex $(\ua, \ub)$ survives in $\mathcal G_p(\bar\ua, \bar\ub)$.
On the contrary, if it does not, Proposition~\ref{prop:graph} again ensures that there is no edge in $\mathcal G_p(\ua,\ub)$ ending at $(\ua,\ub)$.
Consequently, the corresponding vector $m_{(\ua,\ub)}$ cannot be in $M_0^\et$ and we deduce
$$\dim_{\F_q(x)} {M^\et_0} \leq \dim_{\F_q(x)} {M_0} - 1 = w_0.$$
In all cases, we thus have $\dim_{\F_q(x)} {M^\et_k} \leq w_k$, which finally implies that $s \leq \min_k w_k = w$ as desired.
\end{proof}

\subsubsection{Approaching Galois groups}

Our objective is now to reinterpret the previous results along the lines of Conjecture~\ref{conj:weak}.
For this, we need to exhibit a number field whose residue fields are somehow related to the values $\ell_p(\ua, \ub)$.

In all what follows, our number fields will all be defined as subfields of the field of complex numbers $\C$.
For a positive integer $d$, we set $\zeta_d = \exp(2\i\pi/d) \in \C$; it is a primitive $d$-th root of unity and it generates the cyclotomic field $\Q(\zeta_d)$.
We recall that $\Q(\zeta_d)$ is a Galois extension of $\Q$ and that its Galois group is canonically isomorphic to $(\Z/d\Z)^\times$:
an element $\lambda \in (\Z/d\Z)^\times$ acts on $\Q(\zeta_d)$ by $\zeta_d \mapsto \zeta_d^\lambda$.
In particular, we notice that the element $-1 \in (\Z/d\Z)^\times$ acts as complex conjugation.

As before, we let $d$ be the smallest common denominator of all the parameters $\ua$ and $\ub$. Further we define the group
\[D(\ua, \ub)\coloneqq\big\{\,\lambda\in (\Z/d\Z)^\times \,:\, \lambda{\cdot} \ua\equiv \ua \!\!\!\!\pmod \Z \text{ and } \lambda{\cdot} \ub \equiv\ub \!\!\!\!\pmod \Z\text{ as sets}\,\big\}.\]
We can map $D(\ua, \ub)$ to a product of symmetric groups.
Precisely, let $\mathfrak S(\ua)$ (resp. $\mathfrak S(\ub)$) be the group of bijections of $\ua$ (resp. $\ub$) viewed as a multiset.
We then have a group morphism $\sigma : D(\ua, \ub) \to \mathfrak S(\ua) \times \mathfrak S(\ub)$ that takes $\lambda$ to the multiplication by $\lambda$.

\begin{lem} \label{lem:Dembeds}
The morphism $\sigma$ is injective.
\end{lem}

\begin{proof}
We write $\alpha_i = \frac{a_i} {d}$, $\beta_i = \frac{b_i}{d}$.
By definition of $d$, we have $\gcd(a_1, \ldots, a_n, b_1, \ldots, b_{n-1}, d) = 1$ as otherwise the fractions $\frac{a_i} {d}$ and $\frac{b_i}{d}$ could all be simplified by a common factor, which is excluded.
By Bézout's theorem, there exist integers $u_1, \ldots, u_n, v_1, \ldots, v_{n-1}, w$ such that
$$u_1 a_1 + \cdots + u_n a_n + v_1 b_1 + \cdots + v_{n-1} b_{n-1} + w d = 1.$$
We now assume that we are given $\lambda, \mu \in D(\ua, \ub)$ such that $\sigma(\lambda) = \sigma(\mu)$.
Then, following the definition, we find that $\lambda a_i \equiv \mu a_i \pmod d$ and $\lambda b_i \equiv \mu b_i \pmod {d}$ for all $i$.
Weighting these congruences with the coefficients $u_i$ and $v_i$ and summing them, we get
\begin{multline*}
\lambda \cdot (u_1 a_1 + \cdots + u_n a_n + v_1 b_1 + \cdots + v_{n-1} b_{n-1}) \\
\equiv \mu \cdot (u_1 a_1 + \cdots + u_n a_n + v_1 b_1 + \cdots + v_{n-1} b_{n-1}) \pmod {d}
\end{multline*}
which finally reduces to $\lambda \equiv \mu \pmod {d}$.
\end{proof}

\begin{defi}
We define $K(\ua, \ub)$ as the subextension of $\Q(\zeta_{d})$ corresponding to $D(\ua, \ub)$ \emph{via} Galois correspondence.
\end{defi}

In concrete terms, $K(\ua, \ub)\coloneqq \Q(\zeta_{d})^{D(\ua, \ub)}$.
The field $K(\ua, \ub)$ is undoubtedly the number field we are looking for, as we shall demonstrate in the sequel.
First of all, we observe that it is closely related to the Galois group of the hypergeometric differential equation in characteristic zero.
Indeed, it is the subfield generated by the entries of the monodromy matrices exhibited in Theorem~\ref{thm:monodromy} and we know that those matrices generate the differential Galois group as an algebraic group.

Now, we make the connection between the field $K(\ua,\ub)$ and the numbers $\ell_p(\ua, \ub)$,
which will eventually relate $K(\ua,\ub)$ to the Galois groups of the reduction of the hypergeometric series $\pFq{n}{n-1}{\ua}{\ub}{x}$ modulo the primes.

\begin{prop}
\label{prop:residuefield}
Let $p$ be a prime number which does not divide $d$, and let $\pp$ be a prime in $K(\ua,\ub)$ above $p$. 
Then the residue field of $K(\ua,\ub)$ at $\mathfrak{p}$ has degree $\ell_p(\ua,\ub)$ over $\mathbb{F}_p$.
\end{prop}

\begin{proof}
Throughout the proof, we omit the parameter $(\ua, \ub)$ in the notation and simply write $K$ and $\ell_p$ for $K(\ua, \ub)$ and $\ell_p(\ua, \ub)$ respectively.
The assumption on $p$ implies that $p$ is not ramified in the extension $\Q(\zeta_d)/\Q$.
Let $\qq$ be a prime of $\Q(\zeta_d)$ above $\pp$ and write $k_\pp$ (resp. $k_\qq$) for the residue field of $K$ at $\pp$ (resp. of $\Q(\zeta_d)$ at $\qq$).
The Galois group of the extension $k_\qq / \F_p$ is the subgroup of $(\Z/d\Z)^\times$ generated by $p$; we call it $P$.
Since $K$ is defined as the fixed subfield by $D$, we deduce that $k_\pp$ is the subfield of $k_\qq$ fixed by $P \cap D$. Thus
$$\Gal(k_\pp / \F_p) = P / (D \cap P).$$
Besides the cardinality of $P / (D \cap P)$ is the smallest positive integer $e$ such that $p^{-e} \in D$.
The latter condition is equivalent to the statement ``$\D^e(\ua) = \ua$ and $\D^e(\ub) = \ub$''
given that, according to Lemma~\ref{lem:dworkmap}.(2), the Dwork map acts by division by $p$ on $\Zp/\Z$.
Comparing with the definition of~$\ell_p$, we find that $P / (D \cap P)$ has cardinality $\ell_p$, which concludes the proof.
\end{proof}

\begin{thm} \label{thm:GLw}
Let $p$ be a prime number with $p > 2 d\cdot \max\{|\alpha_i|+1, |\beta_j|+1\}$ and let $\pp$ be a prime in $K(\ua, \ub)$ above $p$.
Then 
\[\Gal\big(\pFq{n}{n-1}{\ua}{\ub}{x}\,\big|\,k_{\mathfrak p}(x)\big) \, \subseteq \, \GL_{w}(k_{\mathfrak{p}})\]
where $w$ denotes the width of the graph $\mathcal G_p(\ua, \ub)$.
\end{thm}

\begin{proof}
By Proposition~\ref{prop:residuefield}, we have $k_{\mathfrak{p}}\simeq \F_{p^{\ell_p(\ua, \ub)}}$. With this, the theorem follows from Theorem~\ref{thm:hypalg}. 
\end{proof}

We conclude this section by looking at some examples we already considered in Section~\ref{sec:picture}.

\begin{ex}
    We recall the example from Subsection~\ref{sssec:ex:binom3}:
    \[f(x)=\sum_{n=0}^\infty \binom{2n}{n}^3 \cdot x^n= {\textstyle \pFq{3}{2}{\big(\frac 1 2, \frac 1 2, \frac 1 2\big)}{(1, 1)}{64x}}.\]
    Clearly $\ell_p(\ua, \ub)=1$ for all prime numbers $p$, as $(\Z/2\Z)^\times$ is trivial. Moreover, as all bottom parameters are equal to $1$, it follows from Example~\ref{ex:graph1} that the width $w$ of the graph is $1$, and consequently, the Galois group $\Gal\big(f(x)\,\big|\, \F_p(x)\big)$ embeds into $\F_p^\times$. This agrees with the \emph{ad-hoc} arguments deployed in Subsection~\ref{sssec:ex:binom3}. However, we cannot provide a systematic argument explaining the dichotomy of the Galois groups exhibited in Equation~\eqref{eq:ex:binom3} here.
\end{ex}

\begin{ex}
\label{ex:3F2}
We now consider the example $f(x)=\pFq{3}{2}{\big(\frac 1 9, \frac 4 9, \frac 5 9\big)}{\big(\frac 1 3, 1\big)}{x}$ from Subsection \ref{sssec:ex:3F2}.
We have seen in Example~\ref{ex:gbalg} that $f(x)$ is globally bounded.
An easy computation shows that $\ell_p(\ua, \ub)=\ell_p=\operatorname{ord}_9(p)$ and these values are given by 
\[\ell_p=\begin{cases}
    1 & \text{if } p\equiv 1 \pmod 9\\
    2 & \text{if } p \equiv 8 \pmod 9\\
    3 & \text{if } p \equiv 4, 7 \pmod 9\\
    6 & \text{if } p \equiv 2, 5 \pmod 9
\end{cases}\]
We next discuss the width of the graphs $\mathcal{G}_p(\ua, \ub)$. 
In Figure~\ref{fig:interlacing}, we see that, for $\lambda=2, 5, 8$, we have $\mathcal{M}(\lambda^k \beta_j, \lambda^k)=0$ only for $\beta_j=1$ and, for $\lambda=1, 4, 7$, the equation also holds for $\beta_j=\frac 1 3$. We have $\langle p \bmod 9 \rangle \subseteq \{1, 4, 7\}$ if and only if $p\equiv 1, 4, 7$. This shows that the width $w$ of the graph is $2$ in the first case and $1$ in the second case. By Theorem~\ref{thm:GLw}, we conclude that $\Gal(f(x)\mid\F_q(x))$ embeds into $\GL_w(\F_q)$, where $q=p^{\ell_p(\ua, \ub)}$. If we pass to $\Gal(f(x)\mid\F_p(x))$ we have additionally the Frobenius morphism acting.
Altogether, this shows that the Galois groups of $f(x)$ over $\F_p(x)$ embed into the semidirect products given in Equation~\eqref{eq:GalGrpsEx}. We cannot make any statement about the index or precise nature of the subgroup with the techniques developed here, although we conjecture equality for all prime numbers based on computer experiments.
\end{ex}

\subsection{The case of Gaussian hypergeometric functions}
\label{ssec:2F1}

We now consider the case of Gaussian hypergeometric functions with bottom parameter $1$, that are functions of the form $\pFq 2 1 \ua {(1)} x$ with $\ua = (\alpha_1, \alpha_2) \in \Q^2$ and $\ub = (1)$.
While this is an extremely particular case of globally bounded $D$-finite functions, already for such function an astonishing amount of nontrivial phenomena occur.
We present strong evidence towards an affirmative answer to Conjecture~\ref{conj:weak} (and its refinements) in this case. 

Throughout this subsection, the letter $p$ refers to a prime number which is coprime with $d$.
For simplicity, we will often omit the parameter $(\ua,\ub)$ in the notation when no confusion can occur;
for example, we shall simply write $\ell_p$ for $\ell_p(\ua,\ub)$ and similarly, we set $D \coloneqq D(\ua, \ub)$ and $K \coloneqq K(\ua, \ub)$.
By Lemma~\ref{lem:Dembeds}, the group $D$ has cardinality at most $2$, \emph{i.e.}, $D = \{1, \lambda\}$ with $\lambda \in (\Z/d\Z)^\times$, $\lambda^2 = 1$.
We also define $e_p \coloneqq \ell_p / 2$.

\subsubsection{Estimating the Galois group} \label{sssec:GalGauss}

We know from Example~\ref{ex:graph1} that, in our case of interest, the graph $\mathcal G_p(\ua,\ub)$ has the $\rho$-shape drawn in Figure~\ref{fig:graph1}.
Let $(\ua_k, (1))$ denote the parameter corresponding to the vertex of level~$k$ in the circle part, \emph{i.e.}, $\ua_k = \Dp^{k'}(\ua)$ where $k'$ is large enough and congruent to $k$ modulo $\ell_p$.
Let also $H_k(x)$ denote the reduction modulo $p$ of the hypergeometric series attached to the parameters $(\ua_k, (1))$.
Similarly we define $B_k(x)$ as the truncation at $x^p$ of $H_k(x) \bmod p$; thus $B_k(x)$ is a polynomial of degree less than $p$ with coefficients in $\Fp$.
Proposition~\ref{prop:daniel} ensures that
\begin{equation}\label{eq:GaussianRel}
H_k(x) \equiv B_k(x) \cdot H_{k+1}^p(x).
\end{equation}
Combining these equations, we get the relation $H_0(x) = A(x) \cdot H_0(x)^q$ where $q = p^{\ell_p}$ and
\begin{equation} \label{eq:defA}
    A(x) = B_0(x) \cdot B_1(x)^p \cdots B_{\ell_p - 1}(x)^{p^{\ell_p-1}}.
\end{equation}
Similarly, we find that there exist a polynomial $\tilde A(x) \in \F_q[x]$ and a nonnegative integer $m$ such that $\pFq 2 1 {\ua}{(1)} x = \tilde A(x) H_0(x)^{q^m}$.
We observe moreover that the relation
$$\pFq 2 1 {\ua}{(1)} x \equiv \tilde A(x) H_0(x)^{q^m} = \tilde A(x) \cdot A(x)^{1 + q + \cdots + q^{m-1}} H_0(x) \pmod p$$
indicates that the functions $\pFq 2 1 {\ua}{(1)} x$ and $H_0(x)$ generate the same extension over $k_\pp(x)$.
In other words, we can assume without loss of generality that the parameters $\alpha_1$ and $\alpha_2$ are both in the interval $(0,1]$, \emph{i.e.}, $H_0(x) = \pFq 2 1 {\ua}{(1)} x \bmod p$.

Besides, if $\alpha_1$ or $\alpha_2$ is equal to $1$, the hypergeometric series $\pFq 2 1 \ua {(1)} x$ is equal to $(1 - x)^{-\alpha}$ (where $\alpha$ is the value of the other parameter),
a case we already covered in Subsection~\ref{sssec:ex:alg}.
Therefore, from now on, we suppose $\alpha_1, \alpha_2 \in (0,1)$.

The Galois groups we are interested in can be computed using Kummer theory.

\begin{prop}
\label{prop:kummer}
Let $\pp$ be a prime of $K$ above $p$.
The Galois group $\Gal\big(\pFq 2 1 {\ua}{(1)} x \mid k_\pp(x)\big)$ naturally embeds in $k_\pp^\times$ and hence is a cyclic group.
Moreover, its order is
$$\frac {q{-}1} {\gcd(r, q{-}1)}$$
where $r$ denotes the largest integer such that $A(x)$ is $r$-th power in $k_\pp(x)$.
\end{prop}

\begin{proof}
By what precedes, we know that $H_0(x)$ is a $(q{-}1)$-th root of the polynomial $A(x) \in \Fp[x] \subseteq k_\pp(x)$.
Since $k_\pp$ contains all the $(q{-}1)$-th roots of unity, Kummer's theory tells us that $\Gal(H_0(x) \mid k_\pp(x))$ embeds into $\um_{q-1}(k_\pp(x)) = k_\pp^\times$
and that its order is $\frac {q{-}1} {\gcd(r, q{-}1)}$.
\end{proof}

After Proposition~\ref{prop:kummer}, we are led to determine when $A(x)$ is a perfect power. In general, we do not expect this to occur, although it could happen ``by accident.''
There is, however, one important case where this happens systematically, namely in the case when $\ell_p$ is even and $-p^{e_p} \in D$.

In all what follows, if $\gamma \in \Zp$, we let $\gamma^{(0)}, \gamma^{(1)}, \ldots$ denote the ``digits'' of the $p$-adic integer $(-\gamma)$, \emph{i.e.}, $\gamma^{(k)} \in \{0, \ldots, p{-}1\}$ for all $k$ and the following expansion
$$-\gamma = \gamma^{(0)} + p \gamma^{(1)} + p^2 \gamma^{(2)} + \cdots$$
holds in the ring of $p$-adic integers.
When $\ug = (\gamma_1, \gamma_2)$, we also let $B(\ug; x) \in \Fp[x]$ be the truncation at $x^p$ of $\pFq 2 1 {\ug}{(1)} x \bmod p$.
By definition, we thus have $B_k(x) = B(\ua_k; x)$.

\begin{lem} \label{lem:euler}
Let $\ug = (\gamma_1, \gamma_2) \in (\Zp \cap (0,1))^2$. Then
\begin{align*}
\deg B(\ug; x) & = \min(\gamma_1^{(0)}, \gamma_2^{(0)}), \\
\deg B(1 - \ug; x) &= p - 1 - \max(\gamma_1^{(0)}, \gamma_2^{(0)}).
\end{align*} 
Moreover, we have the relation
\begin{equation}
\label{eq:Bg}
B(\ug; x) = (1-x)^{\gamma_1^{(0)} + \gamma_2^{(0)} -  (p - 1)} \cdot B(\boldsymbol{1}-\ug; x).
\end{equation}
\end{lem}

\begin{proof}
Set $\delta = \min(\gamma_1^{(0)}, \gamma_2^{(0)})$.
By definition, the coefficient of $x^r$ for $r<p$ in $B(\ug; x)$ is $\frac{(\gamma_1)_r (\gamma_2)_r}{(r!)^2}$. Hence it vanishes modulo $p$ if and only if the numerator does. When $r > \delta$, the numerator contains the factor $(\gamma_1 + \delta)\cdot (\gamma_2 + \delta)$, and so it vanishes modulo $p$. On the contrary, when $r = \delta$, we have
$$(\gamma_1)_r (\gamma_2)_r = \gamma_1 (\gamma_1 + 1) \cdots (\gamma_1 + \delta - 1) \cdot \gamma_2 (\gamma_2 + 1) \cdots (\gamma_2 + \delta - 1) \cdot,$$
and, again by definition of $\delta$, none of the above factors vanish. We conclude that $B(\ug; x)$ has degree $\delta$, as claimed. The second statement on the degree of $B(1 - \ug; x)$ is proved similarly.

Thanks to Euler's transformation, we know that 
$$\pFq 2 1 {\ug} {(1)} x = (1-x)^{1-\gamma_1-\gamma_2} \cdot \pFq 2 1 {1-\ug} {(1)} x.$$
We now note that the exponent $1-\gamma_1-\gamma_2$ is congruent to $\gamma_1^{(0)} + \gamma_2^{(0)} -  (p - 1)$ modulo $p$. Therefore,
$$B(\ug; x) \equiv (1-x)^{\gamma_1^{(0)}+\gamma_2^{(0)} - (p - 1)} \cdot B(1-\ug; x) \pmod{x^p}.$$
If $\gamma_1^{(0)}+\gamma_2^{(0)} \geq p - 1$, both sides are polynomials of degree at most $p$, and we can conclude that they are indeed equal. On the contrary, when $\gamma_1^{(0)}-\gamma_2^{(0)} \leq p - 1$, we rewrite the above congruence as
$$B(1-\ug; x) \equiv (1-x)^{(p-1) - \gamma_1^{(0)}-\gamma_2^{(0)}} \cdot B(\ug; x) \pmod{x^p}$$
and we conclude similarly.
\end{proof}

We introduce one more numerical invariant attached to the situation; it is the integer $m$ defined by
$$m \coloneqq \frac d {\gcd(d, \, a_1{+}a_2)}$$
where we have written $\alpha_i=a_i/d$ for $i=1, 2$.
We note that $m$ is also the denominator of the irreducible fraction representing $\alpha_1 + \alpha_2$.

\begin{prop} \label{prop:exponent}
    We assume that $\ell_p$ is even and that $-p^{e_p} \in D$. Then
    \[A(x)=(1-x)^{(1+p^{e_p})(1-\alpha_1-\alpha_2)}\left(B_{e_p}B_{e_p+1}^p\cdots B_{\ell_p-1}^{p^{e_p-1}}\right)^{1+p^{e_p}}.\]
    It is a perfect power with exponent $\displaystyle \frac{1+p^{e_p}} m$.
\end{prop}

\begin{proof}
We recall from Lemma~\ref{lem:dworkmap}.(2) that the Dwork operator $\Dp$ acts by division by $p$ on $\Zp/\Z$.
The assumption of the proposition then tells us that $\Dp^{e_p}(\ua) \equiv - \ua \pmod{\Z}$, which implies more generally $\Dp^{e_p + k}(\ua) \equiv - \Dp^k(\ua) \pmod{\Z}$ for all~$k$. 
Since our parameters are all in $(0,1)$, we deduce that $\Dp^{e_p + k}(\ua) = 1- \Dp^k(\ua)$.
On the other hand, we recall from Lemma~\ref{lem:dworkmap} that $\Dp^k(\gamma)^{(0)} = \gamma^{(k)}$ for all~$k$.
Hence, Lemma~\ref{lem:euler} gives
$$B_k(x) = (1-x)^{\alpha_1^{(k)} + \alpha_2^{(k)} - (p-1)} B_{k + e_p}(x).$$
Coming back to the definition of $A(x)$, we find
\begin{align*}
A(x) = \prod_{k=0}^{\ell_p - 1} B_k(x)^{p^k}
& = \prod_{k=0}^{e_p - 1} (1-x)^{p^k (\alpha_1^{(k)} + \alpha_2^{(k)} - (p-1))} 
    B_{k + e_p}(x)^{p^k + p^{k+e_p}} \\
& = (1-x)^v \cdot \left(\prod_{k=0}^{e_p - 1} B_{k + e_p}(x)^{p^k}\right)^{1 + p^{e_p}},
\end{align*}
where the exponent $v$ is given by $v = \sum_{k=0}^{e_p - 1} p^k u_k$ with
$u_k = \alpha_1^{(k)} + \alpha_2^{(k)} - (p-1)$.
Now, we remark that, for $\gamma \in \Zp$, we have
$$-(1 - \gamma) = -1 - \sum_{m=0}^\infty p^m \gamma^{(m)} = \sum_{m=0}^\infty p^m (p - 1 -\gamma^{(m)}),$$
implying that $(1 - \gamma)^{(m)} = p - 1 - \gamma^{(m)}$ for all $m \in \N$.
Applying this result with $\gamma = \Dp^k(\alpha_i)$ ($i \in \{1, 2\}$) and remembering that $\Dp^{e_p + k}(\ua) = 1- \Dp^k(\ua)$, we conclude that
$u_{e_p + k} = -u_k$. Weighting by $p^k$ and summing over $k$ (and doing the computation in the ring of $p$-adic integers), we find
\begin{align*}
\sum_{k=0}^\infty p^k u_k
= \left(\sum_{k'=0}^\infty (-1)^{k'} p^{e_p k'} \right) \cdot \left(\sum_{k=0}^{e_p - 1} p^k u_k \right) 
= \frac 1 {1 + p^{e_p}} \sum_{k=0}^{e_p - 1} p^k u_k.
\end{align*}
On the other hand, coming back to the definition, we have
$$\sum_{k=0}^\infty p^k u_k
= \sum_{k=0}^\infty p^k \big(\alpha_1^{(k)} + \alpha_2^{(k)} - (p-1)\big)
= 1 - \alpha_1 - \alpha_2.$$
Hence, we conclude that $v = (1 + p^{e_p}) \cdot (1 - \alpha_1 - \alpha_2)$ and the first part of the proposition is proved.

We write $\alpha_1 + \alpha_2 = a/m$ with $a$ coprime with $m$. We thus have 
$$v = \frac {(1 + p^{e_p})(m-a)} m.$$
Since $m-a$ is coprime with $m$, we deduce that $m$ divides $1 + p^{e_p}$ and then that $v$ is a multiple of $r \coloneqq \frac {1 + p^{e_p}} m$.
Given that $(1 + p^{e_p})$ is also obviously a multiple of $r$, we conclude that $A(x)$ is a $r$-th power.
\end{proof}

As a conclusion, we have proved the following.

\begin{cor}
\label{cor:galoismodp}
Let $\pp$ be a prime of $K$ above~$p$.
The Galois group $\Gal(\pFq 2 1 {\ua} {(1)} x \mid k_\pp(x))$ embeds in $k_\pp^\times$ and hence is a cyclic group.
Moreover, if $\ell_p$ is even and $-p^{e_p} \in D$, its cardinality divides $(p^{e_p}-1)\cdot m$
\end{cor}

\subsubsection{Uniformizing the Galois groups}

We now aim at finding a common origin of the Galois groups given by Corollary~\ref{cor:galoismodp} in the terms of Conjecture~\ref{conj:weak}.
For this, we shall consider several subgroups of $\GL_1(K) = K^\times$.
The first one is $\um_m(K)$, the group of $m$-th roots of unity in $K$. The following lemma ensures that $K$ contains all $m$-th roots of unity, meaning that $\um_m(K)$ is a cyclic group of cardinality $m$.

\begin{lem}
We have $\zeta_m \in K$.
\end{lem}

\begin{proof}
It follows immediately from the definition that $m$ divides $d$. 
Hence $\zeta_m \in \Q(\zeta_d)$ and it is enough to prove that $\zeta_m = \zeta_m^\lambda$ for any $\lambda \in D$.
Let $\lambda \in D$. By definition, the multiplication by $\lambda$ induces a permutation of $\{\alpha_1, \alpha_2\}$ modulo $\Z$.
Thus
$$\lambda \cdot (a_1 + a_2) \equiv a_1 + a_2 \pmod d$$
and we conclude that $\lambda \equiv 1 \pmod m$. The lemma follows.
\end{proof}

The second subgroup of $\GL_1(K)$ that will play a major role in what follows is the group $\GL_1(K^+)$ where $K^+ \coloneqq K \cap \R$.
Since the complex conjugacy is represented by the element $-1 \in (\Z/d\Z)^\times$, we deduce that $K^+$ is the subfield of $\Q(\zeta_d)$ cutted out by the subgroup 
$$D^+ \coloneqq D \cdot \{\pm 1\} \subseteq (\Z/d\Z)^\times.$$

\begin{defi} \label{def:G}
We define the subgroup $G$ of $\GL_1(K)$ by
\[ G\coloneqq 
\begin{cases} 
  \GL_1(K) & \text{if $K = K^+$} \\
  \GL_1(K^+) \cdot \langle \xi - \bar\xi \rangle & \text{if $K \neq K^+$ and $m = 2$} \\
  \GL_1(K^+) \cdot \langle 1 + \zeta_m \rangle & \text{if $K \neq K^+$ and $m \neq 2$}
\end{cases}\]
where $\xi$ is an element of $K\setminus K^+$ and $\bar\xi$ is its complex conjugate.
\end{defi}

\begin{rem}
The reason why the definition is slightly different when $m = 2$ comes simply from the fact that $1 + \zeta_m$ vanishes in this case.
We notice moreover that, as soon as $\xi \in K\setminus K^+$, the element $\xi - \bar\xi$ is a nonzero real multiple of $\i \in \C$.
The group $\GL_1(K^+) \cdot \langle \xi - \bar\xi\rangle$ is then the subgroup 
$$(K^\times \cap \R) \sqcup (K^\times \cap \R\i) \subseteq K^\times.$$
In particular, it does not depend on the choice of $\xi$, and so Definition~\ref{def:G} is not ambiguous.
\end{rem}

We are going to prove that the group $G$ fulfills the requirements of Conjecture~\ref{conj:weak}.
We start by a lemma that reformulates the condition of Corollary~\ref{cor:galoismodp} in more abstract terms.

\begin{lem} \label{lem:kplus}
Let $\pp$ be a prime in $K$ above $p$ and let $k_\pp$ (resp. $k^+_\pp$) denote the residue field of $K$ (resp. $K^+$) at $\pp$.
The extension $k_\pp / k^+_\pp$ is nontrivial if and only if $\ell_p$ is even and $-p^{e_p} \in D$.
Moreover, when this occurs, it has degree $2$.
\end{lem}

\begin{proof}
Let $P$ denote the subgroup of $(\Z/d\Z)^\times$ generated by $p$.
We know from the proof of Proposition~\ref{prop:residuefield} that $\Gal(k_\pp / \F_p) = P / (D \cap P)$,
and similarly one proves that $\Gal(k^+_\pp / \F_p) = P / (D^+ \cap P)$.
Hence the extension $k_\pp / k^+_\pp$ is trivial if and only if 
\begin{equation} \label{eq:DD+}
D \cap P = D^+ \cap P.
\end{equation}
If $-1 \in D$, we have $D = D^+$ and the condition~\eqref{eq:DD+} is obviously satisfied.
On the other hand, still assuming that $-1 \in D$, the condition $-p^{e_p} \in D$ is equivalent to $p^{e_p} \in D$, which cannot be true by minimality of $\ell_p$.
The equivalence stated in the lemma is then proved in this case.

We now assume that $-1 \not\in D$. Then $D^+$ is the disjoint union of $D$ and $(-D)$, from what we deduce that the condition~\eqref{eq:DD+} is equivalent to the fact that $P$ does not meet $(-D)$.
Clearly, if $-p^{e_p} \in D$, then $P$ meets $(-D)$. Conversely, let us assume that $P \cap (-D) = \emptyset$, \emph{i.e.}, that there exists an exponent $e \in \N$ such that $-p^e \in D$.
Then $p^{2e} \in D$ and it follows that $2e$ needs to be a multiple of $\ell_p$: there exists an integer $n$ such that $2e = n\ell_p$.
Necessarily $n$ has to be odd because otherwise, we would deduce that $\ell_p$ divides $e$ and, consequently, that $p^e \in D$; this cannot be true given that $-p^e \in D$ and $-1 \not\in D$ by assumption.
It follows that $\ell_p$ is even. Now, writing $n = 2n' + 1$, we get $p^e = p^{n' \ell_p + e_p} \equiv p^{e_p} \pmod d$ and so $-p^{e_p} \in D$.

Finally, the fact that $k_\pp / k^+_\pp$ has degree at most $2$ follows from the fact that $D \cap P$ has index at most $2$ in $D^+ \cap P$.
\end{proof}

Following Conjecture~\ref{conj:weak}, for any place $\pp$ of $K$, we set
\begin{equation} \label{eq:Gp}
    G_\pp \coloneqq \text{image}\big(G \cap \GL_1(\Opp) \longrightarrow \GL_1(k_\pp)\big)
\end{equation}
where $\Opp$ denotes the localization of the ring of integers of $K$ at $\pp$ and $k_\pp = \O_{(\pp)} / \pp \O_{(\pp)}$ is the residue field.

\begin{prop} \label{prop:GaussGal}
Let $\pp$ be a prime ideal above $p$. 
Then $G_{\pp}$ is equal to $\GL_1(k_\pp)$ if $\ell_p$ is odd or $-p^{e_p}\not\in D$; otherwise, it is its cyclic subgroup of order $(p^{e_p}-1)\cdot m$.
\end{prop}

\begin{proof}
We only write the proof when $m \neq 2$, the case $m = 2$ being treated in a similar fashion.

To start with, we observe that $G$ contains $\um_m(K)$, which follows from writing 
\[\zeta_m = (\zeta_m+1)^2 \cdot (\zeta_m + \zeta_m^{-1} + 2)^{-1}\]
and noticing that the second factor lies in $\GL_1(K^+)$ since it is invariant under the transformation $\zeta_m\mapsto \zeta_m^{-1}$.
Therefore, we equally have $G = \GL_1(K^+) \cdot H$ with $H \coloneqq \um_m(K) \cdot \langle 1 + \zeta_m \rangle$.
We note that the former equality also holds when $K = K^+$ since $G = \GL_1(K)$ in this case.

We now compute the intersection $G \cap \GL_1(\Opp)$.
Given that $p$ is coprime with $d$, and hence with $m$, Hensel's lemma ensures that the morphism $\um_m(K) \to \um_m(k_\pp)$ is an isomorphism.
Thus $\zeta_m$, which is a generator of $\um_m(K)$, reduces modulo $\pp$ to a generator of $\um_m(k_\pp)$.
Therefore $1 + \zeta_m$ does not vanish in $k_\pp$.
This shows that the group $H$ sits inside $\GL_1(\Opp)$ and, consequently, that
$G \cap \GL_1(\Opp) = \GL_1(\Opp^+) \cdot H$
where $\Opp^+$ is the localization at $\pp$ of the ring of integers of $K^+$.
Reducing modulo $\pp$, we finally find
$$G_\pp = \GL_1(k_\pp^+) \cdot H_\pp$$
where $H_\pp$ is the reduction of $H$ modulo $\pp$.

When $\ell_p$ is odd or $-p^{e_p}\not\in D$, Lemma~\ref{lem:kplus} shows that $k_\pp = k_\pp^+$ and so $G_\pp = \GL_1(k_\pp)$ as claimed.

We now consider the case where $\ell_p$ is even and $-p^{e_p} \in D$.
Applying Lemma~\ref{lem:kplus} again, we find that the cardinality of $\GL_1(k_\pp^+)$ is $p^{e_p}-1$ and so the cardinality of the quotient $\kpt/\kppt$ is $p^{e_p}+1$.
Moreover, both of these groups are cyclic. 
We then reduced to prove that the image of $H_\pp$ in $\kpt/\kppt$ has cardinality $m$. For this, we observe that the kernel of $\um_m(k_\pp)\to \kpt/\kppt$ is $\um_m(k_\pp^+)$,
which is a cyclic group of order $\gcd(m, p^{e_p}-1)=\gcd(m, 2)$ because $m$ divides $p^{e_p}+1$. 

By what precedes, it is enough to show that the order of $1 + \zeta_m$ in the quotient group $\kpt/\kppt \um_m(k_\pp)$ is $1$ when $m$ is odd and $2$ when $m$ is even.
For this, we first notice that
$$(\zeta_m+1)^2 = \zeta_m \cdot (\zeta_m + \zeta_m^{-1} + 2) \in  \kppt \um_m(k_\pp)$$
since the first factor is in $\um_m(k_\pp)$ and the second factor lies is $k_\pp^+$.
Hence the order of $1 + \zeta_m$ is $1$ or~$2$.
Now, we observe that $1 + \zeta_m$ lies in $\kppt \um_m(k_\pp)$ if and only if there exists an integer $a$ such that $\zeta_m^a (1 + \zeta_m)$ is invariant by the transformation $\zeta_m \mapsto \zeta_m^{-1}$.
After simplification, this further reduces to $\zeta_m^{2a+1} = 1$.
When $m$ is odd, such an element $a$ obviously exists (we can take $a = \frac{m-1} 2$).
On the contrary, when $m$ is even, it does not exist because $\zeta_m$ has order $m$ in $\kppt$ and $m$ cannot divide the odd number $2a + 1$.
\end{proof}

Combining Proposition~\ref{prop:GaussGal} with Corollary~\ref{cor:galoismodp}, we get the first part of Theorem~\ref{thintro:2F1} of the introduction.

\subsubsection{Approaching the Galois group from below}
\label{ssec:boundbelow}

Up until now, we have only bounded the Galois groups of reductions of hypergeometric series from above.
We now develop also bounds from below, that will eventually give a complete proof of Theorem~\ref{thintro:2F1}.

We recall from Proposition~\ref{prop:kummer} that the order of the Galois group of $\Gal(\pFq 2 1 {\ua} 1 x\mid k_\pp(x))$ with $\ua=(\alpha_1, \alpha_2)$ is given by $$\frac {q{-}1} {\gcd(r, q{-}1)},$$ where $q=p^{\ell_p}$ and $r$ is the largest integer such that the polynomials $A(x)$ introduced in Equation~\eqref{eq:defA}, is a perfect $r$-th power.
To estimate $r$, we rely on the fact that $r$ must divide the valuation of $A(x)$ at any point, and also its degree (\emph{i.e.}, its valuation at infinity). 

\paragraph{Valuation at singular points.}

The degree, which we denote by $\nu_\infty$, has been already essentially computed in Subsection~\ref{sssec:GalGauss}.
Precisely, we recall from there that $B_k(x)$ is defined as the truncation at $x^p$ of $\pFq 2 1 {\Dp^k(\ua)} 1 x \bmod p$.
Then, Lemma~\ref{lem:euler} shows that $B_k(x)$ has degree $\min(\alpha_1^{(k)}, \alpha_2^{(k)})$, where $\alpha_i^{(k)}$ denotes the $k$-th digit in the $p$-adic expansion of $-\alpha_i$ for $i \in \{1,2\}$.
Therefore
\begin{equation}
\label{eq:degAx}
\nu_\infty = \sum_{k=0}^{\ell_p-1} \min(\alpha_1^{(k)}, \alpha_2^{(k)}) \cdot p^k.
\end{equation}

The valuation at $1$ of $A(x)$ can be computed similarly.
As in Subsection~\ref{sssec:GalGauss}, we write $B(\ug;x)$ for the truncation of $\pFq 2 1 \ug 1 x \bmod p$ at $x^p$.

\begin{lem}
    Let $\ug = (\gamma_1, \gamma_2) \in (\Zp \cap (0,1))^2$. Then
    \[ \val_1\big(B(\ug;x)\big) = \max\big(0,\, \gamma_1^{(0)} + \gamma_2^{(0)} - p + 1\big) \]
    where $\val_1$ denotes the valuation at $1$.
\end{lem}

\begin{proof}
    We know from Lemma~\ref{lem:euler} that 
    \[B(\ug; x) = (1-x)^{\gamma_1^{(0)} + \gamma_2^{(0)} - p + 1} \cdot B(1-\ug; x).\]
    Without loss of generality, we assume that $\gamma_1^{(0)} + \gamma_2^{(0)}\geq p-1$; otherwise we exchange the roles of $\ug$ and $1-\ug$.
    The local exponents at $1$ of $\Hyp(\ug, (1))\bmod p$ are $0$ and $\gamma_1^{(0)} + \gamma_2^{(0)}-p-1$, while that of $\Hyp(1-\ug, 1)\bmod p$ are $0$ and $-1-(\gamma_1^{(0)} + \gamma_2^{(0)})$; besides $B(\ug;x)$ are $B(1-\ug; x)$ are solutions of the associated differential equations respectively.
    The only option to conceal these conditions is to have
    \begin{align*}
    \val_1\big(B(1-\ug;x)\big) & \equiv -1-\big(\gamma_1^{(0)} + \gamma_2^{(0)}\big) \pmod p\\
    \val_1\big(B(\ug;x)\big) & \equiv \gamma_1^{(0)} + \gamma_2^{(0)} - p + 1 + \val_1\big(B(1-\ug;x)\big)\equiv 0 \pmod p.
    \end{align*}
    Since, moreover, $\val_1(B(\ug;x)) \leq \deg B(\ug; x) < p$, we conclude that $\val_1(B(\ug;x)) = 0$.
\end{proof}

Applying the previous lemma with $\ug = \Dp^k(\ua)$ for $k$ varying between $0$ and $\ell_p{-}1$, we find
\begin{equation}
\label{eq:valAx}
\nu_1 \coloneqq \val_1\big(A(x)\big) = \sum_{k=0}^{\ell_p-1} \max\big(0, \, \alpha_1^{(k)} + \alpha_2^{(k)} - p + 1\big) \cdot p^k.
\end{equation}

Apart from the singular points $1$ and $\infty$, the last singular point of the hypergeometric differential equation is $0$. However, the valuation at $0$ of $A(x)$ is clearly $0$, by definition so this does not provide any further information in our quest to determine $r$. 

\paragraph{Valuation at ordinary points.}

Equation~\eqref{eq:GaussianRel} ensures that we can apply Lemma~\ref{lem:vals} to $B(\ug;x)$ and conclude that each $\xi \in \bar \F_p$ different from $1$ is at most a simple root of each of the polynomials $B_k(x)$.
Then, each root of $A(x)$ different from $1$ has multiplicity 
\begin{equation} \label{eq:RootsMult}
	\epsilon_0+ \epsilon_1\cdot p+\cdots + \epsilon_{\ell_p-1}p^{\ell_p-1}
\end{equation} for some $\epsilon_i\in \{0, 1\}$. 

Assuming now, that $A(x)$ is a perfect $r$-th power, we know that for each of its roots, the expression~\eqref{eq:RootsMult} must be divisible by $r$. 
We use this information is a very special case. For each $k$,
we let $b_k$ be the number of distinct roots of $B_k(x)$ different from $1$, and we order them by size:
\[b_{j_1}<\cdots <b_{j_\ell}.\]
We underline that the values of the $b_k$ can obviously be deduced from the degrees of $B_k(x)$ and the multiplicity of $1$ as a root of each of them, two quantities we have alreadly computed previously.
We take $c$ maximal such that $b_{j_1}+\cdots + b_{j_c}<b_{j_\ell}$.
Then there exists a root $\xi_0$ of $B_{j_\ell}(x)$, that is not a root of any of the polynomials $B_{j_1}(x),\ldots, B_{j_c}(x)$.
Set $J\coloneqq \{j_{k+1},\ldots, j_{\ell-1}\}$, so that the multiplicity of $\xi_0$ as a root of $A(x)$ is given by 
\[\nu_{J'}\coloneqq p^{j_\ell}+\sum_{j\in J'} p^j\] for some subset $J'\subseteq J$.

Putting all the constraints together, we conclude that $r$ has to be a divisor of
\[ g_{J'} \coloneqq \gcd\big(\nu_\infty,\, \nu_1,\, \nu_{J'},\, p^{\ell_p}-1\big) \]
for some $J' \subseteq J$.
Consequently, the index of $\Gal(\pFq 2 1 \ua {(1)} x\mid k_\pp(x))$ in the conjectured group $G_\pp$ of Proposition~\ref{prop:GaussGal} is bounded from above by
$g_{J'}/h$ where
\[ h \coloneqq \begin{cases}
p^{e_p} + 1 & \text{if $\ell_p$ is even and $p^{-e_p} \in D$} \\
1 & \text{otherwise.}
\end{cases} \]
Of course, the set $J'$ is \emph{a priori} not known, but we can even bound the index by taking the maximum (or the least common multiple) of all the previous quantities when $J'$ varies.

\paragraph{Uniformity in $p$.}

To make statements about all primes simultaneously, we need to study the behavior of $g$ with respect to $p$.
The key observation is that $\alpha_1^{(k)}$ and $\alpha_2^{(k)}$ depend only on $s \coloneqq \lfloor p/d\rfloor$ and $t \coloneqq p\bmod d$, as shown by the following lemma.

\begin{lem} \label{lem:polynomial-s}
    Let $\gamma \in \Zp \cap (0,1)$ with denominator $d$, and let $t \in \{1, \ldots, d{-}1\}$ with $\gcd(t,d) = 1$.
    Then there exists a polynomial $\Gamma_t(x) \in \Z[x]$ of degree at most $1$ such that, for all positive integer $s$ such that $p = ds + t$ is a prime number, 
    $\Gamma_t(s)$ is the unique representative of $(-\gamma) \bmod p$ in the interval $[0,p)$.
\end{lem}

\begin{proof}
We write $p = ds + t$.
We consider a Bézout relation $ut + vd = 1$ with $u, v \in \Z$ and notice that $us - v \equiv -\frac 1 d \pmod p$ for all $s$ since
	\[-d(us - v)\equiv ut + vd = 1 \pmod p.\]
If $\gamma = \frac a d$, we now define
\[\Gamma_t(x) \coloneqq aux - av - \left \lfloor \frac {au} d \right \rfloor (d x + t).\]
Clearly, $\Gamma_t(x)$ is a polynomial with integral coefficients of degree at most $1$,
and $\Gamma_t(s) \equiv -\gamma \pmod p$ for all $s$.
Besides the constant coefficient of $\Gamma_t(x)$ is in the range $[0,d)$, while its linear coefficient is in $[0,t)$ given that
$$-1 < -\gamma = -av - \frac{au}d t
 \leq -av - \left\lfloor \frac{au} d \right\rfloor t
 < -av - \left(\frac{au} d - 1 \right) t = t - \gamma.$$
Therefore $0 \leq \Gamma_t(s) < ds + t = p$ for all $s$.
\end{proof}

Applying Lemma~\ref{lem:polynomial-s} with $\gamma = \Dp^k(\alpha_i)$ (with $i \in \{1,2\})$, we find that each $\alpha_i^{(k)}$ is a polynomial of degree at most $1$ in $s$ (for each fixed congruence class $t$ modulo $d$).
Given that $\ell_p$ also only depends on the congruence of $p$ modulo $d$, we find that all the quantities $\nu_\infty$, $\nu_1$, $p^{\ell_p}{-}1$ and all $\nu_{J'}$ are polynomials in $s$ as well.
If the greatest common divisor of these polynomials computed in $\Q[s]$ is equal to $h$ for all $J'$, we derive that the index of $\Gal(\pFq 2 1 \ua {(1)} x\mid k_\pp(x))$ in $G_\pp$ is bounded independently of $p$, and we can even explicitly compute a bound by writing down Bézout relations between these polynomials.
In some cases, this bound is equal to $1$, so that we can conclude that $\Gal(\pFq 2 1 \ua {(1)} x\mid k_\pp(x)) = G_\pp$ for all prime $p$ which is congruent to $t$ modulo $d$.

\begin{ex}
	For any parameter $\ua$, and any prime number $p$ for which $\ell_p=1$, we immediately obtain from Lemma~\ref{lem:vals} that if $A(x)=\pFq 2 1 \ua {(1)} x \bmod p, x^p$ is not of the form $(x-1)^{2d}$ for some $d$, then $A(x)$ is not a perfect power and so $\Gal(\pFq 2 1 \ua {(1)} x \mid k_\pp(x))=\F_p^\times$ for any prime ideal $\pp$ over $p$.
 This is one particular instance on our ideas used for zeroes of $A(x)$ that are nonsingular points of the associated differential equation. For example, taking $\ua=(\frac 1 2, \frac 1 2)$ we have seen the very same argument for hypergeometric function \[\pFq 2 1 \ua {(1)} {4x}=\sum_{n=0}^\infty \binom{2n}{n}^2x^n\] in Subsection~\ref{sssec:ex:binom2}.
\end{ex}
\begin{ex}
Let us take $\ua=(\frac 1 8, \frac 3 8)$. Table~\ref{tab:ex:index} shows for each congruent class of $p$ modulo $d$, the values of $\ell_p$, $\nu_\infty$, $\nu_1$,
the lcm of $\nu_{J'}$ for the relevant $J'$ and the corresponding value for $g$.
We see that the last two columns agree, which proves that the actual Galois groups agree with the conjectured ones.
\begin{table}
    \centering
		\begin{tabular}{r||c|c||c|c|c||c|c}
		 & $\ell$ & $-p^{\frac \ell 2}\in D$? & $\nu_\infty$ & $\nu_1$ & $\text{lcm}(\nu_{J'})$ & $\text{lcm}(g_{J'})$ & $h$\\ \hline \hline
	    $p=8s + 1$ & $1$ & - & $s$ & $0$ & $1$ & $1$ & $1$\\ \hline
		$p=8s + 3$ & $1$ & - & $s$ & $0$ & $1$ & $1$ & $1$\\ \hline
		  $p=8s + 5$ & $2$ & no & $8s^2+10s+3$ & $4s+3$ & $8s+5$ & $1$ & $1$\\ \hline
        $p=8s + 7$ & $2$ & yes & $8s^2+12s+4$ & $ 4s+4$ & $64s^2 + 88s + 30$ & $4s+4$ & $4s+4$
        \end{tabular}
		\caption{Proof that for $\ua=(\frac 1 8, \frac 3 8)$ the Galois group is cyclic of order $\frac{p^{\ell_p}-1}{h}$}
		\label{tab:ex:index}
\end{table}
\end{ex}

\begin{ex}
Let us now take $\ua=(\frac 1 7, \frac 3 7)$. In Table~\ref{tab:ex:index7}, we again investigate for all congruence classes of $p$ modulo $7$ the implications on the size on the Galois groups. We display the ratio $\text{lcm}(g_{J'})/h$ measuring the discrepancy between the expected size of the Galois group and the bound we prove.
The quotient is $1$ if and only if we succeeded to prove that the Galois group $\Gal(\pFq 2 1 \ua {(1)} x \mid k_\pp(x))$ is of index $h$ in $k_\pp^\times$.
We also display $\gcd(\nu_\infty, \nu_1, p^\ell-1)/h$ to emphasize the benefit of the computation of the $\nu_{J'}$ in finding bounds for the size.
\begin{table}[htb]
    \centering
		\begin{tabular}{r||c|c||c|c||c}
		 & $\ell$ & $-p^{\frac \ell 2}\in D$? &  $\gcd(\nu_\infty, \nu_1, p^\ell-1)/h$ & $\text{lcm}(g_{J'})/h$ & $h$\\ \hline \hline
	    $p=7s + 1$ & $1$ & - & $s$ & $1$ & $1$\\ \hline
        $p=7s + 2$ & $3$ & - & $4$ & $\mathbf{2}$ & $1$\\ \hline
        $p=7s + 3$ & $6$ & yes & $1$ & $1$ & $ 49s^3 + 63s^2 + 27s + 4$\\ \hline
        $p=7s + 4$ & $3$ & - & $5$ & $1$ & $ 1$\\ \hline
        $p=7s + 5$ & $6$ & yes & $1$ & $1$ & $ 49s^3 + 105s^2 + 75s + 18$\\ \hline
        $p=7s + 6$ & $2$ & yes & $1$ & $1$ & $ s+1$
        \end{tabular}
		\caption{For $\ua=(\frac 1 7, \frac 3 7)$ and prime numbers congruent to $2$ modulo $7$ the techniques presented do not suffice to prove that $A_p(x)$ is not a square. }
		\label{tab:ex:index7}
	\end{table}   
\end{ex}

\begin{ex}
    From the previous examples one might hope, that using our methods for any set of parameters we are able to find a fixed bound, valid for all prime numbers, on the maximal index of the Galois group in the cyclic group of order $(p^{\ell_p}-1)/h$. However, taking $\ua=(\frac 1 {15}, \frac {11} {15})$ dashes this hope. Writing $p=15s+t$, for each value of $t\in\{2, 7, 8, 13, 14\}$, we find each time a subset $J'$ for which $g_{J'}(\alpha)$ is a nonconstant polynomials in $s$.
\end{ex}

We have implemented the methods described in this chapter in SageMath\footnote{The code is available at \url{https://plmlab.math.cnrs.fr/caruso/gal-Dfin-modp}}, to carry out the necessary computations for the examples and to test the bounds given for small values of the common denominator of the parameters.
Doing this, we showed that the equality
\[\Gal(\pFq 2 1 \ua {(1)} x\mid k_\pp(x)) = G_\pp\]
holds whenever $d \in \{2, 3, 4, 6, 8, 12, 24\}$, which completes the proof of Theorem~\ref{thintro:2F1} of the introduction.
Similarly, we also proved by computations that the index of $\Gal(\pFq 2 1 \ua {(1)} x\mid k_\pp(x))$ in $G_\pp$ is always at most $5$ when $d \leq 12$.

\printbibliography

\textsc{CNRS; IMB, Université de Bordeaux, 351 Cours de la Libération, 33405 Talence,
France}\\
\textit{Email: }\href{mailto:xavier@caruso.ovh}{\texttt{xavier@caruso.ovh}}\\

\textsc{Faculty of Mathematics, University of Vienna, Oskar-Morgenstern-Platz 1, 1090 Vienna, Austria}\\
\textit{Email: }\href{mailto:florian.fuernsinn@univie.ac.at}{\texttt{florian.fuernsinn@univie.ac.at}}\\

\textsc{Institut de Mathématiques de Toulouse, Université Paul Sabatier, 118 route de Narbonne, 31062 Toulouse Cedex 9, France
}\\
\textit{Email: }\href{mailto:daniel.vargas-montoya@math.univ-toulouse.fr}{\texttt{daniel.vargas-montoya@math.univ-toulouse.fr}}\\

\end{document}